\documentclass[eqno,fleqn]{article}
\usepackage[T1]{fontenc}
\usepackage{amsmath}
\usepackage{amssymb}
\usepackage[shortlabels]{enumitem}
\usepackage[onehalfspacing]{setspace}
\usepackage{graphicx}
\usepackage{xcolor}
\usepackage[lowtilde]{url}
\usepackage{hyperref}

\newcommand{\mbm}[1]{\mbox{\boldmath{$#1$}}}
\newcommand{\mbmss}[1]{\mbox{\scriptsize\boldmath{$#1$}}}
\newcommand{\oo}{$\!$\hspace{0.02cm}\'{}$\!\!$\'{}$\!\!\!\!$\hspace{0.02cm}o}
\newcommand{\bbbn}{\mathbb{N}} 
\newcommand{\bbbr}{\mathbb{R}} 
\newcommand{\qed}{\hfill{$\Box$}} 
\newtheorem {overview}{Overview} 
\newtheorem{satz}{Satz}[section] 
\newtheorem{definition}[satz]{Definition}

\newtheorem{proposition}[satz]{Proposition}
\newtheorem{corollary}[satz]{Corollary}
\newtheorem{theorem}[satz]{Theorem}
\newtheorem{remark}[satz]{Remark}

\newtheorem{example}[satz]{Example}
\newtheorem{observ}[satz]{Observation}
\newtheorem{question}[satz]{Question}
\newtheorem{questions}[satz]{Questions}
\newtheorem{agree}[satz]{Agreement}
\newtheorem{propert}[satz]{Properties}
\newtheorem{consequ}[satz]{Consequences}

\begin{document}
\pagestyle{headings} 
\thispagestyle{empty}

\title{Abstract computation over first-order structures. Extras: From Programs to Decision Trees I}
\author{Christine Ga\ss ner}

\begin{center}
{\Large Abstract Computation over First-Order Structures \vspace{0.2cm}}\\{\Large Extras: From Programs to Decision Trees I}
\vspace{0.7cm}\\
{\bf Christine Ga\ss ner}\vspace{0.1cm}\\University of Greifswald, Germany, 2026\\ gassnerc@uni-greifswald.de \end{center} 

\vspace{0.6cm}
\begin{abstract} 
Decisions and their consequences can be described and analyzed by means of decision trees. The decisions themselves depend on questions that, whenever possible, should be answered with {\sf yes} or {\sf no}. The original BSS machines over the real numbers are graphs with computation nodes and branching nodes for decisions. The evaluation of BSS machines and  algebraic decision trees in computer-aided geometry which have been introduced for various types of numbers generally involves the evaluation of systems of literals of first-order logic. Flowchart-like representations and decision trees are also helpful for analyzing decisions made by BSS RAMs over first-order structures. All algorithms defined by machine-oriented programs of BSS RAMs can be illustrated using flowcharts, walks, and paths in program trees. The basic structures of these tools are graphs and their visual representations can help to characterize the behavior of individual BSS RAMs and other first-order machines. We offer a theoretical framework for linking various models. Here, we introduce program paths  and transition systems for transforming partial configurations which are defined syntactically and  can be  extended and refined later. Finally, we define standard orders for program paths and present algorithms for enumerating finite program paths.

\vspace{0.3cm}
\end{abstract}

\newpage
\tableofcontents

\vspace{0.5cm}
{\sf Extras: From Programs to Decision Trees II} \hfill (\textcolor{gray}{abbreviated:} {\sf Extras II})

\noindent
deals with graph-theoretical paths and trees and homomorphisms between such structures. We consider graphs representing program paths where  integers serve as nodes, trees derived from walks for representing finite program paths, and trees for representing all program paths that can be  identified using strings.

\vspace{0.1cm}
{\sf Extras: From Programs to Decision Trees III}  \hfill (\textcolor{gray}{abbreviated:} {\sf Extras III})
\noindent
deals with the evaluation of decision trees for (non-)deterministic machines. We start with notions of first-order logic with identity and discuss the evaluation of program paths and trees, the description of configurations using terms, and the description of trees by formulas.

\vspace{0.1cm}
{\sf Extras: From Programs to Decision Trees IV}  \hfill (\textcolor{gray}{abbreviated:} {\sf Extras IV})

\noindent
provides applications.

\newpage
\section{Introduction and repetition} 
In \cite{GASS25A}, deterministic and non-deterministic BSS RAMs over first-order structures of any signature $\sigma$ (which we will also call {\sf $\sigma$-structures}) were introduced. The characterization of the relationships between the resulting classes of decision problems~-- deterministically semi-decidable or decidable or non-determin\-istically semi-decidable~-- in \cite{GASS25B,GASS25C} also includes a first analysis of so-called {\sf computation paths} and {\sf program paths}. The {\sf program paths of a $\sigma$-program ${\cal P}$} are {\sf label paths} in ${\cal L}_{\cal P}^\infty\cup{\cal L}_{\cal P}^\omega$~-- if ${\cal L}_{\cal P}$ is the set of all labels in ${\cal P}$~--  and  each of these paths determines a sequence of instructions that can be, could be, or might be executed one after the other. Each computation path is a program path. Any {\sf program path} can be illustrated using a {\sf walk} \footnote{ In \cite{GASS25C} (cf.\,\,p.\,\,26), the notion {\sf path} was less formally used for paths and walks. Because we want to discuss more details, we will distinguish now between both notions as usual in graph theory. Then, {\sf walks} are sequences of nodes and edges. The repetition of labels and edges is allowed. Each program path is derived from a walk and the repetition of labels is allowed. For illustrating program paths, we will also use special graphs called {\sf path graphs} or {\sf paths} whose nodes, however, must be different from each other. We will introduce them in {\sf Extras II.}} in a flowchart. 

The evaluation of program and computation paths plays a crucial role in numerous proofs across different areas of computability theory and computer-aided geometry. Since these paths are important for various models of computation, we would like to discuss them in more detail. We will distinguish between {\sf computation paths} determined by sequences of configurations that can be generated by transition systems of BSS RAMs ${\cal M}$ for inputs and the {\sf program paths} which are uniquely determined solely by the programs ${\cal P}_{\cal M}$ themselves. We will investigate computation paths for inputs of a fixed or arbitrary length, define {\sf standard orders} on sets of program paths of finite, fixed, or arbitrary length, and provide algorithms for the {\sf enumeration} of finite program paths, which enable the systematic evaluation of program paths as well as the systematic analysis of the behavior of machines. Every algorithm given by a $\sigma$-program and executable over a $\sigma$-structure can be illustrated using a flowchart and graph-theoretical trees. For each $\sigma$-program ${\cal P}$, such a {\sf flowchart} is a directed graph whose {\sf nodes} are the labels of ${\cal P}$ in ${\cal L}_{\cal P}$ (cf.\,\,Fig.\,\ref{BerWurzeln1}). It can be extended to include, for example, functions that assign the labeled instructions themselves or questions~-- based on the conditions specified by branching instructions~-- to the nodes or that can be used to label edges by the strings `${\rm yes}$' and `${\rm no}$' (cf.\,\,Fig.\,\ref{BerWurzeln2}). Each directed walk in such a flowchart is determined by a sequence of labels  in ${\cal L}_{\cal P}$, and the longest walks starting from label $1$ provides sequences of labels which we call {\sf program paths}. 
In {\sf Extras II}, we will see that these program paths can also be illustrated by new graphs whose nodes belong to the set $\bbbn_+$ of positive integers and which are extended by a function such as $label :\, \subseteq\!\bbbn_+ \to \{1,\ldots,\ell_{\cal P}\}$ to label the nodes. To  recognize whether a program path is a computation path traversed by ${\cal M}$, we need suitable inputs for which ${\cal M}$ traverses this program path, or a {\sf suitable theory} that allows a theoretical proof of the existence of inputs that traverse this path. In the latter case, we could examine certain initial segments of the program paths, describe them syntactically using formulas, and evaluate the descriptions using first-order structures.
Each program path $B^*$ of a $\sigma$-program ${\cal P}$ and the inputs of a fixed length $n$ traversing this path during the execution of the instructions labeled by the members of this path can be described by a system ${\sf sy}_{{\cal P},n}^{\sf decision}(B^*) $ of {\sf conditions} expressed by {\sf atomic $\sigma$-formulas} and {\sf $\sigma$-literals}. These formulas can be evaluated using $\sigma$-structures, which serve to interpret the symbols and assign individuals to the variables. For more details and the description of program paths using {\sf conjunctions of literals}, see {\sf Extras III}. 

\begin{overview}[Program paths and computation paths of machines]\label{Paths}\hfill

\nopagebreak

\noindent \fbox{\parbox{11.8cm}{
\hfill {\small Let ${\cal P}$ be any $\sigma$-program, ${\cal A}$ be a $\sigma$-structure},

\hfill {\small and ${\cal M}\in {\sf M}_{\cal A}\cup {\sf M}_{\cal A}^{\rm NDB} \cup {\sf M}_{\cal A}^{\rm DND} \cup {\sf M}_{\cal A}^{\rm ND}$.}

\vspace{0.1cm}
\em 
\begin{tabular}{c}
\hline\vspace{0.1cm}
\underline{\bf \,A program path $B^*$ of ${\cal P}$\,}\\ 

\vspace{0.1cm}
\begin{tabular}{l}

$\bullet$ is a sequence of labels,\\
$\bullet$ is determined by ${\cal P}$ and derived from ${\cal P}$,\\
$\bullet$ is syntactically defined, \hspace{2.3cm} {\small \em(for details, see Definition \ref{DefProgrPath})}\\
$\bullet$ is a theoretically conceivable computation path,\\ 
$\bullet$ can be described by systems ${\sf sy}_{{\cal P},n}^{\sf decision}(B^*) $ of $\sigma$-literals,\\
$\bullet$ is a computation path if ${\sf sy}_{{\cal P},n}^{\sf decision}(B^*) $ is satisfiable for some $n\geq 1$.\\
\end{tabular} 

\\\hline\vspace{0.1cm}
\underline{\bf \,A computation path $B^*$ of ${\cal M}$\,}\\ 

\vspace{0.1cm}
\begin{tabular}{l}
$\bullet$ is semantically defined,\\
$\bullet$ is really traversed by ${\cal M}$,\\
$\bullet$ is determined by a sequence of configurations of ${\cal M}$,\\ 
$\bullet$ belongs to the set of program paths determined by ${\cal P}_{\cal M}$,\\
$\bullet$ implies that ${\sf sy}_{{{\cal P}_{\cal M}},n}^{\sf decision}(B^*)$ is satisfiable over ${\cal A}$ for some $n\geq 1$.
\end{tabular} \hspace*{0.9cm}\\
\hline\vspace{0.1cm} \underline{\bf \quad Program paths~-- Computation paths \quad}\\ 
Program paths of ${\cal P}_{\cal M}$ are evaluated to recognize whether\\
{\sf a considered program path is a computation path of ${\cal M}$}\\
traversed by ${\cal M}$ for all or some given inputs (and for guesses).\\
\end{tabular}

}}
\end{overview}

Examples for the investigation of various ${\cal A}$-machines and the evaluation of program paths can be found in several proofs in \cite{GASS97, GASS01, GASS08A, GASS08C, GASS09A, GASS09B, GASS10, GASS13, GASS17}. The term {\sf computation path} was often used also in the {\sf sense} of {\sf program paths}. Pascal Koiran, Klaus Meer, Felipe Cucker, Lenore Blum, Michael Shub,\linebreak Steve Smale, and others considered BSS machines over the reals and evaluated their behavior (see, e.g., references \cite{Meer92}, \cite{CSS94}, \cite {BCSS98}). The investigations in \cite{GASS25B, GASS25C} show that certain BSS RAMs over $\sigma$-structures ${\cal A}$ are able to deterministically recognize whether a finite program path of a $\sigma$-program can be traversed by a given BSS RAM over ${\cal A}$ or not whenever a suitable start configuration is given.

Let us repeat the definitions of {\sf program}, {\sf configuration}, {\sf transition system}, and {\sf computation path}. In general, we can consider any signature given by $( N_1; (m_i)_{i\in N_2}; (k_i)_{i\in N_3})$. For programs and for their execution, finite signatures are sufficient. Therefore, in the following, let $\sigma =(n_1; m_1, \ldots,m_{n_2}; k_1,\ldots,k_{n_3})$.

\vspace{0.2cm}
{\bf The set ${\sf P}_{\sigma}$ of $\sigma$-programs.} Any $\sigma$-program ${\cal P}\in {\sf P}_{\sigma}$ has the form

\vspace{0.1cm}
\qquad $1: {\sf instruction}_{1}; \quad \ldots;\quad\ell _{\cal P}-1:\, {\sf instruction}_{\ell _{\cal P}-1};\quad \,\,\,\ell _{\cal P}:\, {\sf stop}.\hfill \textcolor{blue}{(*)}\,\,$

\vspace{0.1cm}
\noindent Let ${\cal L}_{\cal P}=\{1,\ldots, \ell _{\cal P}\}$. For $\ell_{\cal P}> 1$, let ${\cal L}_{\cal P}^\circ=\{1,\ldots, \ell _{\cal P}-1\}$. For each label $\ell \in {\cal L}_{\cal P}^\circ$, ${\sf instruction}_{\ell}$ stands here for an instruction of one of the {\bf types from (1) to (7)}\big[, {\bf or (11)}\big]. ${\sf instruction}_{{\ell}_{\cal P}}$ is the ${\rm S}$-instruction ${\sf stop}$. 

\begin{overview}[$\sigma$-instructions and other instructions]\label{Sigma_Instructions}

\hfill

\nopagebreak 

\noindent \fbox{\parbox{11.8cm}{

\hfill {\small Let $\ell \in {\cal L}_{\cal P}^\circ$ and let $\ell_1$ and $\ell_2$ be labels in ${\cal L}_{\cal P}$.}

\vspace{0.1cm}

{\bf Computation instructions} (with $i\leq n_2$ and $i\leq n_1$, respectively)

$\,\,\,(1)$ \qquad $\ell : \, Z_j:= f_i^{m_i}(Z_{j_1},\ldots, Z_{j_{m_i}}) $ \hfill{\em(${\rm F}$-instructions)}

$\,\,\,(2)$ \qquad $\ell : \,Z_j:=c_i^0$ \hfill{\em(${\rm F}_0$-instructions)}

\vspace{0.1cm}
{\bf Copy instructions} (with $j,k\leq k_{\cal P}$)

$\,\,\,(3)$ \qquad $\ell : \,Z_{I_j}:=Z_{I_k}$ \hfill{\em(${\rm C}$-instructions)}

\vspace{0.1cm}
{\bf Branching instructions} (with $i\leq n_3$) 

$\,\,\,(4)$ \qquad {\sf $\ell : \,$ if $r_i^{k_i}(Z_{j_1},\ldots, Z_{j_{k_i}})$ then goto $\ell _1$ else goto $\ell _2$}\hspace*{0.5cm} \hfill{\em(${\rm T}$-instructions)}

\vspace{0.1cm}
{\bf Index instructions} (with $j,k\leq k_{\cal P}$) 

$\,\,\,(5)$ \qquad {\sf $\ell : \,$ if $I_{j}=I_k$ then goto $\ell _1$ else goto $\ell _2$} \hfill{\em(${\rm H}_{\rm T}$-instructions)}

$\,\,\,(6)$ \qquad $\ell : \,I_j:=1$ \hfill{\em(${\rm H}_1$-instructions)}

$\,\,\,(7)$ \qquad $\ell : \,I_j:=I_j+1$ \hfill{\em(${\rm H}_{+1}$-instructions)}

\vspace{0.1cm}
{\bf Stop instruction}

$\,\,\,(8)$ \qquad$\ell_{\cal P} : $ {\sf stop}\hfill{\em(${\rm S}$-instruction)}

\vspace{0.1cm}
{\bf Instructions for guessing labels} 

$\,\,\,(11)$ \quad\, {\sf $\ell :\,$ goto $\ell_1$ or goto $\ell_2$}\hfill{\em(${\rm B}$-instructions)}

\vspace{0.1cm}}}
\end{overview}

\begin{agree}[The symbols \,$\sigma$, \,${\cal A}$, \,${\cal M}$, \,${\cal P}_{\cal M}$, \,and \,${\cal P}$] \hfill Unless other\-wise specified, we consider, in the following, any finite signature $\sigma$ as defined above and, for such a signature, any $\sigma$-structure ${\cal A}$ or, for an arbitrary first-order structure ${\cal A}$ of an arbitrary signature, a reduct of ${\cal A}$ whose signature is $\sigma$. ${\cal M}$ is any infinite 1-tape ${\cal A}$-machine or a BSS RAM in ${\sf M}_{\cal A}\cup {\sf M}_{\cal A}^{\rm ND}\cup {\sf M}_{\cal A}^{\rm DND}\cup {\sf M}_{\cal A}^{\rm NDB}$. Its program ${\cal P}_{\cal M}$ and ${\cal P}$ are $\sigma$-programs of the form $(*)$. In the case of ${\cal M}\in {\sf M}_{\cal A}^{\rm DND}$, we assume that ${\cal A}\in {\sf Struc}_{c_1,c_s}$ holds.
\end{agree}

{\bf The set ${\cal L}_{\cal P}$.} For any $\sigma$-program ${\cal P}$ \big[with instructions of type (11) for machines in ${\sf M}_{\cal A}^{\rm NDB}$\big], let ${\cal L}_ {\cal P}$ be the union ${\cal L}_{{\cal P},{\rm F}}\cup {\cal L}_{{\cal P},{\rm F}_0}\cup {\cal L}_{{\cal P},{\rm C}}\cup {\cal L}_{{\cal P},{\rm T}} \cup {\cal L}_{{\cal P},{\rm H}_{\rm T}}\cup {\cal L}_{{\cal P},{\rm H}_1}\cup {\cal L}_{{\cal P},{\rm H}_{+1}}\cup {\cal L}_{{\cal P},{\rm S}}\,\big[\!\cup {\cal L}_{{\cal P},{\rm B}}\big]$ of the pairwise disjoint sets ${\cal L}_{{\cal P},{\rm F}}, {\cal L}_{{\cal P},{\rm F}_0}, \ldots $ consisting of the labels of all ${\rm F}$-instructions in ${\cal P}$, the labels of all ${\rm F}_0$-instructions in ${\cal P}$, and so forth, and let ${\cal L}_{{\cal P},{\rm S}}=\{\ell_{\cal P}\}$. In the following, the set ${\cal L}_{\cal P}$ is {\sf always considered to be a subset of the universe $\bbbn_+$ of ${\cal A}_{\bbbn}$ ordered by $\leq$}. 

\vspace{0.2cm}
{\bf The set ${\cal L}_{\cal M}$.} The processing of inputs by an ${\cal A}$-machine ${\cal M}$ is determined by its program ${\cal P}_{\cal M}$. For any ${\cal A}$-machine ${\cal M}$, let ${\cal L}_ {\cal M}$ be the union ${\cal L}_{{\cal M},{\rm F}}\cup {\cal L}_{{\cal M},{\rm F}_0}\cup {\cal L}_{{\cal M},{\rm C}}\cup {\cal L}_{{\cal M},{\rm T}} \cup {\cal L}_{{\cal M},{\rm H}_{\rm T}}\cup {\cal L}_{{\cal M},{\rm H}_1}\cup {\cal L}_{{\cal M},{\rm H}_{+1}}\cup {\cal L}_{{\cal M},{\rm S}}\,\big[\!\cup {\cal L}_{{\cal M},{\rm B}}\big]$ of the pairwise disjoint sets ${\cal L}_{{\cal M},{\rm F}}, {\cal L}_{{\cal M},{\rm F}_0}, \ldots $ consisting of the labels of all ${\rm F}$-instructions in ${\cal P}_{\cal M}$, the labels of all ${\rm F}_0$-instructions in ${\cal P}_{\cal M}$, and so forth. In particular, we have ${\cal L}_{{\cal M},{\rm S}}=\{\ell_{{\cal P}_{\cal M}}\}$. 
Thus,  the execution of instructions by ${\cal M}$  can be uniquely described by means of sequences of ${\cal M}$-configurations.

\vspace{0.3cm}
{\bf ${\cal M}$-configurations in ${\sf S}_{\cal M}$ and partial ${\cal M}$-configurations.} For any infinite 1-tape ${\cal A}$-machine ${\cal M}$, its overall state at a given point of time is determined by an ${\cal M}$-{\sf configuration} $(\ell, \nu_1,\ldots,\nu_{k_{\cal M}}, u_1,u_2,\ldots)$ also denoted by $(\ell\,.\,\vec \nu\,.\,\bar u)$ where $\ell \in {\cal L}_{\cal M}$, $\vec \nu=( \nu_1,\ldots,\nu_{k_{\cal M}})\in \mathbb{N}_+^{k_{\cal M}}$, and $\bar u=( u_1,u_2,\ldots)\in U_{\cal A}^\omega$ hold. ${\sf S}_{\cal M}$ is the set of all ${\cal M}$-configurations. Moreover, we will consider {\em partial ${\cal M}$-configurations $(\ell\,.\,\vec \nu)$}. Each label $\ell$ can be assigned to a register $B$ of ${\cal M}$ and the indices $\nu_1,\ldots,\nu_{k_{\cal M}}$ can be assigned to the index registers of ${\cal M}$. ${\sf S}_{\cal M}^{\rm ind}$ is the set of all {\em partial ${\cal M}$-configurations} given by ${\sf S}_{\cal M}^{\rm ind}=\{(\ell\,.\,\vec \nu) \mid \ell \in {\cal L}_{\cal M} \,\,\,\&\,\,\, \vec \nu \in \bbbn_+^{k_{\cal M}}\}$.

\vspace{0.3cm}
{\bf Transition systems ${\cal S}_{\cal M}$ and ${\cal S}_{{\cal P},\kappa}^{\rm ind}$.} The transition system $({\sf S}_{\cal M},\to_{\cal M})$ belongs to a computational system that we here define once more for ${\cal A}$-machines ${\cal M}$ (for ${\cal M}\in {\sf IM}_{\cal A}$, see also \cite[p.\,584]{GASS20}). For transforming one configuration $con_t\in {\sf S}_{\cal M}$ into the next configuration $con_{t+1}$, we use the transition function $\to _{\cal M }$ (for a summary, see Overview \ref{TransfConf}). The single-valued or multi-valued function $\to _{\cal M }$ enables the change of a finite number of components of $con_t$ by applying functions from ${\cal F}_{\cal M}$ defined in \cite{GASS20,GASS25A} (cf.\,\,\cite[Overviews 20 and 21]{GASS25A}). In Section \ref{SectTransSyst}, we consider the transition systems $({\sf S}_{\cal P}^{\rm ind},\to_{\cal P}^{\rm ind})$ and $({\sf S}_{{\cal P},\kappa}^{\rm ind},\to_{{\cal P},\kappa}^{\rm ind})$ determined by $\sigma$-programs ${\cal P}$ for transforming partial $({\cal P},\kappa)$-configurations $(\ell\,.\,\vec \nu)$ with $\vec \nu \in\bbbn_+^\kappa$. Based on this, we will describe full configurations of ${\cal A}$-machines formally by using first-order $\sigma$-terms independently of any $\sigma$-structure ${\cal A}$ and introduce the transition systems $({\sf S}_{{\cal P},\kappa},\to_{{\cal P},\kappa})$ for transforming {\sf $({\cal P},\kappa)$-term-configurations} in {\sf Extras $\!$III} (cf.\,\,Overviews \ref{OutlookTermConf1} and \ref{OutlookTermConf2}). 

\vspace{0.3cm}
{\bf Computational systems ${\cal R}_ {\cal M}$.} For any ${\cal A}$-machine ${\cal M}$ that should be able to execute instructions of the types (1) to (8), and (11), let the {\em computational system} ${\cal R}_ {\cal M}=( {\cal S}_{\cal M}, {\rm Input}_{\cal M}, {\rm Output}_{\cal M}, {\rm Stop}_{\cal M})$ be given by
\begin{itemize}\parskip -0.2mm
\item the {\em transition system}
${\cal S}_{\cal M}=( {\sf S}_{\cal M}, \to_{\cal M})$ defined by \footnote{ Since any function is by definition a relation, we have $\to_{\cal M}\subset {\sf S}_{\cal M}^2$ for each single-valued function $\to_{\cal M}:{\sf S}_{\cal M}\to {\sf S}_{\cal M}$ and for each multi-valued function $\to_{\cal M}:{\sf S}_{\cal M}\genfrac{}{}{0pt}{2}{\longrightarrow} {\longrightarrow} {\sf S}_{\cal M}$. We generally use the infix notation $con \to_{\cal M} con'$ instead of $(con,con') \in \,\to_{\cal M}$. Moreover, we allow the notation $con \genfrac{}{}{0pt}{2}{\longrightarrow} {\longrightarrow}_{\cal M}\, con'$ instead of $con \to_{\cal M} con'$ if we want to consider a multi-valued function $\to_{\cal M}$. If $\to_{\cal M}$ is a single-valued function, then we can use the prefix notation $\to_{\cal M}\!\!(con)=con'$.}

\vspace{0.2cm}
$\!\!\!\!\!\!\begin{array}{ll}
\, \,\textcolor{red}{\to_{\cal M}} \,=\hspace{0.26cm}\{((\ell\,.\,\vec \nu\,.\, \bar u), (\ell +1\,.\,\vec \nu\,.\, \textcolor{blue}{F_{\ell }(\bar u)}))&\!\!\in {\sf S}_{\cal M}^2\mid \ell \in{\cal L}_{{\cal M}, { \rm F} } 
\cup \,{\cal L}_{{\cal M}, { \rm F} _0} \}\vspace{0.05cm}\\
\hspace{1.3cm}\cup \,\{((\ell\,.\,\vec \nu\,.\, \bar u), (\ell +1\,.\,\vec \nu\,.\, \textcolor{blue}{C_{\ell }(\vec \nu,\bar u)}))\!\!&\!\!\in {\sf S}_{\cal M}^2\mid \ell \in{\cal L}_{{\cal M}, { \rm C} } \}\vspace{0.05cm}\\
\hspace{1.3cm}\cup \,\{((\ell\,.\,\vec \nu\,.\, \bar u), (\textcolor{blue}{T_{\ell }(\bar u)}\,.\,\vec \nu\,.\,\bar u))&\!\!\in {\sf S}_{\cal M}^2\mid \ell\in {\cal L}_{{\cal M},{\rm T}}\}\vspace{0.05cm}\\
\hspace{1.3cm}\cup \,\{((\ell\,.\,\vec \nu\,.\, \bar u), (\textcolor{blue}{T_{\ell }(\vec \nu)}\,.\,\vec \nu\,.\,\bar u))&\!\!\in {\sf S}_{\cal M}^2\mid \ell\in {\cal L}_{{\cal M},{\rm H}_{\rm T}} \}\vspace{0.05cm}\\
\hspace{1.3cm}\cup \,\{((\ell\,.\,\vec \nu\,.\, \bar u), (\textcolor{blue}{\ell'}\,.\,\vec \nu\,.\,\bar u))&\!\!\in {\sf S}_{\cal M}^2\mid \ell\in {\cal L}_{{\cal M},{\rm B}} \,\,\&\,\, \ell'\in G_{\ell }\}\vspace{0.05cm}\\
 \hspace{1.3cm}\cup \,\{((\ell\,.\,\vec \nu\,.\, \bar u), (\ell +1\,.\,\textcolor{blue}{H_{\ell }(\vec \nu)}\,.\, \bar u))\!&\!\!\in {\sf S}_{\cal M}^2\mid \ell\in {\cal L}_{{\cal M},{\rm H}_1}\cup \, {\cal L}_{{\cal M},{\rm H}_{+1}} \}\vspace{0.05cm}\\
 \hspace{1.3cm}\cup \, \{((\ell\,.\,\vec \nu\,.\, \bar u), (\ell\,.\,\vec \nu\,.\, \bar u))&\!\!\in {\sf S}_{\cal M}^2 \mid \ell=\ell_{{\cal P}_{\cal M}}\},\vspace{0.05cm}\\\end{array}$

\vspace{0.2cm}
\item the {\em input procedure} ${\rm Input}_{\cal M}: {\sf I}_{\cal M}\to {\sf S}_{\cal M}$ defined by 

\vspace{0.2cm}
\qquad ${\rm Input}_{\cal M}= \{({\sf i}, (1\,.\,\vec \nu \,.\,\bar u)) \mid {\sf i}\in {\sf I}_{\cal M} \,\,\&\,\,({\sf i},(\vec \nu \,.\,\bar u)) \in { \rm In}_{\cal M} \} $,

\vspace{0.2cm}
\item the {\em output procedure} ${\rm Output}_{\cal M}\!: {\sf S}_{\cal M}\to {\sf O}_{\cal M}$ defined, for 
$(\ell\,.\,\vec \nu\,.\, \bar u)\in\! {\sf S}_{\cal M}$, by 

\vspace{0.2cm}
\qquad ${\rm Output}_{\cal M}((\ell\,.\,\vec \nu\,.\, \bar u))={\rm Out}_{\cal M}((\vec \nu\,.\, \bar u))$,

\vspace{0.2cm}
\item the {\em stop criterion} ${\rm Stop}_{\cal M} (\ell\,.\,\vec \nu\,.\, \bar u ) =0$ satisfied if $\ell=\ell_{{\cal P}_{\cal M}}$.
\end{itemize}
For other types of instructions, the transition function $\to _{\cal M}$ can be defined in a similar way. Instead of ${\rm Output}_{\cal M}((\ell\,.\,\vec \nu\,.\, \bar u))$, we also write ${\rm Output}_{\cal M}(\ell\,.\,\vec \nu\,.\, \bar u)$ and the like. For each BSS RAM ${\cal M}$ over ${\cal A}$, we have ${\sf I}_{\cal M}={\sf O}_{\cal M}=U_{\cal A}^\infty$ and the output function ${\rm Out}_{\cal M}$ is defined by ${\rm Out}_{\cal M}(\vec \nu\,.\, \bar u) = (u_1,\ldots, u_{\nu_1})$. For ${\cal M}\in {\sf M}_{\cal A}$, we have ${\rm In}_{\cal M}(\vec x)=(\vec \nu\,.\, \vec x\,.\,(x_n,x_n, \ldots ))$ with $\vec \nu=(n, 1,\ldots, 1)$ if $\vec x\in U_{\cal A}^n$ holds. For ${\cal M}\in {\sf M}_{\cal A}^{\rm ND}$, we allow $(\vec x, (\vec \nu\,.\, \vec x\,.\,(y_1,\ldots,y_m,x_n,x_n, \ldots ))) \in {\rm In}_{\cal M}$ for $\vec \nu=(n, 1,\ldots, 1)$ and any $(y_1,\ldots,y_m)\in U_{\cal A}^\infty $ if $\vec x\in U_{\cal A}^n$ holds. For ${\cal A}\in {\sf Struc}_{c_1,c_s}$, the restriction $(y_1,\ldots,y_m)\in \{c_1,c_2\}^\infty $ leads to ${\cal M}\in {\sf M}_{\cal A}^{\rm DND}$. 

\vspace{0.3cm}
{\bf $({\cal M},{\sf i})$-paths in ${\sf Con}_{{\cal M},{\sf i}}$ and halting sets $H_{\cal M}$.} For any ${\cal A}$-machine ${\cal M}$ and any input ${\sf i}\in {\sf I}_{\cal M}$, an {\sf $({\cal M},{\sf i})$-path} is a sequence $(con_t)_{t\geq 1} $ that belongs to ${\sf S}_{\cal M}^\omega$ and can be \big[non-\big]deterministically generated by applying ${\rm Input}_{\cal M}$ and the transition system $({\sf S}_{\cal M},\to_{\cal M})$. According to the types of the machines considered here, these $({\cal M},{\sf i})$-paths are determined by $ (trans\,1)$, $(trans\,2)$, or $(trans\,3)$.

\vspace{0.2cm}\noindent \,\,\,
\begin{tabular}{rcllcllr} 
$con_1$&$\!\!\! \!\!\!\!\textcolor{blue}{=} \!\!\!\!$&$\!\!\!{\rm Input}_{\cal M} ({\sf i})$&\,\,and\quad &$con_t$&$\!\!\! \!\!\!\!\textcolor{red}{ \to_{\cal M} } \!\!\!\!$&$\!\!\!con_{t+1}$\, for $t\geq 1$& {\small $(trans\,1)$} \vspace{0.05cm}\\
$ ({\sf i},con_1)$&$\!\!\! \!\!\!\!\textcolor{blue}{\in} \!\!\!\!$&$\!\!\!{\rm Input}_{\cal M}$&\,\,and\quad &$con_t$&$\!\!\! \!\!\!\!\textcolor{red}{\to_{\cal M}} \!\!\!\!$&$\!\!\!con_{t+1}$\, for $t\geq 1$& {\small $(trans\,2)$} \vspace{0.05cm}\\
$con_1$&$\!\!\! \!\!\!\!\textcolor{blue}{=} \!\!\!\!$&$\!\!\!{\rm Input}_{\cal M} ({\sf i})$&\,\,and\quad &$con_t$&$\!\!\! \!\!\!\!\textcolor{red}{\genfrac{}{}{0pt}{2}{\longrightarrow} {\longrightarrow}_ {\cal M}} \!\!\!\!$&$\!\!\!con_{t+1}$\, for $t\geq 1$&\,\,\,  {\small $(trans\,3)$} \vspace{0.15cm}\\
\end{tabular}

\vspace{0.2cm}\noindent Let ${\sf Con}_{{\cal M},{\sf i}}$ be the set of all $({\cal M},{\sf i})$-paths. A path $(con_t)_{t\geq 1} $ in ${\sf Con}_{{\cal M},{\sf i}}$ given by $con_t=(\ell_t\,.\,\vec \nu^{\,(t)}\,.\, \bar u^{(t)})$ for all $t\geq 1$ determines an {\sf infinite} {\em computation path} $(\ell_t)_{t\geq 1} $ if $\ell_t\not=\ell_{{\cal P}_{\cal M}}$ holds for all $t\geq 1$. An initial segment of $(con_t)_{t\geq 1} $ determines a {\sf finite} {\em computation path} $(\ell_t)_{t=1..s} $ of ${\cal M}$ if $\ell_s=\ell_{{\cal P}_{\cal M}}$ for some $s\geq 1$ and $\ell_t\not=\ell_s$ for all $t< s$ hold. If an $({\cal M},{\sf i})$-path determines a finite computation path, then $\sf {i}$ belongs to the halting set $H_{\cal M}$ of ${\cal M}$. For $(con_{\vec x,t})_{t\geq 1} \in {\sf Con}_{{\cal M},\vec x}$ and $(con_{\vec z,t})_{t\geq 1} \in {\sf Con}_{{\cal M},\vec z}$ given by $con_{\vec x,t}=(\ell_{\vec x, t}\,.\,\vec \nu^{\,(\vec x,t)}\,.\, \bar u^{(\vec x,t)})$ and $con_{\vec z,t}=(\ell_{\vec z, t}\,.\,\vec \nu^{\,(\vec z,t)}\,.\, \bar u^{(\vec z,t)})$, both sequences $(\ell_{\vec x,t})_{t\geq 1} $ and $(\ell_{\vec z,t})_{t\geq 1} $ can match. If $(\ell_{\vec x,t})_{t\geq 1} $ is not an infinite computation path and $(\ell_{\vec x,t})_{t\geq 1} $ and $(\ell_{\vec z,t})_{t\geq 1} $ match, then there is some $s\geq 1$ such that $(\ell_{\vec x,t})_{t=1..s} $ and $(\ell_{\vec z,t})_{t=1..s} $ are the same computation path of {\em length $s$}. For deterministic machines ${\cal M}$ and ${\sf i}\in {\sf I}_{\cal M}$, ${\sf Con}_{{\cal M},{\sf i}}$ contains only one sequence $((\ell_t\,.\,\vec \nu^{\,(t)}\,.\, \bar u^{(t)}))_{t\geq 1}$ that can be uniquely generated by ${\rm Input}_{\cal M} ({\sf i})$ and $\to_{\cal M}$ as described by $(trans\,1)$ and in accordance with the sequence $({\sf instruction}_{\ell_t})_{t\geq 1}$ of instructions given by ${\cal P}_{\cal M}$. In the following, we also write $(\ell_t\,.\,\vec \nu^{\,(t)}\,.\, \bar u^{(t)})_{t\geq 1}$ for $((\ell_t\,.\,\vec \nu^{\,(t)}\,.\, \bar u^{(t)}))_{t\geq 1}$, ${\rm Stop}_{\cal M}(con_t)$ for ${\rm Stop}_{\cal M}(\ell_t\,.\,\vec \nu^{(t)}\,.\, \bar u^{(t)})$, and the like. 

\vspace{0.2cm}
{\bf Sequences of configurations deterministically generated.} 
Let the sequence $(con_t)_{t\geq 1}$ belong to $ {\sf Con}_{{\cal M},{\sf i}}$. If $\to_{\cal M}$ is a single-valued mapping, we also use the prefix notation and write $(\to_{\cal M})(con_t)=con_{t+1}$ and $(\to_{\cal M})^1(con_t)=con_{t+1}$ instead of $con_t\to_{\cal M}\!con_{t+1}$. The composition $(\to_{\cal M})^r$ is recursively defined. Let $(\to_{\cal M})^{0}(con)=con$ and $(\to_{\cal M})^{r}(con)=(\to_{\cal M}) ((\to_{\cal M})^{r-1}(con) )$ for all $con\in {\sf S}_{\cal M}$ and $r\geq 1$. The partial function $(\to_{\cal M})_{{\rm Stop}_{\cal M}} :\,\,\subseteq\! {\sf S}_{\cal M} \to {\sf S}_{\cal M}$ can be introduced as follows. Let $(\to_{\cal M})_{{\rm Stop}_{\cal M}}(con) = (\to_{\cal M})^{st_0}(con)$ hold if the set ${St}_{con}$ given by ${St}_{con} = \{st\in \bbbn\mid {\rm Stop}_{\cal M}((\to_{\cal M})^{st}(con))=0\}$ is not empty~-- where the stop criterion ${\rm Stop}_{\cal M}(\ell\,.\,\vec \nu\,.\, \bar u)=0$ is here satisfied only for $\ell=\ell_{{\cal P}_{\cal M}}$~-- and $st_0= \min {St}_{con}$ holds. Otherwise, $(\to_{\cal M})_{{\rm Stop}_{\cal M}}(con)$ is not defined. Consequently, for every finite computation path $ (\ell_t)_{t=1..s} $, each initial part $ (\ell_t\,.\,\vec \nu^{\,(t)}\,.\, \bar u^{(t)})_{t=1..s} $ of an $({\cal M},{\sf i})$-path contains a so-called {\sf stop configuration} $(\ell_s\,.\,\vec \nu^{\,(s)}\,.\, \bar u^{(s)})$ that is given by 

\vspace{0.2cm}
\quad $(\to_{\cal M})_{{\rm Stop}_{\cal M}}(\ell_1\,.\,\vec \nu^{\,(1)}\,.\, \bar u^{(1)}) = (\to_{\cal M})^{s-1}(\ell_1\,.\,\vec \nu^{\,(1)}\,.\, \bar u^{(1)})= (\ell_s\,.\,\vec \nu^{\,(s)}\,.\, \bar u^{(s)})$ 

\vspace{0.2cm}
\noindent with $\ell_s=\ell_{{\cal P}_{\cal M}}$. 

\vspace{0.2cm}
{\bf Traversal of computation paths.} Each $({\cal M},{\sf i})$-path determines a computation path of ${\cal M}$. This path belongs to ${\cal L}_{\cal M}^\infty \cup{\cal L}_{\cal M}^\omega$ and can be denoted by $(\ell_1,\ell_2,\ldots\big[,\ell_s\big])$, $\vec \ell$, $\bar \ell$, or $B',B^*, B_0^*, \ldots, B_1,B_2,\ldots, B^{(1)}, B^{(2)},\ldots, B^{\rm string}(\vec x), $ $\ldots$.\,\,\footnote{ \label{EmptyList} We use ${\cal L}^{\infty}=\bigcup_{s\geq1}{\cal L}^ s$, ${\cal L}^{\omega}=\{(\ell_1,\ell_2,\ldots) \mid (\forall t\geq 1)(\ell_t\in {\cal L}) \}$, $\bar \ell=(\ell_1,\ell_2,\ldots)=(\ell_t)_{t\geq1}$, $\vec \ell=(\ell_1,\ldots \ell_s)=(\ell_t)_{t=1..s}$, etc. $\ell_2,\ldots,\ell_1$ can here stand for the empty list. The sequence $(\ell_1,\ell_2,\ldots\big[,\ell_1\big])$ is either infinite or the tuple $(\ell_1)$, also denoted by $\ell_1$, of length 1. The symbols $ B',B^*,\ldots $ are derived from the word {\sf Berechnungspfad}. $^*$ is here an extension that belongs to the symbol $B^*$. Program paths will also be denoted by $B^*, B_0^*, \ldots$ (cf.\,\,Part IIb). Note, that each label can be stored in a register (an {\sf instruction counter}) denoted by $B$ (for {\sf Befehlsz\"ahler}). A placeholder $\ell$ can also be called an {\em instruction counter}. A label itself is not a counter.  If $^*$ is added to a set $N$, then $^*$ is a unary operation and $N^*$ is the set of all strings over $N$. } We say that ${\cal M}$ {\em goes through $B^*$ for ${\sf i}$} or ${\cal M}$ {\em traverses $B^*$ for ${\sf i}$} or {\em ${\sf i}$ traverses the computation path $B^*$}. All computation paths of ${\cal M}$ are {\sf program paths}.

\vspace{0.2cm}
\noindent \quad For more details, see the introductions in \cite{GASS20, GASS25A} and \cite{GASS25B,GASS25C}. Our metatheory is the Zermelo-Fraenkel set theory with the axiom of choice (ZFC). If we use inequalities such as $t\leq s$, $t\geq 4$, or $t>k$ without further explanation, then let $t,s,k\in\! \bbbn_+$, etc.  Any $\vec x$ has a length $n\geq 1$ which means $\vec x=(x_1,\ldots,x_n)$, etc. 

\section{Flowcharts and program paths}\label{SectionFlowch} 
As shown in Fig.\,\,\ref{BerWurzeln1}, we can use flowcharts for describing the execution of an algorithm given by a $\sigma$-program ${\cal P}$. These flowcharts can be directed graphs whose nodes~-- apart from two end nodes~-- represent the instructions of the program. The directed edges indicate the possible directions of executing the program. Each computation path of a machine ${\cal M}$ executing ${\cal P}$ can be represented by a walk from the input node to the output node. The traversal of the flowchart along this walk and the generation of a sequence of configurations of ${\cal M}$ are closely related. For more details, see Definition \ref{DefFlowLabel}.

\begin{figure}[t] \centering
{\small

\begin{tabular}{l}
{\em The program ${\cal P}$:}\\
\qquad 1: {\sf if $I_1=I_2$ then goto $2$ else goto $1$};\\
\qquad 2: {\sf $Z_2:= c_2^0$};\\
\qquad 3: {\sf $Z_3:=c_1^0$};\\
\qquad 4: {\sf $Z_1:= f_2^2(Z_1, Z_1)$};\\
\qquad 5: {\sf $Z_2:=f_1^2(Z_2,Z_3)$};\\
\qquad 6: {\sf if $r_1^2(Z_1,Z_2)$ then goto $7$ else goto $5$};\\
\qquad 7: {\sf $Z_1:= c_1^0$};\\
\qquad 8: {\sf stop.}\\
{\em Alternatively:}\\
\qquad 1: {\sf if $I_1=I_2$ then goto $2$ else goto $1$};\\
\qquad 2: {\sf $Z_2:= -1$};\\
\qquad 3: {\sf $Z_3:=1$};\\
\qquad 4: {\sf $Z_1:= Z_1 * Z_1$};\\
\qquad 5: {\sf $Z_2:=Z_2+ Z_3$};\\
\qquad 6: {\sf if $Z_1=Z_2$ then goto $7$ else goto $5$};\\
\qquad 7: {\sf $Z_1:= 1$};\\
\qquad 8: {\sf stop.}\\
{\em Symbols:}\\
\qquad $+$ and $*$ stand for $f_1^2$ and $f_2^2$.\\
\qquad $=$ stands for $r_1^2$.\\
\qquad $1$ and $-1$ stand for $c_1^0$ and $c_2^0$.\\
{\em Infix notation:}\\\qquad$Z_2+Z_3$ stands for $f_1^2(Z_2,Z_3)$.\\
\qquad$Z_1*Z_1$\, stands for $f_2^2(Z_1,Z_1)$.\\
\qquad$Z_1=Z_2$ stands for $r_1^2(Z_1,Z_2)$.\\
\end{tabular}} 
\hspace{0.6cm} \begin{tabular}{c} \includegraphics[width=32mm]{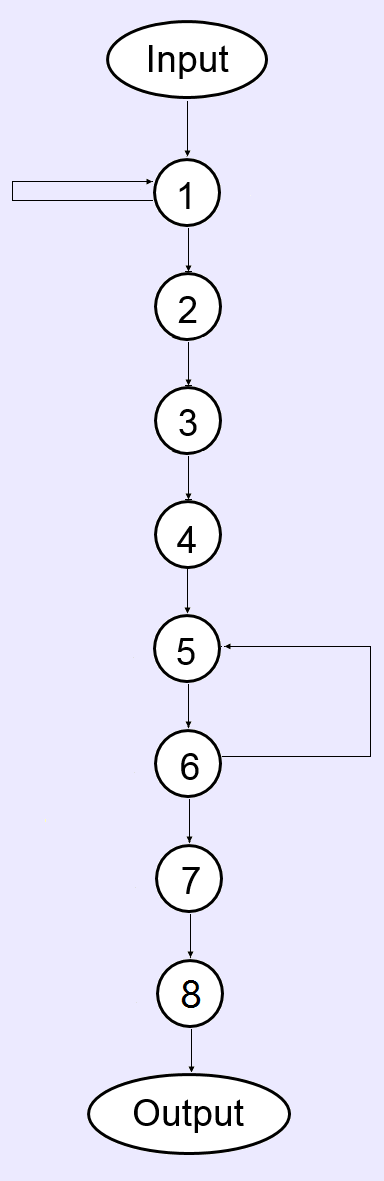} \end{tabular} 
\caption{The flowchart $\mathfrak{Fl\,}_{\cal P}^{\rm label}$ for a program ${\cal P}\in {\sf P}_{(2;2,2;2)}$}\label{BerWurzeln1}
\end{figure}

\subsection{Flowcharts and walks} 
We want to use finite graphs to visualize all possible sequences of processing steps during the execution of a program ${\cal P}\in{\sf P}_{\sigma}$ for an arbitrary input space. The {\em simplest form} of a flowchart could be a graph that belongs to $\mathfrak{Fl\,}_{\cal P}^{\sf label}$. Each of the nodes in $\mathfrak{Fl\,}_{\cal P}^{\sf label}$~-- apart from the input and the output node~-- is the label of an instruction in ${\cal P}$ and each of the edges is a pair of nodes labeling an instruction and a further instruction that can be executed immediately afterward.

In Overview \ref{LabelBetas}, the components and some essential features of flowcharts are summarized. For details, see Remark \ref{RelatStruc}, Definition \ref{DefFlowLabel} where we also use the subsets of ${\cal L}_{\cal P}$ given in Overview \ref{LabelSets}. and Agreement \ref{AgreeFlowNodes}. 

\begin{overview}[Some subsets of ${\cal L}_{\cal P}$] \label{LabelSets}
\hfill

\nopagebreak 
\noindent \fbox{\parbox{11.8cm}{\hfill{\small Let ${\cal P}\in {\sf P}_{\sigma}$.
}

\begin{tabular}{lcl}
${\cal L}_{\cal P}^{\rm NB}\!$&$\!\!\!=$&$\!\!\!\! {\cal L}_{{\cal P},{\rm F}}\cup {\cal L}_{{\cal P},{{\rm F}_0}}\cup{\cal L}_{{\cal P},{\rm C}}\cup {\cal L}_{{\cal P},{{\rm H}_1}}\cup{\cal L}_{{\cal P},{{\rm H}_{+1}}}$,\vspace{0.02cm}\\ 
${\cal L}_{\cal P}^{\rm B}$ &$=$&$\!\!\!\! {\cal L}_{{\cal P},{\rm T}}\cup {\cal L}_{{\cal P},{{\rm H}_{\rm T}}}\cup{\cal L}_{{\cal P},{\rm B}}$, \vspace{0.02cm}\\ 
${\cal L}_{\cal P}$&$\!\!\!=$&$\!\!\!\! {\cal L}_{\cal P}^{\rm NB} \cup {\cal L}_{\cal P}^{\rm B} \cup {\cal L}_{{\cal P},{\rm S}}$,\vspace{0.02cm}\\ 
${\cal L}_{\cal P}^\circ$&$\!\!\!=$&$\!\!\!\! {\cal L}_{\cal P}^{\rm B}\cup {\cal L}_{\cal P}^{\rm NB} ={\cal L}_{\cal P}\setminus {\cal L}_{{\cal P},{\rm S}} $.\\\end{tabular}

\hfill {\small ${\rm B}$ stands for {\sf branching}. 

\hfill ${\rm NB}$ is derived from the word {\sf non-branching}.} 
}}
\end{overview}

\begin{overview}[The structure $\mathfrak{Fl\,}_{\cal P}^{\sf label}$ and its partial $\beta_{\cal P}^\circ$-functions] \label{LabelBetas}
\hfill

\nopagebreak 
\noindent \fbox{\parbox{11.8cm}{\hfill{\small Let ${\cal P}\in{\sf P}_\sigma$.}

\vspace{0.1cm}
\em 
$\mathfrak{Fl\,}_{\cal P}^{\sf label}$ is the structure $((V,E),\beta_{\cal P}, \beta_{\cal P}^+,\beta_{\cal P}^-)$ with 

\vspace{0.15cm}
$\,\,\bullet$\,\, the partial $\beta_{\cal P}^\circ$-functions satisfying

\vspace{0.1cm}
\quad $\begin{array}{rcl}
\beta_{\cal P},\beta_{\cal P}^+,\beta_{\cal P}^- &:&\subseteq {\cal L}_{\cal P}^\circ\to {\cal L}_{\cal P}, \hspace{5.49cm} \mbox{\scriptsize \em \textcolor{gray}{cf.\,\,{\sf (FL3)}}}\,\vspace{0.15cm}\\
\beta_{\cal P}^+&:&{\cal L}_{\cal P}^{\rm B}\to {\cal L}_{\cal P},\\
\beta_{\cal P}^- &:&{\cal L}_{\cal P}^{\rm B}\to {\cal L}_{\cal P},\\
\beta_{\cal P}&:&{\cal L}_{\cal P}^{\rm NB}\to {\cal L}_{\cal P}, \end{array}$

\vspace{0.1cm}
\quad$\begin{array}{lll}
\beta_{\cal P}(\ell)=\ell+1 & \mbox{ for $\ell \in {\cal L}_{\cal P}^{\rm NB}$,} \\
\beta_{\cal P}^+(\ell)= \ell_1& \mbox{ for $\ell \in {\cal L}_{\cal P}^{\rm B}$} \,\textcolor{gray}{\mbox{ \small \em and }} \\
\beta_{\cal P}^- (\ell)=\ell_2  &\mbox{ for $\ell \in {\cal L}_{\cal P}^{\rm B}$}\hspace{5cm}\mbox{\scriptsize \em \textcolor{gray}{cf.\,\,{\sf (FL4)}, {\sf (FL5)}}}\\
&\textcolor{gray}{\mbox{ \small \em  if ${\sf instruction}_\ell$ is of ${\rm B}$-type $({\sf cond}_\ell,\ell_1,\ell_2)$ or $(\ell_1,\ell_2)$}},\\
\end{array}$

\vspace{0.25cm}
$\,\,\bullet$\,\, a graph $(V,E)$ satisfying (a) or (b).

\vspace{0.1cm}
$\,\,\,\begin{array}{rcl}
\mbox{(a)}\,\,\, V\!&=\!&{\cal L}_{\cal P}\cup \{{\sf Input},{\sf Output}\}, \hspace{4.85cm} \mbox{\scriptsize \em \textcolor{gray}{cf.\,\,{\sf (FL1)}}}\,\\
E&=\!&E_{\cal P}\, \cup\{({\sf Input},1), (\ell_{\cal P},{\sf Output})\}, \hfill \mbox{\scriptsize \em \textcolor{gray}{cf.\,\,{\sf (FL2)}}}\,\vspace{0.05cm}\\
\,\,\,E_{\cal P}\!&=\!&\{ (\ell, \ell') \mid \ell \in {\cal L}_{\cal P}^\circ  \,\,\&\,\, \ell' \in Succ_{\cal P}(\ell)\}
\hfill \mbox{\scriptsize \em \textcolor{gray}{cf.\,\,{\sf (FL6)}}}\,\\
Succ_{\cal P}&=&\{(\ell, \{\beta_{\cal P}(\ell)\})\mid \ell \in {\cal L}_{\cal P}^{\rm NB}\} \cup \{(\ell, \{\beta_{\cal P}^+(\ell),\beta_{\cal P}^-(\ell)\})\mid \ell \in {\cal L}_{\cal P}^{\rm B}\}. 
\vspace{0.1cm}\\
\mbox{(b)}\,\,\, V\!&=\!&{\cal L}_{\cal P},\\ E&=\!&E_{\cal P}\subset {\cal L}_{\cal P}^2.\\
\end{array}$

\hfill {\small \em Note, that $\subset$ means both $\subseteq$ and $\not =$. }
}} \end{overview} 

\vspace{0.2cm}
\begin{remark}[Relational structures and partial functions]\label{RelatStruc}\hfill According to \linebreak a definition in \cite[p.\,41]{TuschikWolter}, we can also say that any binary relation $r_0$ on a non-empty set $V$ defines a relational structure $(V,r_0)$ that we call a {\em graph}. More precisely, such a structure is a {\em directed graph} that can be represented by using an arrow for each pair in $r_0\subseteq V\times V$ and vertices for the elements in $V$. As usual in graph theory, it is possible to identify the graphs and their representations largely. $\mathfrak{Fl\,}_{\cal P}^{\sf label}$ could also be given by $(V;;\beta_{\cal P}, \beta_{\cal P}^+,\beta_{\cal P}^-;E)$. Whereas $(V;;;E)$ is a $(0;;2)$-structure that belongs to ${\sf Struc}$, $(V;;\beta_{\cal P}, \beta_{\cal P}^+,\beta_{\cal P}^-;E)$ does not belong to ${\sf Struc}$. The latter structure is not a first-order structure as considered in standard first-order logic since the $\beta_{\cal P}^\circ$-functions are partial. 
\end{remark}

\begin{definition}[Flowcharts $\mathfrak{Fl\,}_{\cal P}^{\sf label}$]\label{DefFlowLabel}
Let ${\cal P}$ be any $\sigma$-program of the form $(*)$. The structure $((V,E),\beta_{\cal P}, \beta_{\cal P}^+,\beta_{\cal P}^-)$ is a {\em flowchart} denoted by $\mathfrak{Fl\,}_{\cal P}^{\sf label}$ if $(V,E)$ is a connected or disconnected directed graph and $ \beta_{\cal P}$, $ \beta_{\cal P}^+$, and $\beta_{\cal P}^-$ are partial functions such that the conditions {\sf (FL1)} to {\sf (FL6)} are satisfied.

\begin{enumerate}[label={\sf (FL\arabic*)}]
\labelwidth0.7cm \leftmargin0cm \itemsep2pt plus1pt \topsep1pt plus1pt minus1pt \labelsep3pt \parsep0.5pt plus0.1pt minus0.1pt

\item{\bf The set of nodes.} 

Let $V$ be the set given by $V={\cal L}_{\cal P}\cup\{{ \sf Input}, { \sf Output} \}$. Its elements are called {\em nodes}. Each node in $V$ can be represented by a dot or illustrated using a small circle or a small ellipse (or another planar figure) in a 2D plane. The {\em input node} ${ \sf Input}$ represents the {\sf source of information} and the {\em output node} ${ \sf Output}$ represents a {\sf data sink}. Each of the other nodes is a label in ${\cal L}_{\cal P}=\{1, \ldots, \ell_{\cal P}\}$ (where $1\leq \ell_{\cal P}$).

\item{\bf The set of edges.} 

Let $E$ be a subset of $V\times V$. The ordered pairs of nodes in $E$ are called {\em edges}. Each of these edges can be represented by an arrow from a node to a node. This corresponds to the fact that each edge $(v_1,v_2)\in E$ is directed and connects the first component $v_1\in V$ with the second component $v_2\in V$. If $(v_1,v_2)\in E$ holds, then $v_2$ is one of the successors of $v_1$. Let $e\in E$. There are three types of edges in $E$.
\begin{itemize} \parskip -0.3mm
\item $e$ can be the edge $({\sf Input},1)$ {\sf from} the input node {\sf to} the node $1$. 
\item$e$ can belong to $ {\cal L}_{\cal P}^\circ \times {\cal L}_{\cal P}$.
\item $e$ can be the edge $(\ell_{\cal P},{\sf Output})$ {\sf from} the node $\ell_{\cal P}$ {\sf to} the output node.
\end{itemize} 
$E$ is a usual set and no multiset, there are not multiple edges in $E$.
On the other hand, {\em loops $(\ell,\ell)$} from a node $\ell\in {\cal L}_{\cal P}^\circ$ to $\ell$ are allowed. 

\item {\bf The $\beta_{\cal P}^\circ$-functions for determining the successors in ${\cal L}_{\cal P}$.} 

For any node in ${\cal L}_{\cal P}^\circ$, all its successors are determined by one of the (partial) functions $\beta_{\cal P}, \beta_{\cal P}^+,\beta_{\cal P}^-$ from ${\cal L}_{\cal P}^\circ$ to ${\cal L}_{\cal P}$ \footnote{ The notations $\beta_{\cal P},\beta_{\cal P}^+$, and $\beta_{\cal P}^-$ go back to \cite[p.\,11]{BSS89} where $\beta,\beta^+,\beta^-$ are used.}  and  $Succ_{\cal P}: {\cal L}_{\cal P}^\circ \to {\mathfrak P}({\cal L}_{\cal P})$.

\item {\bf The successors $\beta_{\cal P}^+ (\ell)$ and $\beta_{\cal P}^-(\ell)$  of branching nodes $\ell$  in $Succ_{\cal P}(\ell)$.} 

Each node $\ell$ that is a label of a branching instruction is a {\em branching node} and belongs to ${\cal L}_{\cal P}^{\rm B}$ defined by ${\cal L}_{\cal P}^{\rm B} = {\cal L}_{{\cal P},{\rm T}}\cup {\cal L}_{{\cal P},{{\rm H}_{\rm T}}}\cup{\cal L}_{{\cal P},{\rm B}}$. It has at most two successors $\beta_{\cal P}^+(\ell)$ and $\beta_{\cal P}^-(\ell)$ where $\beta_{\cal P}^+: {\cal L}_{\cal P}^{\rm B}\to {\cal L}_{\cal P}$ and $\beta_{\cal P}^- : {\cal L}_{\cal P}^{\rm B}\to {\cal L}_{\cal P}$ are defined as follows. Let $\beta_{\cal P}^+(\ell)$ and $\beta_{\cal P}^-(\ell)$ be given by $\beta_{\cal P}^+(\ell)=\ell_1$ and $\beta_{\cal P}^- (\ell)=\ell_2$ and  let $Succ_{\cal P}(\ell)=\{\ell_1,\ell_2\}$  if $\ell$ is the label of a branching instruction of one of the {\sf ${\rm B}$-types} 
 $({\sf cond}_\ell,\ell_1,\ell_2)$ or $(\ell_1,\ell_2)$~-- given by 

\vspace{0.1cm}
\qquad {\sf if ${\sf cond}_\ell$ then goto $\ell_1$ else goto $\ell_2$} \hfill $({\sf cond}_\ell,\ell_1,\ell_2)$,

\vspace{0.05cm}
\qquad {\sf goto $\ell_1$ or goto $\ell_2$} \hfill $(\ell_1,\ell_2)$

\vspace{0.1cm}\noindent~-- that belongs to ${\cal P}$. \footnote{ Here, we could also consider ${\rm O}$-instructions (and thus $\ell\in {\cal L}_{{\cal P},{\rm O}})$.}

We can assume that $\ell_1\not=\ell_2$. \footnote{ \label{BranchGeneral} Note: It is possible to consider only branching instructions of the ${\rm B}$-types $({\sf cond}_ \ell,\ell_1,\ell_2)$ and $(\ell_1,\ell_2)$ where \textcolor{blue}{$\ell_1\not =\ell_2$} holds because all instructions of the form {\sf $\ell : \,$ goto $\ell _1$ or goto $\ell _1$} \, or \, {\sf $\ell : \,$ if ${\sf cond}_ \ell$ then goto $\ell _1$ else goto $\ell _1$} in a program ${\cal P}$ can be replaced by a subprogram {\sf $\ell : $ if $I_1\!=\!I_1$ then goto $\ell _1$ else goto $\ell';$ $ \ell': \,$ goto $\ell _1$ or goto $\ell'$} with a new label $\ell'\not \in {\cal L}_{\cal P}$.}. Otherwise, $\ell$ should have only one successor and one outgoing edge.

\item {\bf The successors $\beta_{\cal P} (\ell)$ of non-branching nodes $\ell\in {\cal L}_{\cal P}^{\rm NB}$ in $Succ_{\cal P}(\ell)$.} 

Each node $\ell \in {\cal L}_{\cal P}^\circ$ that is not the label of a branching instruction belongs to ${\cal L}_{\cal P}^{\rm NB}$ defined by ${\cal L}_{\cal P}^{\rm NB}= {\cal L}_{{\cal P},{\rm F}}\cup {\cal L}_{{\cal P},{{\rm F}_0}}\cup{\cal L}_{{\cal P},{\rm C}}\cup {\cal L}_{{\cal P},{{\rm H}_1}}\cup{\cal L}_{{\cal P},{{\rm H}_{+1}}}$ and has only one successor $\beta_{\cal P}(\ell)$.  For $\ell \in {\cal L}_{\cal P}^{\rm NB}$, let $Succ_{\cal P}(\ell)=\{\ell+1\}$. 
$\beta_{\cal P}: {\cal L}_{\cal P}^{\rm NB}\to {\cal L}_{\cal P}$ is uniquely defined by $\beta_{\cal P}(\ell)=\ell+1$ for all $\ell \in {\cal L}_{\cal P}^{\rm NB}$. \footnote{ Here, we omit also the ${\rm N}$-instructions (and thus $\ell\in {\cal L}_{{\cal P},{\rm N}}$). We think that either $\beta_{\cal P}(\ell)=\ell+1$ or $\beta_{\cal P}^+(\ell)=\ell+1$ and $\beta_{\cal P}^-(\ell)=\ell$ would be possible (see \cite{GASS25A} and \cite{GASS25B}).}

\item {\bf The edges in $E_{\cal P}$.}

Apart from the edges $({\sf Input},1)$ and $ (\ell_{\cal P},{\sf Output})$, all other edges in $E$ are determined by $\beta_{\cal P}$, $\beta_{\cal P}^+$, and $\beta_{\cal P}^-$ such that we have

\vspace{0.2cm}
$E=E_{\cal P} \cup\{({\sf Input},1), (\ell_{\cal P},{\sf Output})\}$ 

\vspace{0.1cm}
and 

\vspace{0.1cm}$E_{\cal P}=\{ (\ell, \ell') \mid \ell \in {\cal L}_{\cal P}^\circ  \,\,\&\,\, \ell' \in Succ_{\cal P}(\ell)\}$. 

\vspace{0.2cm}
If, for a node $\ell$ in ${\cal L}_{\cal P}^\circ$, there is a label $\ell '\in {\cal L}_{\cal P}$ with $(\ell,\ell')\in E_{\cal P}$, then $(\ell,\ell')$ is an {\em outgoing edge} of $\ell$. If, for a node $\ell$ in ${\cal L}_{\cal P}$, there is a label $\ell '\in {\cal L}_{\cal P}^\circ$ with $(\ell',\ell)\in E_{\cal P}$, then $(\ell',\ell)$ is an {\em incoming edge} of $\ell$. $({\sf Input},1)$ is an {\em incoming edge} of $1$. $(\ell_{\cal P},{\sf Output})$ is an {\em outgoing edge} of $\ell_{\cal P}$.
\end{enumerate}
\end{definition} 
More precisely, we have $E\subset (V\setminus \{{\sf Output}\})\times (V\setminus \{{\sf Input}\})$ which means that the node ${\sf Input}$ has no incoming edges whereas the node $1$ can have several incoming edges. We can say with certainty that $E_{\cal P}\subset {\cal L}_{\cal P}^\circ\times {\cal L}_{\cal P}$ holds for $|{\cal L}_{\cal P}| =\ell_{\cal P}\geq 3$ and that $\ell_{\cal P}$ has at most one outgoing edge.
\begin{agree}[{\sf Input} and {\sf Output} are optional]\label{AgreeFlowNodes}
In addition to the labels, we allow also an input node and an output node, but $\,\mathfrak{Fl\,}_{\cal P}^{\sf label}$ can, in any case, be described by

\vspace{0.3cm}
$((\underbrace{{\cal L}_{\cal P}\big [\cup \{{\sf Input},{\sf Output}\} \big]}_{V},\underbrace{E_{\cal P} \big [\cup\{({\sf Input},1), (\ell_{\cal P},{\sf Output})\} \big]}_{E}),\beta_{\cal P}, \beta_{\cal P}^+,\beta_{\cal P}^-)$.

\vspace{0.3cm}
\noindent We can delete either all characters in the square brackets and the square brackets or only all square brackets. This means that the structure $((V,E),\beta_{\cal P}, \beta_{\cal P}^+,\beta_{\cal P}^-)$ can be $(({\cal L}_{\cal P}\cup \{{\sf Input},{\sf Output}\}, E_{\cal P}\cup\{({\sf Input},1), (\ell_{\cal P},{\sf Output})\}),\beta_{\cal P}, \beta_{\cal P}^+,\beta_{\cal P}^-)$~-- used particularly when we  draw it~-- or otherwise $(({\cal L}_{\cal P},E_{\cal P}),\beta_{\cal P}, \beta_{\cal P}^+,\beta_{\cal P}^-)$.
\end{agree}
{\bf Further kinds of flowcharts and their representation.} For any ${\cal P}\in {\sf P}_{\sigma}$ given by $(*)$, let the {\em flowchart $\mathfrak{Fl\,}_{\cal P}^{\sf instr}$} be the structure that results from $\mathfrak{Fl\,}_{\cal P}^{\sf label}$ by adding the function $instr$ that assigns ${\sf instruction}_\ell$ to $\ell$ for any node $\ell\in {\cal L}_{\cal P}$. The introduction of further functions allows us to represent the $\sigma$-programs by flowcharts in the usual way in a 2D plane.\,\,\footnote{ For more information, see https://en.wikipedia.org/wiki/Flowchart. We checked the information on the Wikipedia website on March 20, 2026.} We were already using such and similar representations, for example, on the last pages of \cite{GASS25C}. More precisely, let $\mathfrak{Fl\,}_{\cal P}^{\sf insque}$ be an expansion of $\mathfrak{Fl\,}_{\cal P}^{\sf label}$ that results from $\mathfrak{Fl\,}_{\cal P}^{\sf label}$ by adding new functions $insque$ and $denote$. The function $denote$ helps to mark certain edges and $insque$ can be used to replace some values of $instr$ by new values of $insque$ such that the new flowcharts have the following form (cf.\,\,Fig.\,\,\ref{BerWurzeln2}).
\begin{itemize}
\item Each node $\ell$ in ${\cal L}_{\cal P}^{\rm B}$ corresponding to an instruction of ${\rm B}$-type $({\sf cond}_\ell,\ell_1,\ell_2)$ can be represented by the string {\sf cond$_\ell$?} given by 

\vspace{0.1cm}
\qquad $insque(\ell)=$ \textcolor{blue}{{\sf cond}$_\ell${\sf ?}}.
 
\vspace{0.1cm}
\item The label $\ell$ of an instruction of ${\rm B}$-type $(\ell_1,\ell_2)$ can be represented by the string {\sf random} given by 

\vspace{0.1cm}
\qquad $ insque(\ell)=\textcolor{blue}{{\sf random}}$. 

\vspace{0.1cm}
\item In the other cases, let us use \vspace{0.1cm}

\vspace{0.1cm}
\qquad $insque(\ell)=\textcolor{blue}{instr(\ell)}$.

\vspace{0.1cm}
\item If $\ell$ in ${\cal L}_{\cal P}^{\rm B}$ is the label of an instruction ${\sf instruction}_{\ell}$ of ${\rm B}$-type $({\sf cond}_\ell,\ell_1,\ell_2)$ and there holds
$\ell_1\not= \ell_2$, then let 

\vspace{0.1cm}
\qquad $ denote((\ell,\ell_1))=\textcolor{blue}{{\rm yes}} \quad \mbox{ and } \quad denote((\ell,\ell_2))=\textcolor{blue}{{\rm no}}$. 

\vspace{0.1cm}
\noindent In such a case, the arrows representing the outgoing edges $(\ell,\ell_1)$ and $(\ell,\ell_2)$ from $\ell$ to its successors $\ell_1$ and $\ell_2$ can be marked with the strings `${\rm yes}$' and `${\rm no}$', respectively.
\end{itemize}
For more details, see Overview \ref{SeveralFlowch}.

\vspace{0.15cm}
\begin{agree}[Flowcharts without ${\sf Input}$ and ${\sf Output}$] \hfill In the following \linebreak definitions and propositions, let, for any ${\cal P}\in {\sf P}_{\sigma}$, $\mathfrak{Fl\,}_{\cal P}^{\sf label}=(({\cal L}_{\cal P},E),\beta_{\cal P}, \beta_{\cal P}^+,\beta_{\cal P}^-)$ with $E=E_{\cal P}\subseteq {\cal L}_{\cal P}^2$, i.e., we omit the nodes {\sf Input} and {\sf Output}. \end{agree}
\begin{overview} [Flowcharts: Graphical representations of programs]\label{SeveralFlowch}

\hfill

\nopagebreak 

\noindent \fbox{\parbox{11.8cm}{

{\small 
\hfill Let ${\cal P}\in {\sf P}_{\sigma}$.}

\em
$\,\,\left.\begin{array}{ll}
\mathfrak{Fl\,}_{\cal P}^{\sf label}&=(({\cal L}_{\cal P},E),\beta_{\cal P}, \beta_{\cal P}^+,\beta_{\cal P}^-)\vspace{0.1cm}\\
\mathfrak{Fl\,}_{\cal P}^{\sf instr}&=(({\cal L}_{\cal P},E),\beta_{\cal P}, \beta_{\cal P}^+,\beta_{\cal P}^-,instr)\vspace{0.1cm}\\
\mathfrak{Fl\,}_{\cal P}^{\sf insque}&=(({\cal L}_{\cal P},E),\beta_{\cal P}, \beta_{\cal P}^+,\beta_{\cal P}^-, insque,denote)
\end{array}\right\}{\scriptsize \begin{array}{ll}\\\mbox{\normalsize Flowcharts}\\\mbox{\it \,\,\qquad for representing ${\cal P}$}\end{array}}$

\vspace{0.15cm}
with 

\vspace{0.05cm}
\hspace{0.3cm}\begin{tabular}{lll}
$\beta_{\cal P}, \beta_{\cal P}^+,\beta_{\cal P}^-\!\!\!\!$ &$:\,\subseteq\! {\cal L}_{\cal P}^\circ\to {\cal L}_{\cal P}$,\\
$instr$&$\!\!\!:\hspace{0.28cm} {\cal L}_{\cal P} \to \{{\sf instruction}_\ell\mid \ell \in{\cal L}_{\cal P}\} $,\\
$insque$&$\!\!\!: \hspace{0.28cm}{\cal L}_{\cal P}\to \{{\sf instruction}_\ell\mid \ell \in{\cal L}_{\cal P}\setminus {\cal L}_{\cal P}^{\rm B}\} \cup \{{\sf random}\mid {\cal L}_{{\cal P},{\rm B}}\not =\emptyset\}$\\&\hspace{1.6cm}$\cup \hspace{0.1cm}\{{\sf cond}_\ell${\sf ?}\,$\mid \ell \in{\cal L}_{{\cal P},{\rm T}}\cup {\cal L}_{{\cal P},{\rm H}_{\rm T}} \} $,\\
$denote$&$\!\!\!:\hspace{0.19cm}\subseteq E\to \{{\rm yes},{\rm no} \}$\\
\end{tabular}

\vspace{0.1cm}
where

\vspace{0.05cm}
\hspace{0.3cm}\begin{tabular}{ll}
$instr$&$\!\!\!= \{(\ell,{\sf instruction}_\ell)\mid \ell \in{\cal L}_{\cal P}\}, $\vspace{0.1cm}\\
$insque$&$\!\!\!= \{(\ell,{\sf instruction}_\ell)\mid \ell \in{\cal L}_{\cal P}\setminus {\cal L}_{\cal P}^{\rm B}\} \cup \{(\ell, {\sf random})\mid \ell \in {\cal L}_{{\cal P},{\rm B}}\}$\\
&\hspace{0.4cm}$\cup \hspace{0.1cm}\{(\ell, {\sf cond}_\ell${\sf ?}$) \mid \ell \in{\cal L}_{{\cal P},{\rm T}}\cup {\cal L}_{{\cal P},{\rm H}_{\rm T}} \}$.
\end{tabular}
\vspace{0.05cm}}}
\end{overview}

\vspace{0.3cm}
As demonstrated in Fig.\,\,\ref{BerWurzeln2}, the nodes of a flowchart can be represented by using the function $insque$. Each node $\ell$ in ${\cal L}_{\cal P}$ can be represented by the value $insque(\ell)$ or by using other suitable functions. Generally, the representations of flowcharts are not planar (so that they cannot be drawn in a plane $\subseteq \bbbr^2$ without crossing edges). 

\begin{figure}[t] \centering
{\small
\begin{tabular}{l}
{\em The program ${\cal P}$ in Fig. \ref{BerWurzeln1} can be used to}\\
{\em compute $\bar \chi_P$ by a BSS RAM ${\cal M}$ over ${\cal A}$}\\{\em with ${\cal P}_{\cal M}={\cal P}$ if}\\
\quad $P=\{x_1\in \bbbr\mid (\exists y\in \bbbn) (x_1^2=y)\}$ {\em and}\\
 \quad ${\cal A}=(\bbbr;1,-1;+,\cdot\,;=)$.\vspace{0.1cm}\\
{\em This means that $\bar \chi_P:\,\subseteq\! \bbbr^\infty \to \{1\}$ and }\\
\quad $\bar \chi_P(\vec x)=1$ holds if and only if $\vec x\in \bbbr$ \\ 
\quad and $x_1$ is a square root $\pm\sqrt{n}$ of $n\in \bbbn$.\vspace{0.15cm}\\
{\em The input, ${\cal P}$, and the output:}\vspace{0.1cm}\\
\,\,\,{\sf Input $(x_1,\ldots, x_n)\in \bbbr^\infty$.}\\
\quad 1: {\sf if $I_1=I_2$ then goto $2$ else goto $1$};\\
\quad 2: {\sf $Z_2:= -1$};\\
\quad 3: {\sf $Z_3:=1$};\\
\quad 4: {\sf $Z_1:= Z_1 * Z_1$};\\
\quad 5: {\sf $Z_2:=Z_2+ Z_3$};\\
\quad 6: {\sf if $Z_1=Z_2$ then goto $7$ else goto $5$};\\
\quad 7: {\sf $Z_1:= 1$};\\
\quad 8: {\sf stop.}\\
\,\,\, {\sf Output $c(Z_{c(I_1))}$.}\vspace{0.15cm}\\
{\em The symbols:}\\\quad 
$+$ and $*$ stand for $f_1^2$ and $f_2^2$,\\
\quad $=$ for $r_1^2$,\\
 \quad $1$ and $-1$ for $c_1^0$ and $c_2^0$.\\
\end{tabular}}
\begin{tabular}{c}\includegraphics[width=53mm]{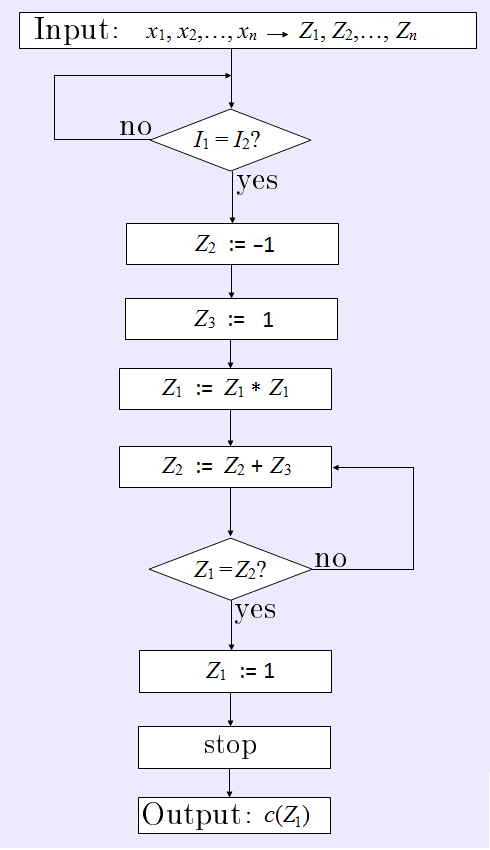}\end{tabular} 
\caption{The flowchart $\mathfrak{Fl\,}_{\cal P}^{\sf insque}$ for ${\cal P}\in {\sf P}_{(2;2,2;2)}$ given in Fig. \ref{BerWurzeln1}}\label{BerWurzeln2}
\end{figure}
\begin{definition}[Walks and simple paths]Let \,$(V,E)$ be a directed graph which means that $V$ is a non-empty set and $E\subseteq V\times V$ holds. A {\em walk} in a graph $(V,E)$ is a finite or an infinite sequence 

\vspace{0.3cm}
 \quad $(v_1)$ \qquad or \qquad $(v_1,e_1, v_2,e_2, \ldots\big[, e_{r-1},v_r\big])$ \hfill {\small (with $r=1$ or $\,r>1$)}

\vspace{0.3cm}
\noindent of nodes $v_i\in V$ and edges $e_i\in E$ such that any $e_i$ is an outgoing edge of $v_i$ and an incoming edge of $v_{i+1}$ for all $i$ with $1\leq i\big[\!< r\big]$. If the walk is finite, then it is a {\em walk from $v_1$ to $v_r$} and its {\em length} is $r-1$.\footnote{ For a definition of finite walks, see also \cite[p.\,9]{Diestel06} and \cite[p.\,98]{Guichard}. Each walk $(v_1,e_1, v_2,$ $ \ldots, e_{r-1},v_r)$ of length $r-1$ contains the nodes of a sequence $(v_1,\ldots,v_r)$ of length $r$.} A \big[finite\big] walk is a {\em simple path} if $1\leq i<j\big[\leq r\big]$ implies $v_i\not =v_j$. Each walk in $(V,E)$ is also a {\em walk in $\mathfrak{Fl\,}_{\cal P}^{\sf label}$} if $\mathfrak{Fl\,}_{\cal P}^{\sf label}=((V,E),\beta_{\cal P}, \beta_{\cal P}^+,\beta_{\cal P}^-)$ holds. The same holds for other flowcharts.
\end{definition}

\vspace{0.3cm}

\begin{proposition}[Walks in flowcharts]\label{PropertWalks}Let ${\cal P}$ be any $\sigma$-program.
\parskip -0.5mm
\begin{enumerate}[label=(\arabic*)] \parskip -0.5mm
\item All simple paths in $\mathfrak{Fl\,}_{\cal P}^{\sf label}$ are finite.
\item There can be an infinite walk in $\mathfrak{Fl\,}_{\cal P}^{\sf label}$ only if there is some $\ell\in {\cal L}_{\cal P}^\circ$ such that $(\ell,1)$ is an incoming edge of node $1$ in $\mathfrak{Fl\,}_{\cal P}^{\sf label}$ or there is a node in ${\cal L}_{\cal P}^\circ$ that has two incoming edges in ${\cal L}_{\cal P}^\circ\times {\cal L}_{\cal P}^\circ$. 
\item Each finite walk $(v_1,e_1,v_2,\ldots, e_{r-1},v_r)$ in $\mathfrak{Fl\,}_{\cal P}^{\sf label}$ that is not simple contains a cycle $ (v_i,\ldots,v_j)$ with $ v_i=v_j$ such that there is an infinite walk and an infinite number of finite walks in $\mathfrak{Fl\,}_{\cal P}^{\sf label}$.
\item All loops $(\ell,\ell)$ in $\mathfrak{Fl\,}_{\cal P}^{\sf label}$ are edges from a node $\ell$ in ${\cal L}_{\cal P}^{\rm B}$ to itself.
\item Any walk $(v_1,e_1, v_2,e_2, \ldots\big[, e_{r-1},v_r\big])$ in $\mathfrak{Fl\,}_{\cal P}^{\sf label}$ is a {\em walk} in $\mathfrak{Fl\,}_{\cal P}^{\sf label}$, $\mathfrak{Fl\,}_{\cal P}^{\sf instr}$, and $\mathfrak{Fl\,}_{\cal P}^{\sf insque}$. It uniquely determines a sequence of nodes, $(v_i)_{i\geq 1}\in({\cal L}_{\cal P}^\circ)^\omega$ and $(v_1, \ldots, v_r)\in ({\cal L}_{\cal P}^\circ)^ {r-1}\times ( {\cal L}_{\cal P}^\circ \cup {\cal L}_{\cal P})$ for some $r\geq 1$, respectively.
\end{enumerate} \end{proposition} 

{\bf Proof.} Let $((V,E),\beta_{\cal P}, \beta_{\cal P}^+,\beta_{\cal P}^-)$ be the flowchart $\mathfrak{Fl\,}_{\cal P}^{\sf label}$, let us assume that a walk $(v_1,e_1, v_2,e_2, \ldots\big[, e_{r-1},v_r\big])$ in $\mathfrak{Fl\,}_{\cal P}^{\sf label}$ is not simple, and let $i$ be the smallest integer with $v_i=v_j$ for some $j>i$. If $v_i$ is the node 1, then we have $(\ell,1)\in E$ for $\ell=v_{j-1}$. If $v_i$ is not the node 1, then $(v_{i-1}, v_i)$ and $(v_{j-1}, v_i)$ are two incoming edges of $v_i$ since $v_{i-1}=v_{j-1}$ contradicts our assumption. Moreover, $v_i\not=\ell_{\cal P} $, $v_{i-1}\not=\ell_{\cal P} $, and $v_{j-1}\not=\ell_{\cal P} $ hold since $\ell_{\cal P}$ has no successor in ${\cal L}_{\cal P}$. Consequently, {\em (2)} is valid. If $(v_1,e_1, v_2,e_2, \ldots,e_{i-1},\textcolor{blue}{v_i},e_{i}, \ldots,e_{j-1},\textcolor{blue}{v_i})$ is an initial part of a walk, then any extension by $e_{i}, \ldots,e_{j-1},v_i$ leads to a longer walk. Thus, {\em (3)} is also valid. 
\qed 

\vspace{0.2cm}
For any ${\cal A}$-machine ${\cal M}$, $(\ell\,.\, \vec \nu\,.\,\bar u)\to _{\cal M}(\ell'\,.\, \vec \mu\,.\,\bar z)$ implies that $(\ell,\ell')$ is an edge in $\mathfrak{Fl\,}_{{\cal P}_{\cal M}}^{\sf label}$ by Definition \ref{DefFlowLabel}. Thus, each finite sequence of changes of ${\cal M}$-configurations starting with an initial configuration~-- by repeatedly applying the transition function $\to _{\cal M}$ of the system $({\sf S}_{\cal M},\to _{\cal M})$~-- until the first stop configuration is reached determines a computation path  of ${\cal M}$ that corresponds to a sequence $(\ell_1,\ldots, \ell_s)$ of nodes in $\mathfrak{Fl\,}_{{\cal P}_{\cal M}}^{\sf label}$ and to exactly one walk~-- if we include the input node and the output node~-- from the input node to the output node in $\mathfrak{Fl\,}_{{\cal P}_{\cal M}}^{\sf label}$. We can furthermore say that any arbitrary sequence of processing steps performed by a usual ${\cal A}$-machine ${\cal M}$ can be visualized and modeled by a walk in $\mathfrak{Fl\,}_{{\cal P}_{\cal M}}^{\sf label}$. 
\begin{proposition}[Computation paths leave a trace in $\mathfrak{Fl\,}_{{\cal P}_{\cal M}}^{\sf label}$]\label{ComputPathInFlow}\hfill Let ${\cal M}$ be \linebreak any ${\cal A}$-machine. For any finite computation path $(\ell_t)_{t=1.. s}$ of ${\cal M}$, there is a finite walk  $(\ell_1,(\ell_1,\ell_2),\ldots,\ell_{s-1},(\ell_{s-1},\ell_s), \ell_s)$ in $\mathfrak{Fl\,}_{{\cal P}_{\cal M}}^{\sf label}$. For any infinite computation path $(\ell_t)_{t\geq 1}$ of ${\cal M}$, there is an infinite walk $(\ell_1,(\ell_1,\ell_2),\ell_2,(\ell_2,\ell_3), \ldots )$ in $\mathfrak{Fl\,}_{{\cal P}_{\cal M}}^{\sf label}$.
\end{proposition}

For understanding the behavior of ${\cal M}$, we will investigate the walks in $\mathfrak{Fl\,}_{{\cal P}_{\cal M}}^{\sf label}$. If an ${\cal A}$-machine ${\cal M}$ traverses the sequence of nodes determined by a walk from 1 to $\ell_{{\cal P}_{\cal M}}$ as described in Proposition \ref{ComputPathInFlow} or the sequence of nodes determined by an infinite walk starting from label 1 in $\mathfrak{Fl\,}_{{\cal P}_{\cal M}}^{\sf label}$, then we can also say that ${\cal M}$ {\em traverses this walk} in $\mathfrak{Fl\,}_{{\cal P}_{\cal M}}^{\sf label}$, $\mathfrak{Fl\,}_{{\cal P}_{\cal M}}^{\sf instr}$, and $\mathfrak{Fl\,}_{{\cal P}_{\cal M}}^{\sf insque}$. For the evaluation of the behavior of ${\cal M}$, we want to consider all these walks and discuss properties that a walk must have in order to represent a computation path of ${\cal M}$. 

\begin{question}[Walks and computation paths]

\hfill

\noindent \fbox{\parbox{11.8cm}{\centering \vspace{0.15cm}
\textcolor{red}{\sf Can any walk from $1$ to $\ell_{\cal P}$ and any infinite walk starting from 1 in $\mathfrak{Fl\,}_{\cal P}^{\sf label}$ \\ represent a computation path or
which conditions have to be satisfied?
} \vspace{0.15cm}}}
\end{question}
 In accordance with Proposition \ref{ComputPathInFlow}, we can say that, for a large number of programs  ${\cal P}\in {\sf P}_{\sigma}$, most walks in $\mathfrak{Fl\,}_{\cal P}^{\sf instr}$ represent  parts of computation paths of machines ${\cal M}$ over arbitrary structures ${\cal A}$ with ${\cal P}_{\cal M}={\cal P}$ and that the set of all these walks includes the set of those walks that contain the theoretically conceivable computation paths. For this reason, we will investigate which walks in $\mathfrak{Fl\,}_{\cal P}^{\sf instr}$ can be really traversed by an ${\cal A}$-machine.

\subsection{Programs paths}
A {\em label path of ${\cal P}\!\in\! {\sf P}_\sigma$} is an arbitrary sequence of labels in ${\cal L}_{\cal P}$. We want to reduce the total number of label paths step by step in order to obtain a selection of paths that could possibly be traversed by a machine. Therefore, it is useful to consider label paths that are determined by walks in $\mathfrak{Fl\,}_{\cal P}^{\sf label}$. They are finite or infinite, belong to ${\cal L}_{\cal P}^\infty$ or ${\cal L}_{\cal P}^\omega$, and can be denoted by $\vec \ell$, $(\ell_t)_{t=1.. s}$, and $(\ell_1,\ldots, \ell_s)$ for some $s\geq 1$ or by $\bar\ell$, $(\ell_t)_{t\geq 1}$, and $(\ell_1,\ell_2,\ldots)$. A sequence $(\ell_1,\ell_2,\ldots\big[,\ell_s\big])$ can only be represented by a walk in $\mathfrak{Fl\,}_{\cal P}^{\sf label}$ if each successor of a label $\ell_t$ results from  the type of  the instruction whose label is ${\ell_t}$ and applying one of the $\beta_{\cal P}^\circ$-functions. Accordingly, we summarize important features of label paths in Overview \ref{PropProgrPaths} (for notations, see also Overview \ref{LabelSets}). A {\sf program path of ${\cal P}$} (also called a {\sf ${\cal P}$-path}) of length $s\geq 2$ is a sequence determined by {\sf (L1)} or {\sf (L2)}. 
\begin{overview}[Properties of ${\cal P}$-paths $(\ell_t)_{t=1.. s}$ and $(\ell_t)_{t\geq 1}$ and more]\label{PropProgrPaths}

\hfill

\nopagebreak 

\noindent \fbox{\parbox{11.8cm}{

\hfill{\small Let ${\cal P}\in {\sf P}_{\sigma}$, $\beta_{\cal P}$, $\beta_{\cal P}^+$, and $\beta_{\cal P}^-$ belong to $\mathfrak{Fl\,}_{\cal P}^{\sf label}$, and \textcolor{blue}{$s\geq 1$}.}

\vspace{0.05cm}
\em
\begin{tabular}{llll}
{\sf (L1)}&$(\ell_1,\ldots, \ell_s)\!\!\!\!\!$&$\!\!\! \in ({\cal L}_{\cal P}^\circ)^{s-1}\times {\cal L}_{{\cal P},{\rm S}}\!\!\!\!\!\!$& if $s>1$. \hspace{0.48cm} \textcolor{gray}{\small\em (only for finite paths)}\\
{\sf (L2)}&$(\ell_1,\ell_2,\ldots)\!\!\!\!\!$&$\!\!\!\in ({\cal L}_{\cal P}^\circ)^\omega$. & \hfill\textcolor{gray}{\small\em (describes infinite paths)}\\
{\sf (L3)}& $\ell_1$&$\!\!\!=1$.\\
{\sf (L4)}$\!\!$&$\ell_{t+1}$&$\!\!\!=\beta_{\cal P}(\ell_t)$ & if $1\leq t\big[< s\big]$ and $\ell_t \in {\cal L}_{\cal P}^{\rm NB}$,\\
&$\ell_{t+1}$&$\!\!\!\in \{\beta_{\cal P}^+(\ell_t),\beta_{\cal P}^- (\ell_t)\}$ & if $1\leq t\big[< s\big]$ and $\ell_t\in {\cal L}_{\cal P}^{\rm B}$.\\
{\sf (L5)}& $\ell_s$&$\!\!\!=\ell_{\cal P}$,\\
{\sf (L6)}& $\ell_{s}=\ell_{s+1}=\!\!\!\!\!$&$\ell_{s+2}=\ell_{s+3}= \cdots $& if {\sf (L5)} holds.
\end{tabular}

\hfill {\small \em {\sf L} stands for {\sf label}.}
}}
\end{overview}
\begin{definition}[Program paths and ${\cal P}$-paths]\label{DefProgrPath} Let ${\cal P}\in {\sf P}_{\sigma}$. For any $s\geq 1$,
a sequence $(\ell_1,\ldots, \ell_s)\in {\cal L}_{\cal P}^s$ is a {\em finite ${\cal P}$-path of length $s$} if there hold {\sf (L1)}, {\sf (L3)}, {\sf (L4)} for all $t<s$, and {\sf (L5)}. A sequence $(\ell_1,\ell_2,\ldots)\in {\cal L}_{\cal P}^\omega$ is an {\em infinite ${\cal P}$-path} if there hold {\sf (L2)}, {\sf (L3)}, and {\sf (L4)} for all $t\geq 1$. A sequence of labels is a {\em ${\cal P}$-path}, if it is a finite ${\cal P}$-path in ${\cal L}_{\cal P}^\infty$ or an infinite ${\cal P}$-path in ${\cal L}_{\cal P}^\omega$. The ${\cal P}$-paths form the set of all {\em program paths of ${\cal P}$}.
\end{definition} 
 Note that ${\sf (L4)}$ can also be expressed by $\ell_t\,\textcolor{red}{\to_{\cal P}^{\rm lab}}\,\,\ell_{t+1}$. For the introduction of a suitable transition system $({\sf S}_{\cal P}^{\rm lab},\to_{\cal P}^{\rm lab})$ that can be used for generating ${\cal P}$-paths, see Overview \ref{TransPPaths}. $\ell\to_{\cal P}^{\rm lab}\ell'$ holds if and only if $\ell \in {\cal L}_{\cal P}^{\rm NB}$ and $\ell'=\beta_{\cal P}(\ell)$ are satisfied or we have $\ell\in {\cal L}_{\cal P}^{\rm B}$ and $\ell'\in \{\beta_{\cal P}^+(\ell),\beta_{\cal P}^- (\ell)\}$.

\begin{propert}\label{PropertOnePath}If $\,\vec \ell\, $ is a ${\cal P}$-path of length $1$, then $\ell_1=\ell_{\cal P}$ holds. Consequently, the existence of a ${\cal P}$-path $\,\vec \ell\, $ of length $1$ implies that ${\cal P}$ is the program  {\sf \,$1\!:$ stop.}\hspace{0.1cm}and $\vec \ell $ is the only program path of ${\cal P}$. $s\!>\!1$ and ${\sf (L1)}$ together imply ${\sf (L5)}$. 
\end{propert}
\begin{proposition}[Prefix-freeness]\label{NoInitIsPath}
No program path is an initial segment of another program path.
\end{proposition} 
Consequently, an order on the set of all ${\cal P}$-paths can be derived solely from the functions $\beta_{\cal P}^+$ and $\beta_{\cal P}^-$ in $\mathfrak{Fl\,}_{\cal P}^{\sf label}$ as described in Overview \ref{OrderAnyPaths}. 

Each computation path of a machine $\cal M$ belongs to ${\cal L}_{{\cal P}_{\cal M}}^\infty \cup {\cal L}_{{\cal P}_{\cal M}}^\omega$. By definition, the properties summarized in Overview \ref{PropProgrPaths} are sufficient to identify the ${\cal P}_{\cal M}$-paths and each computation path of ${\cal M}$ must also have these properties by Proposition \ref{ComputPathInFlow}.
\begin{proposition}[Computation paths are program paths]$\!$Any computation path of an ${\cal A}$-machine ${\cal M}$ is a ${\cal P}_{\cal M}$-path and thus no computation path is an initial segment of another computation path.
\end{proposition}

{\bf The walks $W_{B^*}$.} Let ${\cal P}\in {\sf P}_{\sigma}$. For 
 any ${\cal P}$-path $(\ell_t)_{t\geq 1}$, let $W_{(\ell_t)_{t\geq 1}}$ be the infinite walk $(\ell_1,(\ell_1,\ell_2),\ell_2,(\ell_2,\ell_3), \ldots )$ in $\mathfrak{Fl\,}_{\cal P}^{\sf label}$. For $s\geq 1$ and any finite ${\cal P}$-path $(\ell_t)_{t=1.. s} $, let $W_{(\ell_t)_{t=1.. s}}$ be the finite walk $ (\ell_1,(\ell_1,\ell_2),\ldots,\ell_{s-1},(\ell_{s-1},\ell_s), \ell_s)$. The statements of Proposition \ref{PropertWalks} can be made stronger.

\begin{proposition}[${\cal P}$-paths represented in flowcharts]\label{PropertPathsAsWalks} Let ${\cal P}\!\in\! {\sf P}_{\sigma}$ and $B^*$ be any ${\cal P}$-path. Then, $W_{B^*}$ is the walk determining $B^*$ and representing $B^*$ in $\mathfrak{Fl\,}_{\cal P}^{\sf label}$. If $W_{B^*}$ is a finite walk from $v_1$ to $v_s$ and not simple, then properties (1) to (3) are satisfied.
\parskip -0.4mm
\begin{enumerate}[label=(\arabic*)] \parskip -0.5mm
\item $v_1=1$ holds and $v_s=\ell_{\cal P}$ is satisfied.
\item $W_{B^*}$ contains a node $v_i$ with two outgoing edges and $v_i\not = v_s$.
\item There exists an infinite ${\cal P}$-path and the number of finite ${\cal P}$-paths is infinite.
\end{enumerate}
\end{proposition} 

{\bf Proof.} Let $W_{B^*}$ be a finite walk $(v_1,e_1, v_2, \ldots, e_{s-1},v_s)$ that is not simple and $v_j$ be the last node with $v_i=v_j$ for some $i<j$. Then, $v_j\not =v_s$ and $e_j$ leads to the end $(v_{j+1},e_{j+1}, \ldots, v_s)$ of $W_{B^*}$ (that does not contain $v_j$) whereas $e_i$ is a second outgoing edge of $v_j$ that belongs to the walk $(v_i,e_i, \ldots, v_j)$. Further walks result from adding $(e_i,v_{i+1}, \ldots, v_j)$ after $v_j$ such that $\!${\em (2)} and $\!${\em (3)} hold.
\qed

\vspace{0.2cm}
The following discussion focuses on fundamental properties of several sets of program paths. The main purpose of this discussion is to highlight characteristics that remain significant even when considering further types of instructions and other ${\cal A}$-machines. 

\vspace{0.1cm}
{\bf The set ${\sf Lab}_{\cal P}$.} For ${\cal P}\in {\sf P}_{\sigma}$, let ${\sf Lab}_{\cal P}$ be the the set of all \textcolor{blue}{finite} and all \textcolor{blue}{infinite} program paths of ${\cal P}$. Further notations for several subsets of ${\sf Lab}_{\cal P}$ are given in Overview \ref{SetsLabelsProgrPath}. For these subsets, we have
${\sf Lab}_{\cal P}^{\rm fin}\cap {\sf Lab}_{\cal P}^{\rm infin}=\emptyset$,
${\sf Lab}_{\cal P}^{\rm fin}\cup {\sf Lab}_{\cal P}^{\rm infin}={\sf Lab}_{\cal P}\not=\emptyset$, and
${\sf Lab}_{\cal P}^{\rm infin}\subseteq$ ${\sf Lab}_{\cal P}\cap{\sf Lab}_{\cal P} ^{[\leq \omega]}$.
${\sf Lab}_{\cal P}^{\rm fin} \!=\! \emptyset$ implies ${\sf Lab}_{\cal P}^{\rm infin} =$ ${\sf Lab}_{\cal P}$ $={\sf Lab}_{\cal P} ^{[\leq \omega]}$. 
${\sf Lab}_{\cal P}^{\rm infin} \!=\! \emptyset$ implies ${\sf Lab}_{\cal P}^{\rm fin} ={\sf Lab}_{\cal P}$ and ${\sf Lab}_{\cal P}\cap {\sf Lab}_{\cal P} ^{[\leq \omega]}=\emptyset$. 
\footnote{\label{Footn} ${\cal L}_{\cal P}^\omega$ consists of the infinite sequences with components in ${\cal L}_{\cal P}$, $\{0,1\}^{(\omega)}$ consists of the infinite strings with characters in $\{0,1\}$, ${\sf Lab}_{\cal P} ^{[\leq \omega]}$ is defined in Overview \ref{SetsLabelsProgrPath}, etc. The superscript $^{[\leq \omega]}$ states that we come from sequences $B_0$ of length $\leq \omega$ in $ {\cal L}^\infty\cup {\cal L}^\omega$ and extend them to obtain infinite sequences $B_0^{[\leq \omega]}$ in ${\cal L}^\omega$. The small brackets $[$ and $]$ in the strings such as ${\sf Lab}_{\cal P}^{[\leq \omega]}$ and ${\sf Lab}_{{\cal P},n}^{[*]}$ belong to the notation, whereas we use the big brackets $\big[$ and $\big]$ for indicating options.}
\begin{overview}[Notations: Sets of ${\cal P}$-paths and other label paths]\label{SetsLabelsProgrPath}

\hfill

\nopagebreak 

\noindent \fbox{\parbox{11.8cm}{

\hfill{\small Let ${\cal P}\in {\sf P}_{\sigma}$ and $n\geq 1$.}

\em 
\begin{tabular}{ll}
${\sf Lab}_{\cal P}^{\rm fin}$&the set of all finite ${\cal P}$-paths,\\
${\sf Lab}_{\cal P}^{\rm infin}\!\!$& the set of all infinite ${\cal P}$-paths, \\
${\sf Lab}_{\cal P}$&$\!\!\!=_{\rm df} {\sf Lab}_{\cal P}^{\rm fin}\cup {\sf Lab}_{\cal P}^{\rm infin}$.\vspace{0.05cm}\\
${\sf Lab}_{\cal P} ^{[\leq \omega]}\!\!$& the set of all label paths $B_0^{[\leq \omega]}\in {\cal L}_{\cal P}^\omega$ derived from $B_0\in {\sf Lab}_{\cal P}$ by\\ &
\begin{tabular}{lll} & $\textcolor{blue}{\vec \ell^{\,[\leq \omega]}}=(\ell_1,\ldots,\ell_s,\ell_s,\ldots )$&
if $\vec \ell=\!(\ell_1,\ldots,\ell_s)\! \in {\cal L}_{\cal P}^s$ and\\ &$\textcolor{blue}{\bar \ell ^{\,[\leq \omega]}}=\bar \ell$ & if $\bar \ell\in {\cal L}_{\cal P}^{\omega}$. \end{tabular}\\
${\sf Lab}_{{\cal P}}^{\rm Co}$& the set of all ${\cal P}$-paths containing a subsequence $(\ell_i,\ell_{i+1},\ell_{i+2})$\\& with $\ell_i=\ell_{i+1}\not=\ell_{i+2}$.\\
${\sf Lab}_{{\cal P},n}^{\rm Co}$& the set of all ${\cal P}$-paths for which it can be ensured\\& by evaluating index instructions that they cannot be traversed\\
& by any ${\cal A}$-machine ${\cal M}$ with ${\cal P}_{\cal M}={\cal P}$ for any inputs of length $n$\\&\hfill{\small \em \textcolor{gray}{(for details, see also property {\sf (P6)} in Theorem \ref{PathCanc})}},\\
${\sf Lab}_{{\cal P},n}^{[*]}$& the set of all ${\cal P}$-paths that do not belong to ${\sf Lab}_{{\cal P},n}^{\rm Co}$. 
\end{tabular}}}
\,
\noindent \fbox{\parbox{11.8cm}{
\hspace{0.3cm} {\em For $\bar\ell \in {\sf Lab}_{\cal P} ^{[\leq \omega]}$, let $\textcolor{blue}{\textcolor{blue}{B_{\bar \ell}}}=\vec \ell $ \,\, if $\bar \ell =\vec \ell ^{\,[\leq \omega]}$ and $\vec \ell \in{\sf Lab}_{\cal P}^{\rm fin}$ and 

\hspace{3.35cm} $\textcolor{blue}{B_{\bar \ell}}=\bar \ell $ \,\, if $\bar \ell \in {\sf Lab}_{\cal P}^{\rm infin}$.}

\vspace{0.2cm}
\hspace{0.3cm} {\small\em Note: For $B_1,B_2\in {\sf Lab}_{\cal P}$, $B_1\not= B_2 \mbox{ if and only if }B_1^{[\leq \omega]}\not= B_2^{[\leq \omega]}$ 
(by Prop.\,\ref{NoInitIsPath}).}
}}
\end{overview}

\begin{example}[${\cal P}$-paths]\label{ExBerWurzeln1}
Let ${\cal P}$ be the $(2;2,2;2)$-program given in Fig. \ref{BerWurzeln1}.

\noindent
Sets of labels:

\qquad ${\cal L}_{{\cal P},{\rm T}}=\{6\}$, ${\cal L}_{{\cal P},{\rm H}_{\rm T}}=\{1\}, {\cal L}_{{\cal P},{\rm F}_{0}}=\{2,3,7\}, {\cal L}_{{\cal P},{\rm F}}=\{4,5\}, {\cal L}_{{\cal P},{\rm S}}=\{8\}$.

\vspace{0.05cm}
\noindent Some program paths (summarized and illustrated in Figure \ref{BerWurzelnProgrammpfade}):

\vspace{0.05cm}

\begin{tabular}{l}
$B_1 =(1,2,3,4,5,6,7,8)\in {\cal L}_{\cal P}^\infty$,
$B_2 =(1,2,3,4,5,6,5,6,7,8)\in {\cal L}_{\cal P}^\infty$,\\
$B_3 =(1,2,3,4,5,6,5,6,5,6,7,8)\in {\cal L}_{\cal P}^\infty$, \ldots,\\
$B_0 =(1,2,3,4,5,6,5,6,5,6,5,6,\ldots) \in {\cal L}_{\cal P}^\omega$, 
$B_0' =(1,1,1,\ldots) \in {\cal L}_{\cal P}^\omega$,\\
\end{tabular}

\begin{tabular}{l}
$B_1' =(1,1,2,3,4,5,6,7,8) \in {\cal L}_{\cal P}^\infty$, 
$B_1'' =(1,1,1,2,3,4,5,6,7,8) \in {\cal L}_{\cal P}^\infty$.
\end{tabular}

\vspace{0.05cm}

\noindent A label path:

\hspace{0.2cm}$(B_1)^{[\leq \omega]} \hspace{0.15cm}=(1,2,3,4,5,6,7,8, 8,8\ldots )\in {\cal L}_{\cal P}^\omega$.

\vspace{0.05cm}

\noindent Some sets of program paths and label paths:

\vspace{0.05cm}
\begin{tabular}{lll}
$ {\sf Lab}_{\cal P}$&$=$&$\{B_0, B_0', B_1,B_2,\ldots, B_1',B_1'',\ldots \}$,\\
$ {\sf Lab}_{\cal P}^{\rm fin}$&$=$&$\{B_1,B_2,\ldots, B_1',B_1'',\ldots \}$,\\
$ {\sf Lab}_{\cal P}^{\rm infin}$&$=$&$\{B_0,B_0',\ldots \}$,\\
${\sf Lab}_{\cal P} ^{[\leq \omega]}$&$=$&$\{ B_0, B_0', (\overbrace{1,2,3,4,5,6,7,8}^{B_1},8,8, \ldots ),$\\
& &\hspace{1.58cm}$(\underbrace{1,2,3,4,5,6,5,6,7,8}_{B_2},8,8,\ldots),\ldots \}$,\\
${\sf Lab}_{{\cal P},n}^{[*]}$&$=$&$ \{B_0'\}$ \hfill for $n\geq 2$,\\
${\sf Lab}_{{\cal P},1}^{[*]}$&$=$&${\{B_0, B_1,B_2,\ldots \} \subseteq \sf Lab}_{{\cal P},n}^{\rm Co}$\hfill for $n\geq 2$,\\
$ {\sf Lab}_{{\cal P},1}^{\rm Co}$&$\supseteq$&$\{B_0', B_1', B_1'',\ldots \}$,\\
${\sf Lab}_{{\cal P}}^{\rm Co}$&
$\supseteq$&$\{B_1', B_1'',\ldots\} \subseteq {\sf Lab}_{{\cal P},n}^{\rm Co}$ \hfill for $n\geq 1$.\\
\end{tabular}
\end{example}

For machines, we introduce the following sets.
\begin{overview}[Notations: Sets of computation paths and more]\label{SetsLabels}

\hfill

\nopagebreak 

\noindent \fbox{\parbox{11.8cm}{

\hfill{\small Let ${\cal M}$ be any BSS RAM over ${\cal A}$, $\vec x \in U_{\cal A}^\infty$, and $n\geq 1$.}

\em 
\begin{tabular}{ll}
${\sf Lab}_{\cal M}$& the set of all computation paths of ${\cal M}$,\\
${\sf Lab}_{{\cal M},\vec x}\!\!\!$& the set of all computation paths of ${\cal M}$ traversed by input $\vec x$,\\
${\sf Lab}_{{\cal M},n}\!\!\!\!$& $=_{\rm df }\bigcup_{\vec x\in U_{\cal A}^n }{\sf Lab}_{{\cal M},\vec x} $ \hspace{2.7cm} \textcolor{gray}{$\Rightarrow$ ${\sf Lab}_{{\cal M}} = \bigcup_{n\geq 1}{\sf Lab}_{{\cal M},n} $}. \vspace{0.1cm} \\
${\sf Lab}_{{\cal M},\vec x} ^{[\leq \omega]}\!\!\!\!$& the set of all label paths $(\ell_1,\ell_2,\ldots) \in {\sf Lab}_{{\cal M},\vec x} \cap {\sf Lab}_{{\cal P}_{\cal M}}^{\rm infin}$\\
 &\hspace{1cm} and $(\ell_1,\ldots, \ell_s, \ell_s, \ldots)$ with $(\ell_1,\ldots, \ell_s)\in {\sf Lab}_{{\cal M},\vec x}\cap {\sf Lab}_{{\cal P}_{\cal M}}^{\rm fin}$,\vspace{0.1cm}\\
${\sf Lab}_{{\cal M},n} ^{[\leq \omega]}\!\!\!\!$& $=_{\rm df }\bigcup_{\vec x\in U_{\cal A}^n }{\sf Lab}_{{\cal M},\vec x} ^{[\leq \omega]}$, \quad ${\sf Lab}_{{\cal M}} ^{[\leq \omega]}\,=_{\rm df }\bigcup_{n\geq 1}{\sf Lab}_{{\cal M},n} ^{[\leq \omega]}$, \quad $\ldots$\\
\end{tabular}}}
\end{overview}

\begin{example}[Computation paths traversed by a machine]\label{ExBerWurzeln2}

Let ${\cal P}$ be the $(2;2,2;2)$-program given in Fig.\,\,\ref{BerWurzeln1}, let ${\cal A}=(\bbbr; 1,-1;+,\cdot \,;=)$, and let ${\cal M}$ be a machine in ${\sf M}_{\cal A}$ or ${\sf M}_{\cal A}^{\rm ND}$ with ${\cal P}_{\cal M}={\cal P}$. Then, each computation path of ${\cal M}$ is a ${\cal P}_{\cal M}$-path given by a walk in $\mathfrak{Fl\,}_{{\cal P}_{\cal M}}^{\sf label}$. For better understanding the evaluation of ${\cal P}$-paths, see also $\mathfrak{Fl\,}_{\cal P}^{\sf insque}$ in Fig.\,\,\ref{BerWurzeln2}. Some ${\cal P}$-paths considered in the Example \ref{ExBerWurzeln1} are also computation paths of ${\cal M}$:

\vspace{0.05cm}

$B_1$ is a finite computation path of ${\cal M}$ traversed by input $0$, 

$B_2$ is a finite computation path of ${\cal M}$ traversed by input $1$ and input $-1$, 

$B_3$ is a finite computation path of ${\cal M}$ traversed by input $\sqrt{2}$ and input $-\sqrt{2}$, 

$B_0\in {\cal L}_{{\cal P}_{\cal M}}^\omega$ is traversed by all numbers in $\bbbr \setminus \{0,1,-1,\sqrt{2}, -\sqrt{2}, \sqrt{3}, \ldots\}$,

$B_0'\in {\cal L}_{{\cal P}_{\cal M}}^\omega$ is traversed by all tuples in $\bbbr^\infty \setminus \bbbr$.

\vspace{0.05cm}
\noindent The set of all finite computation paths traversed by ${\cal M}$ is a proper subset of $ {\sf Lab}_{{\cal P}_{\cal M}}^{\rm fin}$. We have $B_1'\in {\sf Lab}_{{\cal P}_{\cal M}}^{\rm fin}\setminus {\sf Lab}_{\cal M}$. The inputs traversing one of the finite computation paths $B_1$, $B_2$, $\ldots$ form together the halting set $H_{\cal M}$ of ${\cal M}$.
\end{example}
\begin{proposition}[Computation and ${\cal P}_{\cal M}$-paths]
For any $\sigma$-structure ${\cal A}$, let ${\cal M}\in {\sf M}_{\cal A}\cup {\sf M}_{\cal A}^{\rm ND}\cup {\sf M}_{\cal A}^{\rm DND} \cup {\sf M}_{\cal A}^{\rm NDB}$. We have ${\sf Lab}_{{\cal P}_{\cal M}}^{\rm Co}\cap {\sf Lab}_{\cal M}=\emptyset$ and ${\sf Lab}_{{\cal M},\vec x}\subseteq{\sf Lab}_{{\cal M},n}\subseteq{\sf Lab}_{{\cal P}_{\cal M},n}^{[*]}\subseteq{\sf Lab}_{{\cal P}_{\cal M}}$ for any $n\geq 1$ and $\vec x \in U_{\cal A}^n$.
\end{proposition}

The following overview provides further insights and outlooks. \footnote{ For $\ell_2,\ldots,\ell_1$ etc., see fn.\,\,\ref{EmptyList}.}
\begin{overview}[From configurations to graphs and other structures]\label{FromConfToPaths}

\hfill

\nopagebreak 

\noindent \fbox{\parbox{11.8cm}{
\centering\em
\begin{tabular}{c}
Any \textcolor{blue}{ $({\cal M},\vec x)$-path} of configurations in ${\sf Con}_{{\cal M},\vec x}$ for an input $\vec x$\\ determines\\
{\bf a computation path} $(\ell_1,\ell_2,\ldots[,\ell_s\big])$ of a machine ${\cal M}$.\\
\begin{tabular}{c} ------------------------------------------------------------------------------------------\\\end{tabular}\\
Any \textcolor{blue}{computation path} in ${\sf Lab}_{\cal M}$\\
is\\a {\bf ${\cal P}_{\cal M}$-path} $(\ell_1,\ell_2,\ldots\big[,\ell_s\big])$ in ${\sf Lab}_{{\cal P}_{\cal M}}$.\\
\begin{tabular}{c} ------------------------------------------------------------------------------------------\\ 
\end{tabular}\\
Each \textcolor{blue}{${\cal P}$-path} $B^*$ of a program ${\cal P}$ in ${\sf Lab}_{\cal P}$\\
can be represented by\\
 a {\bf directed walk} $W_{B^*}$ in $\mathfrak{Fl\,}_{\cal P}^{\sf label}$ given by a sequence\\
$(\ell_1, (\ell_1,\ell_2),\ell_2, \ldots\big[,(\ell_{s-1},\ell_s),\ell_s\big]) $ of nodes and edges.\\
\begin{tabular}{c} ------------------------------------------------------------------------------------------\\\end{tabular}\\ 
Such a \textcolor{blue}{walk} in $\mathfrak{Fl\,}_{\cal P}^{\sf label}$\\ can be converted into\\a {\bf structure}\\ 
containing a function $label: \{1,2,\ldots\big[,s\big]\}\to{\cal L}_{\cal P}$ with\\ 
$label(t)=\ell_t$ for $t\geq 1$ \big[and $t\leq s$\big].\\
\begin{tabular}{c} ------------------------------------------------------------------------------------------\\\end{tabular}\\
Such \textcolor{blue}{structures}\\ can be\\ {\bf labeled paths} (or {\bf linear graphs}) of the form\\
$((\{1,2,\ldots\big[,s\big]\}, \{(1,2),(2,3),\ldots\big[,(s-1,s)\big]\}), label) $,\vspace{0.1cm}\\
{\bf single-sorted structures} 
given by\\
$(\{1,2,\ldots\big[,s\big]\}\cup {\cal L}_{\cal P};\,\,1,\ldots,\ell_{\cal P}\big[,s\big];\,\, label;\,\,\{(1,2),(2,3),\ldots\big[,(s-1,s)\big]\}) $\\ 
 with ${\cal L}_{\cal P}\subset \bbbn_+$ \big[and a partial function $label$ if $s<|{\cal L}_{\cal P}|$\big],  etc.
\end{tabular}}}
\end{overview}

\noindent {\bf Our planned approach to evaluating paths, potential applications, and benefits.} 
Taking into account the observations summarized in Overview \ref{FromConfToPaths}, several kinds of paths and walks can be considered and evaluated. Moreover, the members of program paths can be extended by adding indices and $\sigma$-terms, and certain sequences of these extensions can be generated by new transition systems for describing the execution of machines at different levels of abstraction. 

\begin{itemize} \parskip -0.2mm
\item Each {\sf partial configuration} contains a label and indices. The sequences of {\sf partial configurations} can be syntactically generated by {\sf transition systems} ${\cal S}_{\cal P}^{\rm ind}$ and ${\cal S}_{{\cal P},\kappa}^{\rm ind}$ derived from a program ${\cal P}$ for any $\kappa \geq k_{\cal P}$ (cf.\,\,below).
\item {\sf Graph-theoretical representations} of program paths mentioned in Overview \ref{FromConfToPaths} and collections of such representations can be summarized in trees (cf.\,\,{\sf Extras II}). These trees help to better understand the meaning of partial configurations and the consequences of executing index instructions.
\item A further extension of partial configurations by adding {\sf $\sigma$-terms} leads to {\sf term configurations} that can be used to describe all branching conditions by literals uniquely (cf.\,\,{\sf Extras III}). These $\sigma$-literals belong to {\sf decision trees} that can be generated by {\sf transition systems} ${\cal S}_{{\cal P},\kappa}$ for $\kappa \geq k_{\cal P}$ (cf.\,\,Overview \ref{OutlookTermConf1}).
\item Path graphs~-- also called paths (cf.\,\cite[p.\,6]{Diestel06} and \cite[p.\,83]{Guichard}) or linear graphs~-- represent program paths $B^*$. They can be  mapped isomorphically and label-compatibly  into {\sf trees of term configurations}. This enables a detailed characterization of the behavior of machines ${\cal M}$  based on evaluating single systems ${\sf sy}_{{\cal P}_{\cal M},n}^{\sf decision}(B^*)$ of $\sigma$-literals (cf.\,\,{\sf Extras III}). 
\end{itemize}
In the following, we will discuss the properties of partial configurations in detail.

\section{Partial configurations: Labels and indices}\label{PropertRestrConf0}
We want to extend all members of the ${\cal P}$-paths for any given program ${\cal P}\in {\sf P}_{\sigma}$. All components $\ell_t$ of any label path $\bar \ell\in {\sf Lab}_{\cal P}^{[\leq\omega]}$ should be supplemented by tuples $\vec \nu^{(t)}$ of indices that could be computed by BSS RAMs ${\cal M}\in {\sf M}_{\cal A} \cup {\sf M}_{\cal A}^{\rm ND}\cup {\sf M}_{\cal A}^{\rm DND} \cup {\sf M}_{\cal A}^{\rm NDB}$ with ${\cal P}_{\cal M}={\cal P}$ along the computation path $B_{\bar \ell}$ for inputs in $U_{\cal A}^n$ if $B_{\bar \ell}\in {\sf Lab}_{{\cal M},n}$ holds. For any $\ell\in {\cal L}_{\cal M}$ and $\vec \nu\in \bbbn_+^{k_{\cal M}}$, $(\ell\,.\,\vec \nu)$ is called a {\em partial configuration of ${\cal M}$} or a {\em partial ${\cal M}$-configuration}. ${\sf S}_{\cal M}^{\rm ind}$ is the set of all partial ${\cal M}$-configurations and determined by ${\cal L}_{\cal M}$ and $k_{\cal M}$. The elements of $({\sf S}_{\cal M}^{\rm ind})^\omega$ are used to group the label paths and investigate the ${\cal P}_{\cal M}$-paths with the aim of evaluating the execution of ${\cal P}_{\cal M}$ by ${\cal M}$ on inputs of a fixed length. For characterizing the behavior of ${\cal M}$, we introduce some sets of sequences of ${\cal M}$-configurations (cf.\,\,Overview \ref{SetsCon}), several sets of sequences of partial configurations determined by $\sigma$-programs ${\cal P}$ (cf.\,\,Overview \ref{SetsPartCon2}), and in particular the sets $ {\sf Con}^{\rm ind}_{{\cal P},\kappa,n}$. The sets ${\sf Con}_{{\cal P}_{\cal M}, k_{\cal M}, n}^{\rm ind}$ enable us to take a further step toward reducing the set all sequences of ${\cal M}$-configurations that could be possibly generated by ${\cal M}$ for inputs of length $n$. For any $\bar \ell\in{\sf Lab}_{{\cal P}_{\cal M}}^{[\leq \omega]}$, $n\geq 1$, and all inputs $\vec x\in U_{\cal A}^n$ traversing $B_{\bar \ell}\in {\sf Lab}_{{\cal M},n}$, all $t$-th members of all $({\cal M},\vec x)$-paths $(\ell_t\,.\,\vec \nu^{(t)}\,.\,\bar u^{(t)})_{t\geq 1}$ have the same initial parts $(\ell_t\,.\,\vec \nu^{(t)})$ and the sequence $(\ell_t\,.\,\vec \nu^{(t)})_{t\geq 1}$ of partial ${\cal M}$-configurations belongs to ${\sf Con}_{{\cal P}_{\cal M}, k_{\cal M}, n}^{\rm ind}$. Our tools allow us to describe, for any BSS RAM ${\cal M}$ over a $\sigma$-structure ${\cal A}$ with ${\cal P}_{\cal M}={\cal P}$ and $k_{\cal M}=\kappa$, exactly how a suitable sequence $(\ell_t\,.\,\vec \nu^{(t)})_{t\geq 1}$ can be derived from ${\cal P}$. $ {\sf Con}^{\rm ind}_{{\cal P},\kappa,n}$ includes the set of these sequences and can be defined without using $\to_{\cal M}$. For any $\bar \ell\in {\sf Lab}_{\cal P} ^{[\leq \omega]}$, a suitable sequence $(\vec \nu^{(t)})_{t\geq 1} \in (\bbbn_+^{\kappa})^\omega$ can be constructed in a unique manner by using only ${\cal P}\in {\sf P}_\sigma$, $\kappa\geq k_{\cal P}$, and $n\geq 1$. Neither the $Z$-registers of ${\cal M}$ nor their contents are important. Consequently, we can consider the $\sigma$-programs themselves and the transformation of indices by using ${\cal A}_\bbbn$ for determining and evaluating partial configurations. Let us first describe some observations that are important in this context and lead to the mentioned statements step by step.

\subsection{Partial configurations of BSS RAMs}\label{PropertRestrConf}
We know by \cite[Example 3.4]{GASS25A} that the restriction of the input space of a BSS RAM ${\cal M}$ to inputs of a fixed length $n$, {\bf can lead} to a proper inclusion ${\sf Lab}_{{\cal M},n}\subset {\sf Lab}_{\cal M} $. Examples \ref{ExBerWurzeln1} and \ref{ExBerWurzeln2} show that ${\sf Lab}_{\cal M} \subset {\sf Lab}_{{\cal P}_{\cal M}} $ is possible since $B_1'$ and $B_1''$ are not computation paths of the machine ${\cal M}$ considered in Example \ref{ExBerWurzeln2}. This means that a program path in ${\sf Lab}_{{\cal P}_{\cal M}}$ does not necessarily belong to ${\sf Lab}_{\cal M}$. One possible reason for this is that each execution of an index instruction is a transformation step determined by the {\sf Peano structure ${\cal A}_{\bbbn}$}. In any case, we use ${\cal A}_{\bbbn}=_{\rm df}(\bbbn_+; 1;succ;=)$ and assume that $succ(n)=n+1$ holds for all $n\in \bbbn_+$. Under this assumption, the term $I_j+1$ can be considered to be a formal string that stands, e.g., for $f_0^ 1(I_j)$, such that it can be interpreted by using the function $succ:\bbbn_+\to \bbbn_+$ in all cases and independent of any (further) underlying $\sigma$-structure ${\cal A}$. For evaluating index registers, we can use only the identity relation that belongs to ${\cal A}_\bbbn$. 
\begin{observ} 
Let  $B^*$ be a ${\cal P}_{\cal M}$-path of a deterministic or non-deterministic BSS RAM ${\cal M}$. If ${\cal M}$ traverses this path for inputs, then the content $\nu_{t+1,j}$ of an index register $I_j$ of ${\cal M}$ after $t$ steps is the value of a linear function $f_{B^*,j}(n)=n+t_0$ (only for $j=1$ possible) and a constant function $f_{B^*,j}(n)=1+t_0$, respectively, depending on the length $n\geq 1$ of the input and determined by an integer $t_0\in\{0,\ldots,t\}$ which results from the numbers of executed ${\rm H}_{1}$- and ${\rm H}_{+1}$-instructions applied to $I_j$. For all inputs of length $n$, every index test along this path can be described by two terms (logical expressions) in an atomic formula of the form $n+t_1\!=n+t_1$, $n+t_1\!=1+t_2$, or $1+t_1\!=1+t_2$ for some $t_1,t_2\geq 0$ \footnote{ More formally, we can say that we have terms like $f_0^1(f_0^1(f_0^1(f_0^1(c_0^0)) ))$~-- $(f_0^1)^4(c_0^0)$ for short~--,  the literal $r_0^2((f_0^1)^{(n-1)+t_1}(c_0^0),(f_0^1)^{t_2}(c_0^0))$, and further similar literals. The signature of these literals is $(1;1;2)$. To distinguish the new symbols from the symbols in the $\sigma$-program ${\cal P}_{\cal M}$, we use here the index $_0$. Thus, $c_0^0$ is the symbol for the constant 1, etc.}. 
\end{observ}
So let us assume that the index registers of a BSS RAM ${\cal M}$ over ${\cal A}$ receive the initial values given by $(n,1,\ldots,1)\in \bbbn_+^{k_{\cal M}}$ at the beginning. Let $(\ell_1,\ldots, \ell_t)$ be an initial part of a computation path $B^*\in {\sf Lab}_{{\cal M},n}$ resulting from $t-1$ computation steps by ${\cal M}$ on some input in $U_{\cal A}^n$ and let $(\ell_1,\ldots,\ell_t, \ell_{t+1})$ be an initial part of a ${\cal P}_{\cal M}$-path $B_0$. For $\ell_t\in {\cal L}_{{\cal P}_{\cal M}}^{\rm NB}$, $\ell_{t+1}$ is uniquely determined by $\ell_{t+1}=\ell_t+1$ and also the next label in $B^*$. For $\ell_t\in {\cal L}_{{\cal M},{\rm H}_{\rm T}}$, the values $\nu_{t,1},\ldots,\nu_{t,k_{\cal M}}$ stored in the index registers and the test function $T_{\ell_t}:(\bbbn_+)^{k_{\cal M}}\to {\cal L}_{\cal M}$ in ${\cal F}_{\cal M}$ \footnote{ \label{FootnTl} For each label $\ell$ of an instruction of the ${\rm B}$-type $(I_j=I_k,l_1,l_2)$, we have $T_{\ell}(\vec \nu)=l_1$ if $ \nu_j=\nu_k$ and $T_{\ell }(\vec \nu)=l_2$ if $ \nu_j\not=\nu_k$. For more, see Overview \ref{TransRelConfM} \,and \cite[p.\,10]{GASS25A}. $T_{\ell}$ depends on ${\cal M}$. However, let us mention that, for all BSS RAMs ${\cal M}$, $T_{\ell}:(\bbbn_+)^{k_{\cal M}}\to {\cal L}_{\cal M}$ is defined in the same way if $\ell$ belongs to ${\cal L}_{{\cal M},{\rm H_T}}$ and marks an instruction of ${\rm B}$-type $(I_j=I_k,l_1,l_2)$.} are sufficient to determine the next possible label $T_{\ell_t}(\vec \nu^{(t)})$ in $B^*$ and to recognize whether $T_{\ell_t}(\vec \nu^{(t)})=\ell_{t+1}$ holds or not. Only if $T_{\ell_t}(\vec \nu^{(t)})=\ell_{t+1}$ holds, then $B_0$ could belong to ${\sf Lab}_ {{\cal P}_{\cal M},n}^{[*]}$ and could possibly be traversed by ${\cal M}$ on an input of length $n$. For a discussion of further details, the following notations are useful. 

\begin{overview}[Notations: Sets of $({\cal M},\vec x)$-paths and more]\label{SetsCon}

\hfill

\nopagebreak 

\noindent \fbox{\parbox{11.8cm}{

\hfill{\small Let ${\cal M}$ be a BSS RAM over ${\cal A}$, $n\geq 1$, $\vec x\in U_{\cal A}^n$, $\vec y \in U_{\cal A}^\infty$,

\hfill and $\bar \ell=(\ell_1,\ell_2,\ldots)\in {\sf Lab}_{{\cal P}_{\cal M}}^{[\leq \omega]}$.}

\em
$\!\!\!$\begin{tabular}{ll}
${\sf Con}_{{\cal M},\vec x}$ & the set of all $({\cal M},\vec x)$-paths $(con_t)_{t\geq 1}$\hfill \textcolor{gray}{\em \small (with $(\vec x, con_1)\in{\rm Input}_{\cal M}$)},\\
${\sf Con}_{{\cal M},\vec x,\vec y}\!\!\!\!$& the set of all $(con_t)_{t\geq 1}\in {\sf Con}_{{\cal M},\vec x}$ with\\& \hfill $con_1=(1\,.\,(n,1,\ldots,1)\,.\,(\vec x\,.\,\vec y\,.\,(x_n,x_n,\ldots)))$,\\
${\sf Con}_{{\cal M},n}$& $=_{\rm df} \bigcup_{\vec x\in U_{\cal A}^n}{\sf Con}_{{\cal M},\vec x}$, the set of all {\em $({\cal M},n)$-paths},\\
${\sf Con}_{{\cal M},\bar \ell,n}\!\!\!$ & the set of all $(con_t)_{t\geq 1}\in {\sf Con}_{{\cal M},n}$ for which there is a sequence\\
& \hspace{0.8cm} $(\vec \nu ^{(t)}\,.\,\bar u^{(t)})_{t\geq 1}$ such that $con_t=(\ell_t\,.\,\vec \nu ^{(t)}\,.\,\bar u^{(t)})$ for all $t\geq 1$.\\
\end{tabular}

\vspace{0.05cm}
\noindent $\!\!${\small\em \begin{tabular}{lll}\hline Thus,&$\!\!\!\!{\sf Con}_{{\cal M},\vec x}=\bigcup_{\vec y \in U_{\cal A}^\infty}{\sf Con}_{{\cal M},\vec x,\vec y}$&$\!\!$holds for ${\cal M}\in {\sf M}_{\cal A}^{\rm ND}$,\\
&$\!\!\!\!{\sf Con}_{{\cal M},\vec x}=\bigcup_{\vec y \in \{c_1,c_2\}^\infty}{\sf Con}_{{\cal M},\vec x,\vec y}$ &$\!\!$holds for ${\cal M}\in {\sf M}_{\cal A}^{\rm DND}$ (${\cal A}\in {\sf Struc}_{c_1,c_2}$).\\
\end{tabular}}

\vspace{0.2cm}
{\scriptsize\em $|{\sf Con}_{{\cal M},\vec x}|\!=\!1$ for ${\cal M}\in {\sf M}_{\cal A}$. $|{\sf Con}_{{\cal M},\vec x,\vec y}|\!=\!1$ for ${\cal M}\in {\sf M}_{\cal A}^{\rm ND}\cup {\sf M}_{\cal A}^{\rm DND}$. $|{\sf Con}_{{\cal M},\vec x}|\!\geq\! 1$ for ${\cal M}\in {\sf M}_{\cal A}^{\rm NDB}$.

\vspace{0.05cm}
 }
}}
\end{overview}

{\bf The sets ${\sf Lab}_{{\cal M},n} ^{[\leq \omega]}$.} Let ${\sf Lab}_{{\cal M},n} ^{[\leq \omega]}$ be the set of all sequences $(\ell_t)_{t\geq 1}$ whose members $\ell_t$ can be extended to configurations $(\ell_t\,.\,\vec \nu^{(t)}\,.\,\bar u^{(t)})$ such that one of the resulting sequences $(\ell_t\,.\,\vec \nu^{(t)}\,.\,\bar u^{(t)})_{t\geq 1}$ belongs to ${\sf Con}_{{\cal M},n}$ (and thus to ${\sf Con}_{{\cal M},\bar \ell,n}$). Consequently, for each $\bar \ell \in {\sf Lab}_{{\cal M},n} ^{[\leq \omega]}$, $B_{\bar \ell}$ is a finite or an infinite computation path traversed by ${\cal M}$ for some inputs $\vec x$ of length $n\geq 1$. 

\begin{proposition}[Partial ${\cal M}$-configurations for fixed input length]\label{LabIndFixLen} \hfill Let ${\cal A}$ be a first-order structure, ${\cal M}$ be a BSS RAM in $ {\sf M}_{\cal A}\cup {\sf M}_{\cal A}^{\rm ND}\cup {\sf M}_{\cal A}^{\rm DND} \cup {\sf M}_{\cal A}^{\rm NDB}$, and $n\geq 1$. Then, for any sequence $(\ell_t\,.\,\vec \nu^{(t)}\,.\,\bar u^{(t)})_{t\geq 1} $ in $ {\sf Con}_{{\cal M},n}$, we have

\vspace{0.1cm}
\hspace{1cm} 
$\ell_1\hspace{0.67cm}=\quad 1$ \qquad and \qquad $\vec \nu^{(1)} \hspace{0.45cm} = \quad \,(n,1,\ldots,1) \in \bbbn_+^{k_{\cal M}}$ 
\vspace{0.1cm}

\noindent and, for all $t\geq 1$, 

\hspace{1cm} 
$\ell_{t+1}\quad =\quad \left\{\begin{array}{ll}\beta_{{\cal P}_{\cal M}}(\ell_t)& \hspace{0.47cm} \mbox {if \hspace{0.23cm} $\ell_t \in {\cal L}_{{\cal P}_{\cal M}} ^{\rm NB},$}\\
T_{\ell_t}(\vec \nu ^{(t)})& \hspace{0.47cm}\mbox {if \hspace{0.23cm} $\ell_t\in {\cal L}_{{\cal M},{{\rm H}_{\rm T}}}$},\\
\ell_t & \hspace{0.47cm}\mbox {if \hspace{0.23cm} $\ell_t\in {\cal L}_{{\cal M},{{\rm S}}}$,}\\
\end{array}\right.$

\vspace{0.2cm}
\hspace{1cm} $\ell_{t+1}\quad \in \quad\, \{\beta_{{\cal P}_{\cal M}}^+(\ell_t),\beta_{{\cal P}_{\cal M}}^- (\ell_t)\}\!$ \quad if \hspace{0.23cm} $\ell_t\in {\cal L}_{{\cal M},{\rm T}}\cup{\cal L}_{{\cal M},{\rm B}}$,

and 

\hspace{1cm} $\vec \nu^{(t+1)}\hspace{0.1cm}=\quad \left\{\begin{array}{ll}\vec \nu^{(t)}& \hspace{0.4cm} \mbox {if \hspace{0.23cm} $\ell_t \in {\cal L}_{{\cal M}} \setminus ({\cal L}_{{\cal M},{\rm H}_{1}}\cup {\cal L}_{{\cal M},{\rm H}_{+1}}),$}\\
 H_{\ell_t}(\vec \nu ^{(t)})& \hspace{0.4cm}\mbox {if \hspace{0.23cm} $\ell_t\in {\cal L}_{{\cal M},{\rm H}_{1}}\cup {\cal L}_{{\cal M},{\rm H}_{+1}}$}\\
\end{array}\right.$ 

\vspace{0.2cm}
\noindent where $\,T_{\ell_t},H_{\ell_t} \in {\cal F}_{\cal M}$ and $\beta_{{\cal P}_{\cal M}}$, $\beta_{{\cal P}_{\cal M}}^+$, and $\beta_{{\cal P}_{\cal M}}^-$ are the $\beta_{{\cal P}_{\cal M}}^\circ$-functions of $\mathfrak{Fl\,}_{{\cal P}_{\cal M}}^{\sf label}$.
\end{proposition}

{\bf Proof.} Let ${\cal M}\in {\sf M}_{\cal A}\cup {\sf M}_{\cal A}^{\rm ND}\cup {\sf M}_{\cal A}^{\rm DND} \cup {\sf M}_{\cal A}^{\rm NDB}$ and $n\geq 1$ be given.
 The first member of an $({\cal M},n)$-path is a suitable initial configuration starting with the tuple $(\ell_1\,.\,\vec \nu^{(1)})$  defined by $\ell_1=1$ and $\vec \nu^{(1)}=(n,1,\ldots,1)$. Now, let $t\geq 1$ and $(\ell_1,\ldots, \ell_{t+1})$ be an initial segment of a label path in ${\sf Lab}_{{\cal M},n}^{[\leq \omega]}$. Then,  $\ell_1,\ldots, \ell_{t+1}$ are the first components in the members of an initial segment  of at least one $({\cal M},n)$-path $(\ell_t\,.\,\vec \nu^{(t)}\,.\,\bar u^{(t)})_{t\geq 1}$ where $\bar u^{(1)}$ is determined by $(\vec x,(\ell_1\,.\,\vec \nu^{(1)}\,.\,\bar u^{(1)}))\in {\sf Input}_{\cal M}$ for an input $\vec x$ in $ U_{\cal A}^n$ \big[and additional guesses\big] and $(\ell_{t+1}\,.\,\vec \nu^{(t+1)}\,.\,\bar u^{(t+1)})$ results from applying $\to_{\cal M}$ to $(\ell_t\,.\,\vec \nu^{(t)}\,.\,\bar u^{(t)})$.
We consider two cases.

\noindent {\it 1)} For $\ell_t\!\in\! {\cal L}_{\cal M} \setminus {\cal L}_{{\cal M},{{\rm H}_{\rm T}}}$, the description of $\vec \nu^{(t+1)}$ and $\ell_{t+1}$ given above can be directly~-- and without any further restrictions~-- derived from $(\ell_t\,.\,\vec \nu^{(t)}\,.\,\bar u^{(t)})$. Whether $\ell_{t+1}\!=\!\beta_{{\cal P}_{\cal M}}^+(\ell_t)$ or $\ell_{t+1}\!=\!\beta_{{\cal P}_{\cal M}}^-(\ell_t)$ holds for $\ell_t\!\in\! {\cal L}_{{\cal M},{\rm T}}$ depends on $\bar u^{(t)}$.

\noindent {\it 2)} Let $\ell_t\in {\cal L}_{{\cal M},{{\rm H}_{\rm T}}}$. If there is no $i<t$ with $\ell_i\in {\cal L}_{{\cal M},{\rm H}_{1}}\cup {\cal L}_{{\cal M},{\rm H}_{+1}}$, then $\vec\nu^{(t)}=(n,1,\ldots,1)$ holds and $T_{\ell_t}(\vec \nu ^{(t)})$ is thus by definition (cf.\,\,fn.\,\,\ref{FootnTl} the only successor of $\ell_t$ in a computation path with the initial part $(\ell_1,\ldots,\ell_t)$ which means that $\ell_{t+1}=T_{\ell_t}(\vec \nu ^{(1)})$ holds. Now, let $r_0>0$ and $(\ell_{i_1},\ldots, \ell_{i_{r_0}})$ be the subsequence of $(\ell_1,\ldots,\ell_t)$ that belongs to $({\cal L}_{{\cal M},{\rm H}_{1}}\cup {\cal L}_{{\cal M},{\rm H}_{+1}})^{r_0}$ such that 

\vspace{0.1cm}
\quad $(\{\ell_1,\ldots,\ell_t\}\setminus \{\ell_{i_1},\ldots, \ell_{i_{r_0}}\})\,\cap\, ({\cal L}_{{\cal M},{\rm H}_{1}}\cup {\cal L}_{{\cal M},{\rm H}_{+1}})=\emptyset$,

\vspace{0.2cm}
\noindent $i_{r_0}<t$, and $i_1<\cdots <i_{r_0}$ \textcolor{gray}{(if $r_0>1$)} hold. For $i_1>1$, $(\ell_1, \ldots,\ell_{i_1-1})$ belongs to $ ({\cal L}_{\cal M} \setminus( {\cal L}_{{\cal M},{{\rm H}_1}}\cup {\cal L}_{{\cal M},{{\rm H}_{+1}}}))^{i_1-1}$. This implies $\vec \nu^{(i_1)}=(n,1,\ldots,1)$ in any case. Consequently,

\vspace{0.05cm}
$\begin{array}{llllll}\vec \nu ^{(i_{\alpha+1})} \!\!\!&=\cdots =  \vec \nu ^{(i_\alpha+1)} & =\,H_{\ell_{i_\alpha}}(\vec \nu ^{(i_\alpha)}) &\,\mbox{for all $\alpha\geq 1$ with $\alpha <r_0$ and}\\
\vec \nu ^{(t)} &=\cdots =  \vec \nu ^{(i_{r_0}+1)} \!\!\! & =\,H_{\ell_{i_{r_0}}}(\vec \nu ^{(i_{r_0})}) &\, \textcolor{gray}{(\mbox{where } r_0> 0)} \end{array}$

\vspace{0.05cm}
\noindent are uniquely determined by the initial part $(\ell_1,\ldots, \ell_t)$ and $\to_{\cal M}$, regardless of the specific values $x_1,\ldots,x_n$ of the input \big[and the guesses\big]. By applying $\to_{\cal M}$, we obtain $\vec \nu ^{(t+1)}=\vec \nu ^{(t)}$ and $\ell_{t+1}= T_{\ell_t}(\vec \nu ^{(t)})$.
\qed

\vspace{0.2cm}
{\bf Uniqueness of indices for fixed $n$ and $\bar \ell$.}
From the proof of Proposition \ref{LabIndFixLen}, it can be concluded that traversing any initial part $(\ell_1,\ldots, \ell_{t})$ of $\bar \ell \in {\sf Lab}_{{\cal M},n}^{[\leq \omega]}$ involves computing a sequence $(\vec \nu^{(1)}, \ldots, \vec \nu^{(t+1)})$ and that this leads to the equality $\ell_{t+1}= T_{\ell_t}(\vec \nu ^{(t)})$ such that we can use $\ell_{t+1}= T_{\ell_t}(\vec \nu ^{(t)})$ instead of the condition $\ell_{t+1}\in \{\beta_{\cal P}^+(\ell_t),\beta_{\cal P}^- (\ell_t)\}$ whenever $\ell_t\in {\cal L}_{{\cal M},{{\rm H}_{\rm T}}}$ holds. This results in a stronger statement.

\begin{corollary}[Indices uniquely determined for $B_0\in{\sf Lab}_{{\cal M},n}$]\label{CoroIndices}$\!$Let ${\cal M}$ be any BSS RAM in $ {\sf M}_{\cal A}\cup {\sf M}_{\cal A}^{\rm ND}\cup {\sf M}_{\cal A}^{\rm DND} \cup {\sf M}_{\cal A}^{\rm NDB}$ for an arbitrary $\sigma$-structure ${\cal A}$ and $n\geq 1$. Let $B_0$ be a computation path in ${\sf Lab}_{{\cal M},n}$ and $\bar \ell=B_0^{[\leq \omega]}$. Then, there is exactly one sequence $(\vec \nu^{(1)},\vec \nu^{(2)},\ldots)\in ( \bbbn_+^{k_{\cal M}})^\omega$ with $\vec \nu^{(1)}\!=(n,1,\ldots,1)\in \bbbn_+^{k_{\cal M}}$ such that for all inputs $\vec x \in U_{\cal A}^n$ \big[and all used tuples $\vec y$ of guesses in $U_{\cal A}^\infty$\big] traversing $B_0$, there is a sequence $(\bar u^{(t)})_{t\geq 1}\in (U_{\cal A}^\omega)^{\omega}$ such that the equations $con_t= (\ell_t\,.\,\vec \nu^{(t)}\,.\,\bar u^{(t)})$ hold for all members of $(con_t)_{t\geq 1} \in {\sf Con}_{{\cal M},\bar \ell,n}\cap {\sf Con}_{{\cal M},\vec x\big[,\vec y\big]}$. 
\end{corollary}

Proposition \ref{LabIndFixLen} states that, for any usual BSS RAM ${\cal M}$ over any structure ${\cal A}$ and each possible sequence $\bar \ell\in {\sf Lab}_{{\cal M},n} ^{[\leq \omega]}$ traversed by inputs in $U_{\cal A}^n$, the contents $\vec \nu^{(t)}\in \bbbn_+ ^{k_{\cal M }}$ of the $k_{\cal M}$ index registers are uniquely determined solely by the index instructions of the program ${\cal P}_{\cal M}$. This means that the first components $(\ell_t,\nu_{t,1},\ldots, \nu_{t,k_{\cal M}})$ of all generated configurations of ${\cal M}$ are completely defined by its program ${\cal P}_{\cal M}$ and the length of the inputs. On the one hand, the computation of $(\vec \nu^{(t)})_{t\geq 1}$ depends on the computation path $B_{\bar \ell}$ and, on the other hand, it is independent of the used structure ${\cal A}$. $\bar \ell$ and $n$ determine $(\vec \nu^{(t)})_{t\geq 1}$ uniquely. Thus, this statement can be generalized.

\begin{proposition}[Indices uniquely determined for all ${\cal M}$ by ${\cal P}_{{\cal M}_0}$]\label{Prop} \hfill
 Let \linebreak ${\cal P}$ be any $\sigma$-program whose instructions can be of the types $(1)$ to $(8)$ and $(11)$. Let $\bar \ell$ be any sequence $(\ell_1,\ell_2,\ldots)$ that belongs to ${\sf Lab}_{{\cal M}_0,n}^{[\leq \omega]}$ for some $n\geq 1$ and some BSS RAM ${\cal M}_0$ whose program ${\cal P}_{{\cal M}_0}$ is ${\cal P}$. Then, there is a sequence $(\vec \nu^{(t)})_{t\geq 1}$ in $( \bbbn_+^{k_{{\cal M}_0}})^\omega$ such that, for each BSS RAM ${\cal M}$ over any $\sigma$-structure ${\cal A}$ with ${\cal P}_{\cal M}={\cal P}$ and $k_{\cal M}= k_{{\cal M}_0} $, for each $\vec x \in U_{\cal A}^n$ with $\bar \ell\in {\sf Lab}_{{\cal M},\vec x} ^{[\leq \omega]}$, and for all resulting sequences $(\ell_t\,.\,\vec \mu^{(t)}\,.\,\bar u^{(t)})_{t\geq 1}\in {\sf Con}_{{\cal M},\vec x}$ with $(\bar u^{(t)})_{t\geq 1}\in U_{\cal A}^\omega$, the sequences $(\vec \mu^{(t)})_{t\geq 1}$ and $(\vec \nu^{(t)})_{t\geq 1}$ match.
\end{proposition}

Now, we consider all ${\cal P}$-paths and we do not assume that they have to be computation paths traversed by a machine. For all $t\geq 1$, all $(\ell_1,\ldots,\ell_t)\in {\cal L}_{\cal P}^t$, all $n\geq 1$, the label $\ell_{t+1}$ must be one and the same element from $ \{\beta_{\cal P}^+(\ell_t),\beta_{\cal P}^- (\ell_t)\}$ whenever $\ell_t\in {\cal L}_{{\cal P},{\rm H}_{\rm T}}$ holds and $(\ell_1,\ldots,\ell_{t+1})$ is an initial part of some $\bar \ell\in {\sf Lab}_{{\cal M},n}$ for any arbitrary BSS RAM ${\cal M}$ with ${\cal P}_{\cal M}={\cal P}$ and thus $k_{\cal M}\geq k_{\cal P}$. In such a case, we can assume $\beta_{\cal P}^+(\ell_t)\not=\beta_{\cal P}^- (\ell_t)$ (cf.\,\,$({\sf FL4})$ in Definition \ref{DefFlowLabel}). Consequently, ${\cal L}_{{\cal P},{\rm H}_{\rm T}}\not=\emptyset$ implies that there are ${\cal P}$-paths that cannot be computation paths for inputs of length $n$ and all program paths whose first components are $(\ell_1,\ldots,\ell_t,\ell)$ with $\ell \in \{\beta_{\cal P}^+(\ell_t),\beta_{\cal P}^- (\ell_t)\}\setminus \{\ell_{t+1}\}$ belong to ${\sf Lab}_{{\cal P},n}^{\rm Co}$. Only the reachability of $\ell_{t+1} \in \{\beta_{\cal P}^+(\ell_t),\beta_{\cal P}^- (\ell_t)\}$ for $\ell_t\in {\cal L}_{{\cal P},{\rm T}}$ depends on a $\sigma$-structure ${\cal A}$ and the contents of $Z$-registers. For $\ell_t\in{\cal L}_{{\cal P},{\rm B}}$, a branching is in any case possible. Let us consider all paths in ${\sf Con}_{{\cal M},n}= \bigcup_{\vec x \in U_{\cal A}^n}{\sf Con}_{{\cal M},\vec x}$. 

\begin{theorem}[Indices uniquely determined by ${\cal P}$]\label{TheoUniq1}
Let ${\cal P}$ be a $\sigma$-program and let $ \kappa \geq k_{\cal P}$, $n\geq 1$, and $\bar \ell\in {\sf Lab}_{\cal P} ^{[\leq \omega]}$ be given. There is one and only one sequence $(\vec \nu^{(t)})_{t\geq 1}\in (\bbbn_+^{\kappa} )^\omega$ such that 

\vspace{0.1cm}
$\!\!\{(\vec \mu^{(t)})_{t\geq 1}\in (\bbbn_+^{\kappa} )^\omega\mid 
(\exists {\cal A}\in {\sf Struc})(\exists {\cal M}\in {\sf M}_{\cal A} \cup {\sf M}_{\cal A}^{\rm DND} \cup{\sf M}_{\cal A}^{\rm ND} \cup {\sf M}_{\cal A}^{\rm NDB})$

\hspace{3.2cm} $(\kappa= k_{\cal M}\,\, \&\,\, {\cal P}={\cal P}_{\cal M}\,\,$

\vspace{0.1cm}
\hfill$ \&\,\, (\exists (\bar u^{(t)})_{t\geq 1}\in ({U_{\cal A}^\omega})^\omega)((\ell_t\,.\,\vec \mu^{(t)}\,.\, \bar u^{(t)} )_{t\geq 1}\in{\sf Con}_{{\cal M},n}))\}$

$\subseteq \{(\vec \nu^{(t)})_{t\geq 1}\}$.
\end{theorem} 

We will give the details for a proof in Section \ref{SectTransSyst} where the sequence of partial ${\cal M}$-configurations, $(\ell_t\,.\,\vec \nu ^{(t)})_{t\geq 1} $, determined by $n$ and $\bar \ell$ is denoted by ${\rm con}^{\rm ind}_{ {\cal M},\bar \ell, n}$ and partial configurations are considered for any $\sigma$-program.

\vspace{0.1cm}
{\bf A partial $({\cal P},\kappa)$-configuration}.
For any ${\cal P}\in {\sf P}_{\sigma}$, $\ell\in {\cal L}_{\cal P}$, and $\vec \nu \in \bbbn_+^{\kappa}$, the tuple $(\ell\,.\,\vec\nu)$ is called a {\em partial $({\cal P},\kappa)$-configuration}.

\subsection{Partial configurations determined by programs}\label{SectTransSyst} 
Inspired by the observations in Section \ref{PropertRestrConf}, we want to provide a general framework for describing the transformation of partial configurations in ${\sf S}_{\cal M}^{\rm ind}$ for arbitrary BSS RAMs ${\cal M}$ formally in accordance with the generation of ${\cal M}$-configurations by $\to_{\cal M}$. We consider sets  ${\sf S}_{{\cal P},\kappa}^{\rm ind}$ of partial configurations for any $\sigma$-program ${\cal P}$ and introduce new transition functions $\to_{{\cal P},\kappa}^{\rm ind}: {\sf S}_{{\cal P},\kappa}^{\rm ind} \genfrac{}{}{0pt}{2}{\longrightarrow} {\longrightarrow} {\sf S}_{{\cal P},\kappa}^{\rm ind}$ for transforming partial configurations in ${\sf S}_{\cal M}^{\rm ind}$ and other sets $ {\sf S}_{{\cal P},\kappa}^{\rm ind}$. This introduction leads, in particular, to ${\sf S}_{\cal M}^{\rm ind}={\sf S}_{{\cal P}_{\cal M},k_{\cal M}}^{\rm ind} $ and $\to_{{\cal P}_{\cal M},k_{\cal M}}^{\rm ind}\subset ({\sf S}_{\cal M}^{\rm ind})^2 $ and enables thus a promising step toward an abstract characterization of the transformation of configurations by $\to_{\cal M}$. For $\ell\in {\cal L}_{{\cal M},{\rm H}_1}\cup {\cal L}_{{\cal M},{\rm H}_{+1}} $, we also denote $H_\ell\in {\cal F}_{\cal M}$ by $H_{{\cal M}, \ell}$ and, for $\ell\in {\cal L}_{{\cal M},{\rm H}_{\rm T}}$, we also denote $T_\ell \in {\cal F}_{\cal M}$ by $T_{{\cal M},\ell} $. For generating sequences of partial $({\cal P},\kappa)$-configurations, we use the following functions. 

\begin{definition}[Test and other auxiliary functions in $\,{\cal F}_{{\cal P},\kappa}^{\rm ind}$]\label{AuxilFP}\hfill Let $\,{\cal P}$ be \linebreak any $\sigma$-program, $\kappa\geq k_{\cal P}$, $n\geq 1$, and ${\cal M}$ be an arbitrary BSS RAM with ${\cal P}_{\cal M}={\cal P}$ and $ k _{\cal M}=\kappa$. 
For $\ell\in {\cal L}_{{\cal P},{\rm H}_1}\cup {\cal L}_{{\cal P},{\rm H}_{+1}} $, let \textcolor{magenta}{$H_{{\cal P},\kappa,\ell}$} be the {\em auxiliary function} given by $H_{{\cal P},\kappa,\ell}=H_{{\cal M}, \ell}$. For $\ell\in {\cal L}_{{\cal P},{\rm H}_{\rm T}}$, let \textcolor{magenta}{$T_{{\cal P},\kappa,\ell}$} be the {\em test function} given by $T_{{\cal P},\kappa,\ell}=T_{{\cal M}, \ell}$. These functions $H_{{\cal P},\kappa,\ell}$ and $T_{{\cal P},\kappa,\ell}$ now also belong to the {\em system ${\cal F}_{{\cal P},\kappa}^{\rm ind}$}. For any fixed ${\cal P}\in {\sf P}_\sigma$, we also denote $H_{{\cal P},\kappa,\ell}$ and $T_{{\cal P},\kappa,\ell}$ by \textcolor{magenta}{$H_\ell^{(\kappa)}$} and \textcolor{magenta}{$T_\ell^{(\kappa)}$}, respectively. For details, see Overviews \ref{RestrTransf} and \ref{TransPKappaPaths}.
\end{definition} 

Consequently, the functions in ${\cal F}_{{\cal P},\kappa}^{\rm ind}$ allow to transform and evaluate partial $({\cal P},\kappa)$-configurations independently of the content of any $Z$-register.

\begin{definition}[$({\cal P}\!,\kappa,n)$-paths in $ {\sf Con}^{\rm ind}_{{\cal P},\kappa,n}$]$\!$Let ${\cal P}\!\in\! {\sf P}_\sigma$, $\kappa\geq k_{\cal P}$, and $n\geq 1$. The sequence $(\ell_t\,.\, \vec \nu^{(t)})_{t\geq 1}$ of  partial $({\cal P},\kappa)$-configurations is called {\em a $({\cal P},\kappa,n)$-path} if {\sf (P1)} to {\sf (P5)} hold. Let $ {\sf Con}^{\rm ind}_{{\cal P},\kappa,n}$ be the set of all $({\cal P},\kappa,n)$-paths.
\end{definition}

\begin{overview}[Properties of the $({\cal P},\kappa,n)$-paths $(\ell_t\,.\, \vec \nu^{(t)})_{t\geq 1}$]\label{PropPartCon}

\hfill

\nopagebreak 

\noindent \fbox{\parbox{11.8cm}{

\hfill{\small Let ${\cal P}\in {\sf P}_{\sigma}$, $\kappa\geq k_{\cal P}$, $n\geq 1$, and $\bar \ell=(\ell_1,\ell_2,\ldots)$.}

\em
\begin{tabular}{ll}
{\sf (P1)} & $\bar \ell \in {\sf Lab}_{\cal P}^{[\leq \omega]}$.\\
{\sf (P2)} & $ \vec \nu^{(1)}\in \bbbn_+ ^{\kappa}$,\\
{\sf (P3)} & $\vec \nu^{(t+1)}\hspace{0.1cm}=\quad \left\{\begin{array}{ll}\vec \nu^{(t)}& 
\!\mbox {if \,$\ell_t \in {\cal L}_{{\cal P}} \setminus ({\cal L}_{{\cal P},{\rm H}_{1}}\cup {\cal L}_{{\cal P},{\rm H}_{+1}})$, }\\
H_{{\cal P},\kappa,\ell_t}(\vec \nu ^{(t)})&\!\mbox {if \,$\ell_t\in {\cal L}_{{\cal P},{\rm H}_{1}}\cup {\cal L}_{{\cal P},{\rm H}_{+1}}$}\\
\end{array}\right.$\vspace{0.1cm}\\
&\hspace{8.25cm} for all $t\geq 1$.\\
{\sf (P4)} & $(\ell_1\,.\,\vec \nu^{(1)})=(1,n,1,\ldots,1) \in \bbbn_+^{\kappa +1}$.\\
{\sf (P5)} & $\ell_{t+1}=T_{{\cal P},\kappa,\ell_t}(\vec \nu ^{(t)})$\hspace{1.45cm} if \,$t\geq 1$ and $\ell_t\in {\cal L}_{{\cal P},{{\rm H}_{\rm T}}}$.\vspace{0.2cm}\\
\end{tabular}

\hfill {\small \em {\sf P} stands for {\sf partial}.}
}}
\end{overview}

By {\sf (P1)}, we have one of three cases for $\bar \ell$: 
{\it 1) $s=1$, {\sf (L5)}, and {\sf (L6)}}, 
{\it 2) $s\geq 2$, {\sf (L1)}, {\sf (L6)}, {\sf (L3)}, and {\sf (L4)} for all $t<s$}, or
{\it 3) {\sf (L2)}, {\sf (L3)}, and {\sf (L4)} for all $t\geq 1$}. For {\sf (L1)} to {\sf (L6)}, see Overview \ref{PropProgrPaths}. {\sf (P2)} and {\sf (P3)} together reflect the potential for a recursive definition. The initial values given by {\sf (P4)} indicate that we want to consider $({\cal P},\kappa,n)$-paths {\sf only for an analysis of BSS RAMs}. The justification to consider {\sf (P5)} follows from Proposition \ref{LabIndFixLen}.

By the proof of Proposition \ref{LabIndFixLen} and in accordance with Corollary \ref{CoroIndices}, there is, for each $\bar \ell\in {\sf Lab}_{{\cal M},n} ^{[\leq \omega]}$, exactly one sequence $(\ell_t\,.\,\vec \nu ^{(t)})_{t\geq 1} $ of partial configurations in $ {\sf Con}_{{\cal P}_{\cal M},k_{\cal M},n}^{\rm ind}$ which is of interest to us. We denote it by ${\rm con}^{\rm ind}_{ {\cal M},\bar \ell, n}$ such that ${\rm con}^{\rm ind}_{ {\cal M},\bar \ell, n}\in {\sf Con}_{{\cal P}_{\cal M},k_{\cal M},n}^{\rm ind}$ holds. In this context, we will also consider the sets listed in Overview \ref{SetsPartCon2}. Taking into account all the details of Theorem \ref{TheoUniq1}, we will introduce the transition functions $\to _{{\cal P},\kappa}^{\rm ind}$ for the partial $({\cal P},\kappa)$-configurations in sets ${\sf S}_{{\cal P},\kappa}^{\rm ind}$~-- that can also be denoted by $\genfrac{}{}{0pt}{2}{\longrightarrow} {\longrightarrow}_{{\cal P},\kappa}^{\rm ind}$~-- in order to generate sequences of partial $({{\cal P},\kappa})$-configurations in ${\sf Con}^{\rm ind}_{{\cal P},\kappa,n}$ in a non-deterministic way.

\vspace{0.3cm}
{\bf The sets from $ {\sf Seq}^{\rm ind}_{{\cal P},\kappa}$ to $ {\sf Con}^{\rm ind}_{{\cal P},\kappa,n}$.} Let ${\sf Seq}^{\rm ind}_{{\cal P},\kappa}$ be a subset of the set $\{(\ell_t\,.\, \vec \nu^{(t)})_{t\geq 1}\mid (\ell_t)_{t\geq 1}\in {\sf Lab}_{\cal P} ^{[\leq \omega]}\,\,\&\,\, (\vec \nu^{(t)})_{t\geq 1} \in (\bbbn^{\kappa}_+)^\omega \}$ and consist of all sequences $(\ell_t\,.\, \vec \nu^{(t)})_{t\geq 1}$ of $({\cal P,\kappa})$-configurations $(\ell_t,\nu_{t,1},\ldots, \nu_{t,\kappa})$ determined by a sequence $\bar \ell\in {\sf Lab}_{\cal P} ^{[\leq \omega]}$ and the functions $H^{(\kappa)} _\ell$~-- that belong to ${\cal F}_{{\cal P},\kappa}^{\rm ind}$ for all $\ell\in {\cal L}_{{\cal P},{\rm H}_{1}}\cup {\cal L}_{{\cal P},{\rm H}_{+1}}$~-- such that all components $\nu_{t+1,1},\ldots, \nu_{t+1,\kappa}$ of all tuples $\vec \nu^{(t+1)}\in \bbbn_+ ^{\kappa}$ satisfy the properties given by {\sf (P2)} and {\sf (P3)}. For any $n\geq 1$, let $ {\sf Seq}^{\rm ind}_{{\cal P},\kappa,n}$ be the set of all sequences $(\ell_t\,.\, \vec \nu^{(t)})_{t\geq 1} \in {\sf Seq}^{\rm ind}_{{\cal P},\kappa}$ whose initial values satisfy {\sf (P4)}. Thus, these sets satisfy ${\sf Con}^{\rm ind}_{{\cal P},\kappa,n} \subseteq {\sf Seq}^{\rm ind}_{{\cal P},\kappa,n}\subset {\sf Seq}^{\rm ind}_{{\cal P},\kappa} \subset ({\sf S}_{{\cal P},\kappa}^{\rm ind})^\omega$.

\vspace{0.1cm}
\begin{overview}[Sets of sequences of partial $({\cal P},\kappa)$-configurations]\label{SetsPartCon2}

\hfill

\nopagebreak 

\noindent \fbox{\parbox{11.8cm}{

\hfill{\small Let ${\cal P}\in {\sf P}_{\sigma}$, $\kappa\geq k_{\cal P}$, $n\geq 1$, ${\cal M}$ be a BSS RAM, and 

\hfill $\bar \ell=(\ell_1,\ell_2,\ldots)\in {\sf Lab}_{\cal P}^{[\leq \omega]}$.}

\vspace{0.2cm}

\em
\quad\,
\begin{tabular}{ll}
${\sf Seq}^{\rm ind}_{{\cal P},\kappa}$ &is the set of all $(\ell_t\,.\,\vec \nu ^{(t)})_{t\geq 1} $ determined by {\sf (P1)} to {\sf (P3)},\\
${\sf Seq}^{\rm ind}_{{\cal P},\kappa,n}\!$&is the set of all $(\ell_t\,.\,\vec \nu ^{(t)})_{t\geq 1} $ determined by {\sf (P1)} to {\sf (P4)},\\
${\sf Con}^{\rm ind}_{{\cal P},\kappa,n}\!\!\!$&is the set of all $(\ell_t\,.\,\vec \nu ^{(t)})_{t\geq 1} $ determined by {\sf (P1)} to {\sf (P5)}.\vspace{0.1cm}\\\end{tabular}

\vspace{0.2cm}
{\em Consequently, we will use:}

\vspace{0.2cm}
$\!$\begin{tabular}{ll}
{\it 1)} $ {\sf Con}^{\rm ind}_{{\cal P},\kappa,n}\!\!\!$ & is the set of all $(\ell_t\,.\,\vec \nu ^{(t)})_{t\geq 1} $ in $ {\sf Seq}^{\rm ind}_{{\cal P},\kappa,n}$ satisfying\\
&\,\,\, $(\ell_t\,.\, \vec \nu^{(t)})\genfrac{}{}{0pt}{2}{\longrightarrow} {\longrightarrow}_{{\cal P},\kappa}^{\rm ind}(\ell_{t+1}\,.\, \vec \nu^{(t+1)})$\quad for $t\geq 1$\quad\,\,\, \textcolor{gray}{\small\em (see Def.\,\,\ref{TransRestrConf})},\vspace{0.2cm}\\
{\it 2)} ${\sf Con}_{{\cal M},\bar \ell,n}\!\!\!$&contains only $({\cal M},n)$-paths $(\ell_t\,.\,\vec \nu ^{(t)}\,.\,\bar u^{(t)})_{t\geq 1}$ with \\
&\,\,\, $(\ell_t\,.\,\vec \nu ^{(t)})_{t\geq 1}={\rm con}^{\rm ind}_{ {\cal M},\bar \ell, n} \in {\sf Con}_{{\cal P}_{\cal M},k_{\cal M},n}^{\rm ind}\!$\hfill \textcolor{gray}{\small \em (if ${\sf Con}_{{\cal M},\bar \ell,n}\not = \emptyset$)},\vspace{0.2cm}\\
{\it 3)} $ {\sf Con}^{\rm ind}_{{\cal P}_{\cal M},k_{\cal M},n}\!\!\!$ &is the set of all $({\cal P}_{\cal M},k_{\cal M},n)$-paths 
\\&\,\,\, that can be useful for evaluating the behavior of ${\cal M}$.\\ 
\end{tabular}}}
\end{overview}

\vspace{0.1cm}
\begin{example}[${\cal P}$- and $({\cal P},\kappa,n)$-paths]\label{ExePKappaN2} Let the following program ${\cal P}$ be the $(3;1,1;2)$-program ${\cal P}_{{\cal M}_1}$ of the machine ${\cal M}_1$ considered in \cite[pp.\,\,14, 40]{GASS25C}. Consequently, we can use $\kappa\geq 2$.

\vspace{0.1cm}
\noindent {\sf 
\begin{tabular}{rll}${\cal P}$:
&$1: I_2:=I_2+1;$ \,\, 
$2: $ if $I_1=I_2$ then goto $3$ else goto $2;$ \\
& $3: Z_5:= c_3^0;$ \\
&$4 : Z_3:=f_1^1(Z_1);$ \,\, $5 : Z_4:=f_1^1(Z_2); $ \\
&$6 : Z_1:=f_2^1(Z_1);$ \,\, $7 : Z_2:=f_2^1(Z_2); $ \\
&$8 : $ if $r_1^{2}(Z_3,Z_4)$ then goto $9$ else goto $8;$ \\
&$9 : $ if $r_1^{2}(Z_4,Z_5)$ then goto $10$ else goto $4;$ \\&$\!\!\!10: $ stop. \\
\end{tabular}}

\vspace{0.2cm}
\noindent A finite ${\cal P}$-path: $B_1=(1,2,3,4,5,6,7,8,9, 4, 5,6,7,8,9, 4,5,6,7,8,9,\textcolor{blue}{10})$.

\vspace{0.15cm}
\noindent An infinite ${\cal P}$-path: $B_2: (1,2,3,4,5,6,7,8,9, 4,5,6,7,8,9, 4,5,6,7,8,8,8,\ldots)$. 

\vspace{0.15cm}
\noindent 
Some $({\cal P},\kappa,n)$-paths in ${\sf Con}_{{\cal P},\kappa,n}^{\rm ind}$ for $\kappa\in\{2,3\}$ and $n\in \{2,5\}$.

\vspace{0.1cm}

\noindent A $({\cal P},2,2)$-path: $((1\,.\, (2,1)),(2\,.\, (2,2)),(3\,.\, (2,2)),\ldots,(9\,.\, (2,2) ),\ldots,(8\,.\, (2,2) ),$

$ (9\,.\, (2,2) ), (4\,.\, (2,2) ),\ldots,(\textcolor{blue}{8}\,.\, (2,2)),(\textcolor{blue}{8}\,.\, (2,2)), \ldots)$ {\small \textcolor{gray}{ $\in {\sf Con}_{{\cal M}_1,B_2,2}^{\rm ind}$}},

\vspace{0.05cm}
another notation: $((1,2,1),(2,2,2),(3,2,2),\ldots,(9,2,2),\ldots,(8,2,2), $

$(9,2,2), (4,2,2),\ldots,(\textcolor{blue}{8},2,2),(\textcolor{blue}{8},2,2), \ldots)$.

\vspace{0.15cm}
\noindent A $({\cal P},3,2)$-path: $((1,2,1,1),(2,2,2,1),(3,2,2,1),\ldots,(9,2,2,1),\ldots,(8,2,2,1), $

$(9,2,2,1), (4,2,2,1),\ldots,(\textcolor{blue}{8},2,2,1),(\textcolor{blue}{8},2,2,1), \ldots)$

\vspace{0.15cm}
\noindent A $({\cal P},3,5)$-path: $((1,5,1,1),(2,5,2,1),(\textcolor{blue}{2},5,2,1), (\textcolor{blue}{2},5,2,1), \ldots)$.
\end{example}

\begin{example}[Sequences in ${\sf Lab}_{\cal P}^{[\leq \omega]}\!$ and ${\sf Seq}^{\rm ind}_{{\cal P},\kappa}$]\label{ExePKappaN1}
Let ${\cal P}$ be the program with $k_{\cal P}=2$ considered in Fig. \ref{BerWurzeln1} and in Example \ref{ExBerWurzeln1} and $\kappa\geq 2$.

\vspace{0.15cm}
\noindent Some sequences in ${\sf Lab}_{\cal P}^{[\leq \omega]}$:
\vspace{0.05cm}

\begin{tabular}{ll}
$B_0^{[\leq \omega]} $&$\!\!\!=B_0=(1,2,3,4,5,6,5,6,5,6,5,6,\ldots)$,\\
$B_1^{[\leq \omega]} $&$\!\!\!=(1,2,3,4,5,6,7,8,8,8, \ldots)$,\\
$(B_0')^{[\leq \omega]} $&$\!\!\!=B_0'=(1,1,1,\ldots)$,\\
$(B_1')^{[\leq \omega]} $&$\!\!\!=(1,1,2,3,4,5,6,7,8,8,8, \ldots)$.
\end{tabular}

\vspace{0.15cm}
\noindent $({\cal P},2,1)$-paths: $((1,1,1),\ldots, ( 4,1,1),( 5,1,1),( 6,1,1),( 5,1,1),( 6,1,1),\ldots)$,

$((1,1,1),( 2,1,1),\ldots,( 7,1,1),( 8,1,1),( 8,1,1),\ldots)$.

\vspace{0.15cm}
\noindent A $({\cal P},3,2)$-path: $((1,2,1,1),(1,2,1,1),(1,2,1,1),\ldots)$.

\vspace{0.15cm}
\noindent A $({\cal P},2,3)$-path: $((1,3,1),(1,3,1),(1,3,1),\ldots)$.

\vspace{0.15cm}
\noindent A sequence in ${\sf Seq}^{\rm ind}_{{\cal P},2,5}$: $((1,5,1),(1,5,1),(2,5,1),(3,5,1),(4,5,1),(5,5,1),$

$(6,5,1),(7,5,1),(8,5,1),(8,5,1),(8,5,1), \ldots)$.

\vspace{0.15cm}
\noindent A sequence in ${\sf Seq}^{\rm ind}_{{\cal P},3}$: $((1,7,6,2), (1,7,6,2), (1,7,6,2), \ldots)$.
\end{example}

Based on our observations, we define the transition functions $\to_{{\cal P},\kappa} ^{\rm ind}$ and the disjoint union $\to_{\cal P} ^{\rm ind}$ of all these functions $\to_{{\cal P},\kappa} ^{\rm ind}$ with $\kappa\geq k_{\cal P}$.

\begin{definition}[The transition systems ${\cal S}_{{\cal P},\kappa}^{\rm ind}$ and ${\cal S}_{{\cal P}}^{\rm ind}$]\label{TransRestrConf}
Let ${\cal P}\in {\sf P}_{\sigma}$ and $\kappa\geq k_{\cal P}$. \mbox{$({\sf S}_{{\cal P},\kappa}^{\rm ind},\to_{{\cal P},\kappa}^{\rm ind})$} is a {\em transition system} determined by ${\cal P}$ and $\kappa$ and denoted by $\,{\cal S}_{{\cal P},\kappa}^{\rm ind}\,$ if $\,{\sf S}_{{\cal P},\kappa}^{\rm ind}\,$ is the set of all {\em partial $({\cal P},\kappa)$-configurations} given by 

\vspace{0.3cm}
\quad ${\sf S}_{{\cal P},\kappa}^{\rm ind}= \{ (\ell\,.\,\vec\nu)\mid \ell\in {\cal L}_{\cal P} \,\,\,\&\,\,\, \vec \nu \in \bbbn_+^{\kappa} \}$

\vspace{0.3cm}
\noindent and $\textcolor{red}{\to_{{\cal P},\kappa}^{\rm ind}}: {\sf S}_{{\cal P},\kappa}^{\rm ind} \textcolor{red}{\genfrac{}{}{0pt}{2}{\longrightarrow} {\longrightarrow}} {\sf S}_{{\cal P},\kappa}^{\rm ind} $ is a total multi-valued  {\em transition function} that is defined on ${\sf S}_{{\cal P},\kappa}^{\rm ind}$ by

\vspace{0.3cm}
$\begin{array}{ll}
\textcolor{red}{\to_{{\cal P},\kappa}^{\rm ind}} \,= \hspace{0.23cm}\{((\ell\,.\,\vec \nu), (\textcolor{blue}{\ell +1}\,.\,\vec \nu))&\!\!\in ({\sf S}_{{\cal P},\kappa}^{\rm ind})^2\mid \ell \in{\cal L}_{{\cal P}, { \rm F} } \cup {\cal L}_{{\cal P}, { \rm F}_0 } \cup {\cal L}_{{\cal P}, { \rm C} } \}\\
\hspace{1.3cm}\cup \,\{((\ell\,.\,\vec \nu), (\textcolor{blue}{\beta_{\cal P}^+(\ell)}\,.\,\vec \nu))&\!\!\in ({\sf S}_{{\cal P},\kappa}^{\rm ind})^2\mid \ell\in {\cal L}_{{\cal P},{\rm T}}\cup {\cal L}_{{\cal P},{\rm B}}\}\\
\hspace{1.3cm}\cup \,\{((\ell\,.\,\vec \nu), (\textcolor{blue}{\beta_{\cal P}^-(\ell)}\,.\,\vec \nu))&\!\!\in ({\sf S}_{{\cal P},\kappa}^{\rm ind})^2\mid \ell\in {\cal L}_{{\cal P},{\rm T}}\cup {\cal L}_{{\cal P},{\rm B}}\}\\
\hspace{1.3cm}\cup \,\{((\ell\,.\,\vec \nu), (\textcolor{blue}{T_{\ell}^{(\kappa)}(\vec \nu)}\,.\,\vec \nu))&\!\!\in ({\sf S}_{{\cal P},\kappa}^{\rm ind})^2\mid \ell\in {\cal L}_{{\cal P},{\rm H}_{\rm T}} \}\\
 \hspace{1.3cm}\cup \,\{((\ell\,.\,\vec \nu), (\textcolor{blue}{\ell +1}\,.\,\textcolor{blue}{H_{\ell}^{(\kappa)}(\vec \nu)}))\!\!&\!\!\in ({\sf S}_{{\cal P},\kappa}^{\rm ind})^2\mid \ell\in {\cal L}_{{\cal P},{\rm H}_1}\cup \, {\cal L}_{{\cal P},{\rm H}_{+1}} \}\\
 \hspace{1.3cm}\cup \, \{((\ell\,.\,\vec \nu), (\ell\,.\,\vec \nu))&\!\!\in ({\sf S}_{{\cal P},\kappa}^{\rm ind})^2 \mid \ell=\ell_{\cal P}\}.\\\end{array}$

\vspace{0.3cm}
\noindent $T_{\ell}^{(\kappa)}$ and $H_{\ell}^{(\kappa)}$ belong to $ {\cal F}_{{\cal P},\kappa}^{\rm ind}$, as determined in Definition \ref{AuxilFP}. Moreover, $({\sf S}_{\cal P}^{\rm ind},\to_{\cal P}^{\rm ind})$ is the {\em transition system ${\cal S}_{\cal P}^{\rm ind}$} that is  defined by ${\sf S}_{\cal P}^{\rm ind}=\bigcup\nolimits_{ \kappa \geq k_{\cal P}} {\sf S}_{{\cal P},\kappa}^{\rm ind} $ and $\textcolor{red}{\to_{\cal P}^{\rm ind}}\,=\, \bigcup\nolimits_{ \kappa \geq k_{\cal P}} (\to_{{\cal P},\kappa}^{\rm ind})$. 
\end{definition}

${\sf S}_{\cal P}^{\rm ind}$ consists of all possible {\em partial ${\cal P}$-configurations $(\ell\,.\,\vec \nu)$} whose first component belongs to ${\cal L}_{\cal P}$ and whose length $\kappa+1$ is bounded from below by $ k_{\cal P}+1$ and $\to_{\cal P}^{\rm ind}$ is a transition function on ${\sf S}_{\cal P}^{\rm ind}$. 

Each function $\to_{{\cal P},\kappa} ^{\rm ind}$ is a total multi-valued function that enables to generate sequences of partial $({\cal P},\kappa)$-configurations such that the potential branching behavior of all BSS RAMs caused by the execution of ${\rm T}$-, ${\rm H}_{\rm T}$-, or ${\rm B}$-instructions of ${\cal P}$ can be fully captured in the broadest sense and modeled at a new abstraction level. Consequently, $\to_{\cal P}^{\rm ind}$ is also a multi-valued function. To express this, we can also use the symbols \textcolor{red}{$\genfrac{}{}{0pt}{2}{\longrightarrow} {\longrightarrow}_{{\cal P},\kappa} ^{\rm ind}$} and \textcolor{red}{$\genfrac{}{}{0pt}{2}{\longrightarrow} {\longrightarrow}_{\cal P} ^{\rm ind}$}. 

By repeatedly applying $\to_{\cal P}^{\rm ind}$ to $(1,n,1,\ldots,1)$, we can recursively generate sequences of partial ${\cal P}$-configurations in $\bigcup_{\kappa\geq k_{\cal P}}\bigcup_{n\geq 1} {\sf Con}^{\rm ind}_{{\cal P},\kappa,n}$ where each of these sequences is a $({\cal P},\kappa,n)$-path for some $\kappa \!\geq k_{\cal P}$ and $n\!\geq 1$.

For any partial configurations $(\ell\,.\,\vec\nu)$ and $(\ell'\,.\,\vec\mu)$ in ${\sf S}_{{\cal P},\kappa}^ {\rm ind}$, $(\ell\,.\,\vec\nu)\to_{{\cal P},\kappa}^{\rm ind} (\ell'\,.\,\vec\mu)$ is the infix notation for expressing $((\ell\,.\,\vec\nu), (\ell'\,.\,\vec\mu))\, \in \, \to_{{\cal P},\kappa}^{\rm ind}$ which means

\vspace{0.2cm}
\quad $\to_{{\cal P},\kappa}^{\rm ind}=\{ ((\ell\,.\,\vec\nu),(\ell'\,.\,\vec\mu))\in ({\sf S}_{{\cal P},\kappa}^{\rm ind})^2 \mid (\ell\,.\,\vec\nu)\to_{{\cal P},\kappa} ^{\rm ind}(\ell'\,.\,\vec\mu)\}$. 

\vspace{0.3cm}
{\bf Sequences ${\rm seq}^{\rm ind}_{ {\cal P},\bar \ell, \kappa, n}\!$ and $({\cal P},\kappa,n)$-paths ${\rm con}^{\rm ind}_{ {\cal P},\bar \ell, \kappa, n}$.} For ${\cal P}\!\in {\sf P}_{\sigma}$, $ \kappa \!\geq k_{\cal P}$, $n\!\geq 1$, and $\bar \ell\in {\sf Lab}_{\cal P} ^{[\leq \omega]}$, there is~-- by ${\sf (P1)}$ to ${\sf (P4)}$~-- only one sequence $(\vec \nu^{(t)})_{t\geq 1}\in (\bbbn_+^{\kappa})^\omega$ such that the sequence $(\ell_t\,.\,\vec \nu^{(t)})_{t\geq 1}$ of partial $({\cal P},\kappa)$-configurations\linebreak  $(\ell_t\,.\,\vec \nu^{(t)})$  belongs to ${\sf Seq}^{\rm ind}_{{\cal P},\kappa,n}$. We denote it  by \textcolor{blue}{${\rm seq}^{\rm ind}_{ {\cal P},\bar \ell, \kappa, n}$}. If ${\rm seq}^{\rm ind}_{ {\cal P},\bar \ell, \kappa, n}$ is a $({\cal P},\kappa,n)$-path, then we also denote it  by \textcolor{blue}{${\rm con}^{\rm ind}_{ {\cal P},\bar \ell, \kappa, n}$}.

\vspace{0.1cm}
{\bf $({\cal P}_{\cal M},k_{\cal M},n)$-paths ${\rm con}^{\rm ind}_{ {\cal M},\bar \ell, n}$ in $ {\sf Con}^{\rm ind}_{ {\cal M},\bar \ell, n}$ and $ {\sf Con}^{\rm ind}_{ {\cal M},n}$.} For any ${\cal M}\in {\sf M}_{\cal A}\cup {\sf M}_{\cal A}^{\rm ND}\cup {\sf M}_{\cal A}^{\rm DND}\cup {\sf M}_{\cal A}^{\rm NDB}$,  $\bar \ell\in {\sf Lab}_{\cal M} ^{[\leq \omega]}$, and $n\geq 1$, let $ {\sf Con}^{\rm ind}_{ {\cal M},\bar \ell, n}$ be the set $\{ (\ell_t\,.\,\vec \nu^{(t)})_{t\geq 1} \mid (\exists (\bar u^{(t)} )_{t\geq 1} \in (U_{\cal A}^\omega)^{\omega} ) ((\ell_t\,.\,\vec \nu^{(t)}\,.\,\bar u^{(t)} )_{t\geq 1} \in {\sf Con}_{{\cal M},n})\}$. Moreover, let $ {\sf Con}^{\rm ind}_{ {\cal M}, n}= \bigcup_{\bar \ell\in {\sf Lab}_{{\cal M}} ^{[\leq \omega]} } {\sf Con}^{\rm ind}_{ {\cal M},\bar \ell, n}$. On the one hand, the definition of $\to_{{\cal P},\kappa}^{\rm ind}$ is possible independently of any machine. On the other hand, the definition includes the definition of $\to_{{\cal P}_{\cal M},k_{\cal M}}^{\rm ind}$ and provides thus an initial framework in order to answer the question how a machine ${\cal M}$ can transform its indices. In particular, ${\sf Con}^{\rm ind}_{ {\cal M},\bar \ell, n}\not=\emptyset$ implies that ${\rm con}^{\rm ind}_{ {\cal P}_{\cal M},\bar \ell, k_{\cal M}, n} $ is the only sequence in ${\sf Con}^{\rm ind}_{ {\cal M},\bar \ell, n}$ and, thus, we obtain the relationship ${\sf Con}^{\rm ind}_{ {\cal M},\bar \ell, n} \subseteq \{ {\rm con}^{\rm ind}_{ {\cal P}_{\cal M},\bar \ell, k_{\cal M}, n}\}\subseteq {\sf Con}_{{\cal P}_{\cal M},k_{\cal M},n}^{\rm ind}$. If $B_{\bar \ell}$ belongs to ${\sf Lab}_{{\cal M},n}$, then $ {\rm con}^{\rm ind}_{ {\cal P}_{\cal M},\bar \ell, k_{\cal M}, n}$ can also be denoted by \textcolor{blue}{${\rm con}^{\rm ind}_{ {\cal M},\bar \ell, n}$} and then we have the equality ${\sf Con}^{\rm ind}_{ {\cal M},\bar \ell, n}=\{{\rm con}^{\rm ind}_{ {\cal M},\bar \ell, n}\}= \{ {\rm con}^{\rm ind}_{ {\cal P}_{\cal M},\bar \ell, k_{\cal M}, n}\}$. By using the new notations, we can consequently express Theorem \ref{TheoUniq1} as follows.

\begin{theorem}[Paths of partial $\!{\cal M}$-configurations for $B_0^{[\leq \omega]}$ and $n$]\label{TheoUniq2}\hfill 
Let  ${\cal P}\in {\sf P}_{\sigma}$  and let $ \kappa \geq k_{\cal P}$, $n\geq 1$, and $\bar \ell\in {\sf Lab}_{\cal P} ^{[\leq \omega]}$ be given. Then, there is one sequence ${\rm seq}^{\rm ind}_{ {\cal P},\bar \ell, \kappa, n}$ of partial $({\cal P},\kappa)$-configurations in ${\sf Seq}^{\rm ind}_{{\cal P},\kappa,n}$ such that 

\vspace{0.2cm}
$\!\!\{(\ell_t \,.\,\vec \mu^{(t)})_{t\geq 1}\in {\sf Con}^{\rm ind}_{ {\cal M},\bar \ell, n}\mid \, (\exists {\cal A}\in {\sf Struc})({\cal M}\in {\sf M}_{\cal A} \cup {\sf M}_{\cal A}^{\rm DND} \cup{\sf M}_{\cal A}^{\rm ND} \cup {\sf M}_{\cal A}^{\rm NDB}) $

\hspace{4.1cm}$ \&\,\, k_{\cal M}= \kappa\,\, \&\,\, {\cal P}_{\cal M}={\cal P}\}$

$\!\!\subseteq \{{\rm seq}^{\rm ind}_{ {\cal P},\bar \ell, \kappa, n}\}$.

\vspace{0.2cm}
\noindent holds. If the first set is not empty, then $\,{\rm seq}^{\rm ind}_{ {\cal P},\bar \ell, \kappa, n}$ is the $({\cal P}, \kappa, n)$-path $\,{\rm con}^{\rm ind}_{ {\cal P},\bar \ell, \kappa, n}$.
\end{theorem} 
\begin{corollary}[Paths of partial $\!{\cal M}$-configurations for a fixed $n$]Let ${\cal P}$\linebreak $\in {\sf P}_{\sigma}$, $\kappa \geq k_{\cal P}$, $n\geq 1$, ${\cal M}\in {\sf M}_{\cal A} \cup {\sf M}_{\cal A}^{\rm DND} \cup{\sf M}_{\cal A}^{\rm ND} \cup {\sf M}_{\cal A}^{\rm NDB}$ for an arbitrary $\sigma$-structure ${\cal A}$, ${\cal P}={\cal P}_{\cal M}$, and $k_{\cal M}=\kappa$. Then, we have

\vspace{0.2cm}
\qquad ${\sf Con}_{ {\cal M}, n} ^{\rm ind}= \{ {\rm con}^{\rm ind}_{ {\cal M},\bar \ell, n} \mid \bar \ell\in {\sf Lab}_{{\cal M},n} ^{[\leq \omega]} \} \subseteq \{{\rm seq}^{\rm ind}_{ {\cal P},\bar \ell, \kappa, n}\mid \bar \ell \in {\sf Lab}_{\cal P} ^{[\leq \omega]} \}$. 
\end{corollary}
\vspace{0.1cm}
The inclusion can be strict.
If $\ell_t\in {\cal L}_{{\cal P},{\rm H}_{\rm T}}$ is the label of an instruction  of {\sf ${\rm B}$-type} $({\sf cond}_{\ell_t},l_1,l_2)$,  then $\ell_t\to_{\cal P}^{\rm lab}l_1$ and $\ell_t\to_{\cal P}^{\rm lab}l_2$   hold. By Definition \ref{DefFlowLabel} (see {\sf (FL4)} and fn.\,\,\ref{BranchGeneral}), we can assume that $\{l_1,l_2\}\setminus \{T_{\ell_t}^{(\kappa)}(\vec \nu)\}$ is not empty for any $\vec \nu\in \bbbn_+^\kappa$. 

\begin{theorem}[Exclusion of certain ${\cal P}$-paths for a fixed length $n$]\label{PathCanc}\hfill Let  $\sigma$ be any signature, ${\cal P}\in {\sf P}_{\sigma}$, $ \kappa \geq k_{\cal P}$, and $n\geq 1$. Let ${\cal A}$ be any $\sigma$-structure and ${\cal M}$ be any BSS RAM in $ {\sf M}_{\cal A} \cup {\sf M}_{\cal A}^{\rm DND} \cup{\sf M}_{\cal A}^{\rm ND} \cup {\sf M}_{\cal A}^{\rm NDB}$ with ${\cal P}={\cal P}_{\cal M}$ and $k_{\cal M}=\kappa$.
Then, for any sequence $(\ell_t\,.\,\vec \nu^{(t)})_{t\geq 1}\in {\sf Seq}^{\rm ind}_{{\cal P},\kappa,n}$ with 

\vspace{0.2cm}
\qquad $\ell_{t+1}\not= T_{\ell_t}^{(\kappa)}(\vec \nu ^{(t)})$ for some $\ell_t\in {\cal L}_{{\cal P},{{\rm H}_{\rm T}}}$ and $T_{\ell_t}^{(\kappa)}$ in ${\cal F}_{{\cal P},\kappa}^{\rm ind}$,\hfill {\sf (P6)}

\vspace{0.2cm}
\noindent the ${\cal P}$-path $\bar \ell \in {\sf Lab}_{\cal P}$ cannot be traversed by ${\cal M}$ for any inputs of length $n$. 
\end{theorem}

\vspace{0.1cm}

{\bf From sets such as  ${\sf Lab}_{{\cal P},\kappa,n}^{\rm Co}$ to ${\sf Lab}_{{\cal P},n}^{[*]}$.} For any ${\cal P}\in {\sf P}_{\sigma}$, $ \kappa \geq k_{\cal P}$, and $n\geq 1$, let ${\sf Lab}_{{\cal P},\kappa,n}^{\rm Co}$ be the set consisting of the ${\cal P}$-paths $B_0$ for which ${\rm seq}^{\rm ind}_{ {\cal P},B_0^{[\leq \omega]}, \kappa, n}$ satisfies {\sf (P6)}. Let ${\sf Lab}_{{\cal P},\kappa,n}^{[*]}=_{\rm df } {\sf Lab}_{{\cal P}} \setminus {\sf Lab}_{{\cal P},\kappa,n}^{\rm Co}$ and $ {\sf Lab}_{{\cal P},n}^{\rm Co}=_{\rm df}\bigcup_{\kappa\geq k_{\cal P}}{\sf Lab}_{{\cal P},\kappa,n}^{\rm Co}$. Consequently, ${\sf Lab}_{{\cal P},\kappa,n}^{\rm Co}$ \textcolor{blue}{is equal to} $ {\sf Lab}_{{\cal P},n}^{\rm Co}$. This implies ${\sf Lab}_{{\cal P},n}^{[*]}={\sf Lab}_{{\cal P},\kappa,n}^{[*]}$. 

This kind of exclusion of certain paths is only of meaning for length $n$. Theorem \ref{PathCanc} means that the paths in ${\sf Lab}_{{\cal P}_{\cal M},k_{\cal M},n}^{\rm Co}$ are not useful for the evaluation of ${\cal M}$ on inputs of length $n$. They do not need to be evaluated. These paths or at least the remaining values $(\ell_{t+1}, \ell_{t+2}, \ldots)$ of a sequence $\bar \ell$ can be cancelled if $\ell_{t+1}\not= T_{\ell_t}^{(k_{\cal M})}(\vec \nu ^{(t)})$ holds as described in {\sf (P6)}. Thus, we only need to evaluate ${\sf Lab}_{{\cal P}_{\cal M},k_{\cal M},n}^{[*]}$ or ${\sf Lab}_{{\cal P}_{\cal M},n}^{[*]}$. For more details, see {\sf Extras II} and {\sf Extras III}.

\section{Ordered sets of program paths and more}\label{Order}
The investigation of the behavior of  a BSS RAM ${\cal M}$ is often closely connected with the evaluation of the program paths in ${\sf Lab}_{{\cal P}_{\cal M}}$ and the characterization of the computation paths in ${\sf Lab}_{\cal M}$. The following basic statements confirm once more that, generally, finite paths in ${\sf Lab}_{{\cal P}_{\cal M}}^{\rm fin }$ and infinite paths in ${\sf Lab}_{{\cal P}_{\cal M}}^{\rm infin}$ must be considered separately. Let ${\cal A}$ be any structure in ${\sf Struc}$. By definition (cf.\,\,\cite[p.\,\,15]{GASS25A}), a non-deterministic machine ${\cal M}\in {\sf M}_{\cal A}^{\rm ND}\cup {\sf M}_{\cal A}^{\rm DND} \cup {\sf M}_{\cal A}^{\rm NDB}$ {\sf halts} on $\vec x\in U_{\cal A}^\infty$ if ${\rm Res}_{\cal M}(\vec x)\not=\emptyset$ is satisfied. Then, ${\cal M}$ accepts $\vec x$ and we write $ {\cal M} ( \vec x)\downarrow$. For a deterministic machine ${\cal M}\in {\sf M}_{\cal A}$ and any $\vec x\in U_{\cal A}^\infty$, ${\cal M} ( \vec x)\downarrow $ means that ${\rm {\rm Res}_{\cal M}}(\vec x)$ is defined. Less formally, we can also say that halting means traversing a finite computation path and executing the stop instruction. The {\em halting set $H_{\cal M}$ of any BSS RAM ${\cal M }$} is given by $H_{\cal M}=\{ \vec x \in U_{\cal A}^{\infty} \mid {\cal M} ( \vec x)\downarrow \}$. 

\begin{propert}[Halting requires ${\sf Lab}_{\cal M}\cap{\sf Lab}_{{\cal P}_{\cal M}}^{\rm fin}\not=\emptyset$] Any BSS RAM ${\cal M}$ in ${\sf M}_{\cal A} \cup {\sf M}_{\cal A}^{\rm ND}\cup {\sf M}_{\cal A}^{\rm DND} \cup {\sf M}_{\cal A}^{\rm NDB}$ can halt only if it traverses a finite ${\cal P}_{\cal M}$-path.
\end{propert}

\begin{proposition}[Totality implies ${\sf Lab}_{\cal M}\!\subseteq {\sf Lab}_{{\cal P}_{\cal M}}^{\rm fin}$ for ${\cal M} \!\in\! {\sf M}_{\cal A}$]\label{TotalFuncImplFinPaths} If the result function ${\rm {\rm Res}_{\cal M}}$ of a BSS RAM ${\cal M}$ in ${\sf M}_{\cal A}$ is a total function, then each computation path in ${\sf Lab}_{\cal M}$ has finite length. 
\end{proposition}

Note that this generally is not valid for non-deterministic machines in ${\sf M}_{\cal A}^{\rm ND}\cup {\sf M}_{\cal A}^{\rm DND}\cup {\sf M}_{\cal A}^{\rm DNB}$ even if they compute total single-valued functions etc. 

\begin{corollary}[Computing $\chi_P$ implies ${\sf Lab}_{\cal M}\subseteq {\sf Lab}_{{\cal P}_{\cal M}}^{\rm fin}$] \hfill Let the structure ${\cal A}$ belong to ${\sf Struc}_{c_1,c_2}$ such that ${\cal A}$ contains two constants $c_1$ and $c_2$. If ${\cal M}$ is a BSS RAM in ${\sf M}_{\cal A}$ that computes the characteristic function $\chi_P$ of a decision problem $P\subseteq U_{\cal A}^\infty$, then each computation path in ${\sf Lab}_{\cal M}$ has finite length. 
\end{corollary}

\begin{proposition}[Inclusion $H_{\cal M}\subset U_{\cal A}^\infty$ implies ${\sf Lab}_{\cal M} \cap {\sf Lab}_{{\cal P}_{\cal M}}^{\rm infin}\not=\emptyset$] If ${\cal M}$ is a BSS RAM in $ {\sf M}_{\cal A} \cup {\sf M}_{\cal A}^{\rm ND}\cup {\sf M}_{\cal A}^{\rm DND} \cup {\sf M}_{\cal A}^{\rm NDB}$ whose halting set is a proper subset of $U_{\cal A}^\infty$, then there is an infinite computation path traversed by ${\cal M}$. 
\end{proposition}

\begin{corollary}[Semi-decidability of $P\subset U_{\cal A}^\infty$ means ${\sf Lab}_{\cal M} \cap {\sf Lab}_{{\cal P}_{\cal M}}^{\rm infin}\not= \emptyset$]\hfill Let $P\!\subset U_{\cal A}^\infty$. If ${\cal M}\!\in {\sf M}_{\cal A}$ semi-decides $P$ or computes the partial characteristic function $\bar \chi_P$ of $P$, then there is an infinite computation path traversed by ${\cal M}$. 
\end{corollary}

\subsection{Countability and enumerability of sets of paths}

As in classical recursion theory, we want to consider the notions of {\sf countability} and {\sf enumerability}. A set $N$ is {\em finite} if $N$ is the empty set $\emptyset$ or there are an $n\geq 1$ and a surjective function from $\{1,\ldots, n\}$ onto $N$. Otherwise, $N$ is {\em infinite}. In the case of infinite sets, it can be important to know whether a countably infinite set is enumerable over a $\sigma$-structure or not. 

\vspace{0.1cm}
{\bf Countable sets.} $\!\emptyset$ is {\em countable}. A non-empty set $N$ is {\em countable} if there is a surjective function from $\bbbn_+$ onto $N$. 
It is known that the set of all finite strings (including the empty string) over a finite alphabet is countable. The same holds for any set ${\cal L}^{\infty}$ of all finite tuples (of length $\geq 1$) with components in ${\cal L}$ if ${\cal L}$ is a finite set. By definition, for each $\sigma$-program ${\cal P}$, ${\cal L}_{\cal P}$ is a finite set and all components of all ${\cal P}$-paths belong to ${\cal L}_{\cal P}$. Consequently, all finite ${\cal P}$-paths in ${\sf Lab}_{\cal P}^{\rm fin}$ belong to ${\cal L}_{\cal P}^{\infty}$. Thus, the set $\,{\sf Lab}_{\cal P}^{\rm fin}$ can be infinite, but it is countable.
\begin{proposition}[Countability of ${\sf Lab}_{\cal P}^{\rm fin}$]
Let ${\cal P}\in {\sf P}_\sigma$ and ${\sf Lab}_{\cal P}^{\rm fin}\not=\emptyset$. Then, there is a sequence $(\vec \ell^{(i)})_{i\geq 1 }$ of ${\cal P}$-paths with $\{\vec\ell^{(i)}\in {\cal L}_{\cal P}^{\infty}\mid i\geq 1\} = {\sf Lab}_{\cal P}^{\rm fin}$. 
\end{proposition}

This, together with Proposition \ref{TotalFuncImplFinPaths}, leads to the following conclusion.
\begin{proposition}[$H_{\cal M}=U_{\cal A}^\infty$ implies a countable ${\sf Lab}_{\cal M}$ for ${\cal M} \in {\sf M}_{\cal A}$]\hfill If \linebreak ${\cal M}\in {\sf M}_{\cal A}$ computes a total function, then the set ${\sf Lab}_{\cal M}$ is countable. 
\end{proposition}

On the other hand, for any $\sigma$-program ${\cal P}$ and any $s\geq 1$, the set ${\cal L}_{\cal P}^s$ of tuples is finite and, therefore, the set $\bigcup_{s=1}^r{\cal L}_{\cal P} ^s$ is finite for any $r\geq 1$.
\begin{proposition}[More paths also mean longer paths]\label{LemmaAbzEndlPfade}Let ${\cal P}\in {\sf P}_\sigma$ and ${\sf Lab}_{\cal P}^{\rm fin}$ be infinite. Then, $({\sf Lab}_{\cal P}^{\rm fin})^\omega$ contains a sequence $(\vec \ell^{(i)})_{i\geq 1 }$ of ${\cal P}$-paths with $\vec \ell^{(i)}\in {\cal L}_{\cal P}^{s_i}$ for all $i\geq 1$ and $s_1<s_2< \cdots$. 
\end{proposition}

For any $s>|{\cal L}_{\cal P}|$ and $(\ell_1,\ldots,\ell_s)\in {\cal L}_{\cal P}^s$, there are $j,k\leq s$ with $j\not=k$ and $\ell_{j}=\ell_{k}$. For any $s_i>|{\cal L}_{\cal P}|$ considered in Proposition \ref{LemmaAbzEndlPfade} and every ${\cal P}$-path $(\ell_{i,1},\ldots,\ell_{i,s_i})$ in ${\sf Lab}_{\cal P}^{\rm fin}$, there are even two indices $j< s_i$ and $k< s_i$ with $j\not=k$ and $\ell_{i,j}=\ell_{i,k}$ since such a sequence $(\ell_{i,1},\ldots,\ell_{i,s_i})$ belongs to $({\cal L}_{\cal P}^\circ)^{s_i-1}\times \{\ell_{\cal P}\}$ by {\sf (L1)}. Thus, we obtain the following conclusion by Proposition \ref{PropertPathsAsWalks} (without using the axiom of choice or K\oo nig's Infinity Lemma, unlike in the general case).

\begin{corollary}[Infiniteness of ${\sf Lab}_{\cal P}^{\rm fin}$ implies ${\sf Lab}_{\cal P}^{\rm infin}\not=\emptyset$]Let ${\cal P}\in {\sf P}_\sigma$. If ${\sf Lab}_{\cal P}^{\rm fin}$ is an infinite set, then there is an infinite ${\cal P}$-path in ${\sf Lab}_{\cal P}$.
\end{corollary}

\vspace{0.1cm}
{\bf Enumerable sets.} $\emptyset$ is {\em enumerable over ${\cal A}$}.
According to \cite[p.\,\,600]{GASS20}, we say that any non-empty set $N\subseteq U_{\cal A}^\infty$ is {\em enumerable by an ${\cal A}$-machine ${\cal M}$} if 
\begin{itemize}\parskip -0.5mm
\item the input space ${\sf I}_{\cal M}$ is $U_{\cal A}^\infty$ or $\{u\}^\infty$ for some $u\in U_{\cal A}$,
\item the procedures ${\rm Input}_{\cal M}$ and ${\rm Output}_{\cal M}$ correspond to those of BSS RAMs on ${\sf I}_{\cal M}$ and ${\sf S}_{\cal M}$, respectively,
\item $H_{\cal M}= {\sf I}_{\cal M}$ and $N=\{{\rm Res}_{\cal M}(\vec x)\mid \vec x \in {\sf I}_{\cal M}\}$ hold, and
\item for all $n\geq 1$, $\vec x\in {\sf I}_{\cal M}\cap U_{\cal A}^n$ and $ \vec y\in {\sf I}_{\cal M}\cap U_{\cal A}^n$ imply ${\rm Res}_{\cal M}(\vec x)={\rm Res}_{\cal M}(\vec y)$.
\end{itemize}
In both cases, we say that $N$ is {\em enumerable by an ${\cal A}$-machine} and {\em enumerable over ${\cal A}$} and that {\em the elements of $N$ can be enumerated} ({\em by ${\cal M}$}) ({\em over ${\cal A}$}). The enumeration by a {\sf B}SS RAM {\em is of type} \textcolor{red}{$(B)$}. If ${\sf I}_{\cal M}=\{u\}^\infty$ holds, then the enumeration {\em is of type $(A_{u})$} and ${\rm Res}_{\cal M}$ is called {\em an enumeration function for $N$ of type \textcolor{red}{$(A_u)$}}. 

\vspace{0.1cm}
Now, let $u$ be any element in $U_{\cal A}$. If there is an enumeration of the elements of $N$ that is of type $(B)$, then there is an enumeration function for $N$ of type $(A_u)$. The set $\{u\}^\infty$ contains {\sf exactly one unary code} in $\{u\}^n$, denoted by $\vec u^{[n]}$, for each $n\in \bbbn_+$. Therefore, surjective functions of the form $f:\{u\}^\infty \to N$ are particularly well-suited for describing the enumerability of a set $N$. However, our definition is broader. This can be advantageous in certain situations because, for an enumeration of type $(B)$ by an ${\cal M}\in {\sf M}_{\cal A}$, the specific input values are irrelevant. We can start with $c(I_1)=n$ and arbitrary initial values in the $Z$-registers and obtain the $n$-th member of the sequence $(f(\vec u^{[n]}))_{n\geq 1}$. The advantage is that ${\cal P}_{\cal M}$ can be easily adapted for simulating the behavior of ${\cal M}$ by a $d$-tape machine. It is, for example, possible to use tape 2 for executing an adapted subprogram $({\cal P}_{\cal M})'$~-- regardless of the values stored in the $Z$-registers~-- if it is ensured that the initial values in the index registers for tape 2 are suitably determined and $c(I_{2,1})=n$ and $c(I_{2,2})=\cdots= c(I_{2,k_{{\cal P}_{\cal M}}})=1$ hold. For the definition of $d$-tape machines, see \cite[pp.\,\,591--592]{GASS20} and \cite[pp.\,\,5--6]{GASS25B}. 

Note, that sets in ${\rm SDEC}_{\cal A}$ {\sf do not have to be enumerable} over ${\cal A}$ {\sf or countable}. Moreover, we {\bf cannot assume} that $u$ is computable or that $\{u\}^\infty$ is semi-decidable by a BSS RAM. We say that $u$ {\em is computable over ${\cal A}$} if $f:U_{\cal A}^\infty \to \{u\}$ with $f(\vec x)=u$ for all $\vec x \in U_{\cal A}^\infty $ is computable by a BSS RAM in ${\sf M}_{\cal A}$. Consequently, the constants of ${\cal A}$ are computable over ${\cal A}$. We say that $N$ is {\em \textcolor{blue}{$(u,A_u)$}-enumerable} ({\em by a BSS RAM}) if there is a machine in ${\sf M}_{\cal A}$ that first computes $\vec u^{[n]}\in U_{\cal A}^n$  for all $n\geq 1$ and  any input $\vec x \in U_{\cal A}^n$ before it simulates the execution of an ${\cal A}$-machine computing a surjective function $f:\{u\}^\infty \to N$. 

\begin{proposition}[Enumerability: from type $(A_u)$ to $(B)$]\label{EnumSets2}\hfill If $u\in U_{\cal A}$ is \linebreak computable over ${\cal A}$, $N$ is a subset of $U_{\cal A}^\infty$, and there is an enumeration function for $N$ of type $(A_u)$, then $N$ is $(u,A_u)$-enumerable by a BSS RAM in $ {\sf M}_{\cal A}$.
\end{proposition}

For any ${\cal M}\in {\sf M}_{\cal A}$ and $P\subseteq U_{\cal A}^\infty$, let $f_{{\cal M},P}:\,\subseteq\! P \to U_{\cal A}^\infty$ be the restriction of ${\rm Res}_{\cal M}:\,\subseteq\! U_{\cal A}^\infty \to U_{\cal A}^\infty$ to the domain $P$. If $P\subseteq H_{\cal M}$, then $f_{{\cal M},P}$ is total and we write $f_{{\cal M},P}:P \to U_{\cal A}^\infty$. $({\rm Res}_{\cal M})_{|H_{\cal M}}$ is the restriction $f_{{\cal M},H_{\cal M}}:H_{\cal M} \to U_{\cal A}^\infty$. If $N\subseteq U_{\cal A}^\infty$ and \textcolor{blue}{$\{u\}^\infty\subseteq H_{\cal M}$} hold and \textcolor{blue}{$f_{{\cal M},\{u\}^\infty}$} is a  surjective function onto $N$, then  $f_{{\cal M},\{u\}^\infty}$ is called {\em an enumeration function for $N$ of type \textcolor{red}{$(B_u)$}}. Such an enumeration function  is computable by an ${\cal A}$-machine that uses  the input space $\{u\}^\infty$, executes the  program ${\cal P}_{\cal M}$, and provides then an output like ${\cal M}$.

\begin{proposition}[Enumeration: from type $(B_u)$ to type $(A_u)$]\label{EnumSets} \hfill Let $N$ \linebreak be a subset of $U_{\cal A}^\infty$ and $u\in U_{\cal A}$. Each  enumeration function for $N$ of type $(B_u)$ is also an enumeration function for $N$ of type $(A_u)$.
\end{proposition} 

\begin{proposition}[From type $(B)$ to $({\rm Res}_{\cal M})_{|H_{\cal M}}$ of type $(B_u)$]\label{EnumSets3} Let $\,N\in U_{\cal A}^\infty$ and $u\in U_{\cal A}$.
If $\{u\}\in {\rm SDEC}_{\cal A}$ holds and there is an enumeration of the elements of $N$ of type $(B)$, then there is an ${\cal M}\in {\sf M}_{\cal A}$ such that \textcolor{blue}{$H_{\cal M}=\{u\}^\infty$} holds and $({\rm Res}_{\cal M})_{|H_{\cal M}}$ is an enumeration function for $N$ of type $(B_u)$.
\end{proposition}

The enumerability of certain subsets of ${\sf Lab}_{\cal P}$ or sets of initial segments of ${\cal P}$-paths can be a sufficient condition or sometimes even a necessary condition for an automatic evaluation of ${\cal P}$-paths or for an automatic or computer-assisted proof if this proof requires an evaluation of program paths. However, not every countable set is enumerable within a suitable framework.\,\footnote{ Any non-empty subset of $\bbbn_+$ is countable and enumerable over $(\bbbr;\bbbr;-,\cdot\,;\geq)$ by a suitable BSS RAM over the reduct $\bbbr_{{\sf Ex0},r}=(\bbbr;0,1,10,r;-,\cdot\,;\geq)$ for some $r\in \bbbr$ (cf.\,\,also Example \ref{NichtabzPfade}). More precisely, each machine constant $r\in\,\, ]0,1[$ (in decimal form) with digits in $\{0,1\}$ can be used as a code and its digits can be treated as  function values of the characteristic function of an $N_r\subseteq \bbbn_+$. Therefore, the power set of $\bbbn_+$ is not countable and there are not enough BSS RAMs over ${\cal A}_{\bbbn}$ for enumerating the elements of every non-empty $N\subseteq \bbbn_+$ by a machine ${\cal M}_N$ in ${\sf M}_{{\cal A}_{\bbbn}}$ since the set of $(1;1;2)$-programs and ${\cal A}_{\bbbn}$ are countable.} Regardless of this, a well-order of type $\omega$ can be useful. We say that an enumeration function $enum :\{u\}^\infty \to N$ and the resulting enumeration are {\em $\preceq$-compatible} and {\em compatible with $\preceq$} and allow to enumerate the elements of $N$ {\em in accordance with the order $\preceq$} if $n_1<n_2$ implies $enum(\vec u^{[n_1]})\preceq enum(\vec u^{[n_2]})$ for all $n_1,n_2\in \bbbn_+$. Let us consider some questions.

\vspace{0.05cm}
\begin{questions}[Enumeration functions or orderings for ${\cal P}$-paths]
\vspace{0.01cm}
\hfill

\nopagebreak 
\noindent \fbox{\parbox{11.8cm}{ \centering \vspace{0.05cm}

How can  the sets  $\,{\sf Lab}_{\cal P}^{\rm fin}$ and $\,{\sf Lab}_{\cal P}$ be ordered?
\vspace{0.05cm}}}

\vspace{0.1cm}
\nopagebreak 
\noindent \fbox{\parbox{11.8cm}{\centering \vspace{0.05cm}

Is it  possible to enumerate the  finite ${\cal P}$-paths  or the finite initial segments of  ${\cal P}$-paths 
 in accordance with  a well-order?
\vspace{0.05cm}}}
\end{questions}

\vspace{0.05cm}

\subsection{Finite paths, their enumeration, and a well-order} Let ${\cal P}$ be a $\sigma$-program. For $\ell_{\cal P}>1$ and $s\geq 1$, let $\leq_{{\cal P},s}^{\sf lexic}$ be the lexicographic order on ${\cal L}_{\cal P}^s$ that results from $1<\cdots <\ell_{\cal P}$ for the underlying alphabet ${\cal L}_{\cal P}$. We want to define an order $\,\leq_{\cal P}^{\rm fin}\,$ on ${\sf Lab}_{\cal P}^{\rm fin}$ such that

\vspace{0.2cm}
\quad $\{(B',B'')\!\in\!({\sf Lab}_{\cal P}^{\rm fin})^2\mid B'\leq_{{\cal P},s}^{\sf lexic}B''\}= \{(B',B'')\!\in\!({\cal L}_{\cal P}^s)^2 \mid B'\leq_{\cal P}^{\rm fin}B''\}$ 

\vspace{0.25cm}
\noindent holds for each $s\geq 1$ and thus $\leq_{{\cal P},s}^{\sf lexic}\cap ({\sf Lab}_{\cal P}^{\rm fin})^2$ and $\leq_{\cal P}^{\rm fin}\cap ({\cal L}_{\cal P}^s)^2 $ are the same set. For $s=1$, we only need to consider $\ell_{\cal P}=1$, $\ell_{\cal P} \leq_{{\cal P},s}^{\sf lexic} \ell_{\cal P} $, and  $\ell_{\cal P} \leq_{\cal P}^{\rm fin} \ell_{\cal P} $. If $s=2$ holds, then  we are interested in $(1,\ell_{\cal P}) \leq_{{\cal P},s}^{\sf lexic} (1,\ell_{\cal P}) $ and $(1,\ell_{\cal P}) \leq_{\cal P}^{\rm fin} (1,\ell_{\cal P})$. 

 \vspace{0.05cm}
\begin{overview}[The binary relation $\,\leq_{\cal P}^{\rm fin}\,$ on ${\,\sf Lab}_{\cal P}^{\rm fin}$]\label{OrderOnFinP}

\hfill \vspace{0.05cm}

\nopagebreak 
\noindent \fbox{\parbox{11.8cm}{\hfill{\small Let ${\cal P}\in {\sf P}_\sigma$, ${\cal L}_{\cal P}=\{1,\ldots,\ell_{\cal P}\}$, and $1<\cdots <\ell_{\cal P}$ if $1<\ell_{\cal P}$.}

\vspace{0.1cm}
\em
For any program paths $B^{(1)}$ and $B^{(2)}$ in ${\sf Lab}_{\cal P}^{\rm fin}$ given by 

\vspace{0.2cm}
\hspace{1cm}$B^{(1)}=(\ell_{1,1},\ldots,\ell_{1,s_1})\in {\cal L}_{\cal P}^\infty$ 
 \quad and \quad $B^{(2)} =(\ell_{2,1},\ldots,\ell_{2,s_2})\in {\cal L}_{\cal P}^\infty $,

\vspace{0.2cm}
\noindent let $B^{(1)}\textcolor{red}{<_{\cal P}^{\rm fin} \,} B^{(2)}$ hold
 if either condition {\sf (O1)} or condition {\sf (O2)} holds.

\vspace{0.05cm}

\quad {\sf (O1)}  $s_1<s_2$.

\quad {\sf (O2)}  $s_1=s_2\geq 3$ and there is a $j\in \{2,\ldots, s_1-1\}$ with {\sf (O2a)} and {\sf (O2b)}.

\vspace{0.05cm}

\begin{tabular}{lll}

\qquad\,\, {\sf (O2a)} \, $\ell_{1,i}=\ell_{2,i}$ for all $i \in \{1,\ldots, j-1\}$,\\ 
\qquad\,\,  {\sf (O2b)} \, $\ell_{1,j} < \ell_{2,j}$. &\vspace{0.1cm}\\
\end{tabular}

\vspace{0.05cm}
\noindent
Let $B^{(1)}\textcolor{red}{\leq_{\cal P}^{\rm fin}\,} B^{(2)}$ be true if $B^{(1)}\,\textcolor{blue}{<_{\cal P}^{\rm fin} \,}B^{(2)}\,$ or $\,B^{(1)}\,\textcolor{blue}{=}\, B^{(2)}$ holds.

\em \hfill {\small {\sf O} is derived from {\sf order}.}
}}
\end{overview}
\begin{proposition}[The strict order $\,<_{\cal P}^{\rm fin}$ and the well-order $\,\leq_{\cal P}^{\rm fin}$]\label{PropStand1}\hfill 
$\,<_{\cal P}^{\rm fin}$ \linebreak is irreflexive and a strict order. The relationships $B'\not <_{\cal P} ^{\rm fin} B'' $ and $B''\not <_{\cal P} ^{\rm fin} B' $ mean $B'= B'' $. $\leq_{\cal P}^{\rm fin}$ is a reflexive, transitive, antisymmetric, and linear order. Moreover, any subset of $\,{\sf Lab}_{\cal P}^{\rm fin} $ has a minimal element with respect to $\leq_{\cal P}^{\rm fin}$.
\end{proposition}

{\bf The standard order $\leq_{\cal P}^{\rm fin}$.} We call $\leq_{\cal P}^{\rm fin}$ {\em the standard order on ${\sf Lab}_{\cal P}^{\rm fin}$}. 

\begin{theorem}[Well-order on ${\sf Lab}_{\cal P}^{\rm fin}$]\label{WellOrderFinLabel}For any ${\cal P}\in {\sf P}_\sigma$, the standard order $\leq_{\cal P}^{\rm fin}$ on ${\sf Lab}_{\cal P}^{\rm fin}$ is a well-order of the order type $\omega$.
\end{theorem}

Let us introduce some $\leq_{\cal P}^{\rm fin}$-compatible enumeration functions of the form $enum:\{u\}^\infty \to {\sf Lab}_{\cal P}^{\rm fin}$.
Because of {\sf (O1)},  the relation $\leq_{\cal P}^{\rm fin}$ and, consequently, the resulting functions are {\em length-compatible}. This means by definition that the inequality $n_1<n_2$ implies $length(enum(\vec u^{[n_1]}))\leq length(enum(\vec u^{[n_2]}))$ for all $n_1,n_2\in \bbbn_+$. Such functions are useful for evaluating the complexity of machines ${\cal M}$. If the running time of ${\cal M}$ is bounded by a number $r$, then only ${\cal P}_{\cal M}$-paths of length $s\leq r$ are to be investigated.

If ${\sf Lab}_{\cal P}^{\rm fin}\not=\emptyset$, then the finite ${\cal P}$-paths can be $\leq_{\cal P}^{\rm fin}$-compatibly enumerated by machines over ${\cal A}_{\bbbn}$ and over finite structures such as ${\cal A}_{{\cal L}_{\cal P}}$.

\vspace{0.3cm}
{\bf The structure ${\cal A}_{{\cal L}_{\cal P}}$.} 
Let ${\cal A}_{{\cal L}_{\cal P}}$ be a first-order {\em label structure} of signature $(\ell_{\cal P}+1;1,1,1;2)$ that is given by 
 ${\cal A}_{{\cal L}_{\cal P}}\!=({\cal L}_{\cal P}\cup\{0\}; {\cal L}_{\cal P}\cup\{0\};\hat\beta_{\cal P}, \hat \beta^+_{\cal P}, \hat \beta^-_{\cal P};=)$. 
The total functions $\hat \beta_{\cal P}$, $\hat \beta^+_{\cal P}$, and $\hat\beta^-_{\cal P}$ are derived from the partial $ \beta_{\cal P}^\circ$-functions $\beta_{\cal P}$, $ \beta^+_{\cal P}$, and $\beta^-_{\cal P}$ that belong to $\mathfrak{Fl\,}_{\cal P}^{\sf label}$. They are defined on $\{0,1,\ldots, \ell_{\cal P}\}$ by

\vspace{0.1cm}
\qquad $\hat \beta_{\cal P}(\ell)=\left\{ \begin{array}{ll}\beta_{\cal P}(\ell)\quad & \mbox {\rm if } \ell\in {\cal L}_{\cal P}^{\rm NB},\\
0& \mbox {\rm if } \ell\in ({\cal L}_{\cal P}\cup \{0\})\setminus {\cal L}_{\cal P}^{\rm NB}\,$,$\\
\end{array}\right.$

\vspace{0.1cm}
\qquad $\hat \beta_{\cal P}^{\mbmss{o}}(\ell)=\left\{ \begin{array}{ll}\beta_{\cal P}^{\mbmss{o}}(\ell)\quad &\mbox {\rm if } \ell\in {\cal L}_{\cal P}^{\rm B},\\
0\quad& \mbox {\rm if } \ell\in ({\cal L}_{\cal P}\cup \{0\})\setminus {\cal L}_{\cal P}^{\rm B}\\
\end{array}\right.$ \quad for $\textcolor{red}{{\mbm{o}}}\in \{\textcolor{blue}{+},\textcolor{blue}{-}\}$.

\vspace{0.1cm}
Overview \ref{ExamEnumLab1} contains an algorithm for an enumeration of all finite ${\cal P}$-paths over ${\cal A}_{{\cal L}_{\cal P}}$ if $\ell_{\cal P}>1$ holds.\,\footnote{ In the following representations of algorithms, $\hat \beta_{\cal P}$ is an operation symbol. It is used instead of the  symbol $f_1^2$ that can be interpreted by the function $\hat \beta_{\cal P}$. The same holds for $\hat \beta_{\cal P}^{\mbmss{o}}$.}
If $\ell_{\cal P}=1$ holds, then we use Properties \ref{PropertOnePath} and  define \textcolor{blue}{$enum_{\cal P}^{\sf (O)}:\{1\}^\infty\to \{1\}$}  by $enum_{\cal P}^{\sf (O)}(\vec 1^{[n]})=1$ for all $n\geq 1$.  $enum_{\cal P}^{\sf (O)}$ is an enumeration function of type $(A_1)$ and ${\sf Lab}_{\cal P}^{\rm fin}$ is enumerable over ${\cal A}_{{\cal L}_{\cal P}}$. 

 \vspace{0.2cm}
{\bf An algorithm for enumerating the ${\cal P}$-paths of length $s$.} The ${\cal P}$-paths of any fixed length $s$ can be enumerated in accordance with $\leq_{\cal P}^{\rm fin} \cap\, ({\cal L}_{\cal P}^s)^2$.  Let $1<\ell_{\cal P}$, $s\geq 2$, $\textcolor{blue}{i_0}=|{\sf Lab}_{\cal P}^{\rm fin} \cap {\cal L}_{\cal P}^s|$, and $i_0>0$ and let the $i_0$-th finite ${\cal P}$-path  also be {\sf the $n$-th path} and the  output for input $\vec 1^{[n]}$ if  $n>i_0$ holds. 
 The algorithm presented in Overview \ref{ExamEnumLab0} enables to compute an enumeration function \textcolor{blue}{$enum_{{\cal P},s}^{\sf (O)}: \{1\}^\infty\to ({\sf Lab}_{\cal P}^{\rm fin} \cap {\cal L}_{\cal P}^s)$} for $s\geq 4$.  After computing the paths $enum_{{\cal P},s}^{\sf (O)}(\vec 1^{[1]}), \ldots , enum_{{\cal P},s}^{\sf (O)}(\vec 1^{[n]})$ for input $\vec 1^{[n]}$,  the ${\cal P}$-path $enum_{{\cal P},s}^{\sf (O)}(\vec 1^{[n]})$ is  output.  For the used subprograms, see also \cite{GASS25B}. For ${\cal P}$-paths of length $s$, it is enough to use $s-2$ loop constructs. The number $i_0$ is bounded by $i_0\leq |{\cal L}_{\cal P}^{s-2}|$. It can be computed, for instance, by executing a subprogram summarized by the pseudo instruction $I_2:=succ^{i_0-1}(I_2)$. This means that there is one program for any possible value $i_0$ and one of these programs enables a complete enumeration. If we know $i_0$, then we know which program provides all paths of length $s$. The values stored in $Z_1,\ldots,Z_s$ are constants of ${\cal A}_{{\cal L}_{\cal P}}$ that belong to ${\cal L}_{\cal P}$. For evaluating the test conditions such as $(cond\,2)$ by a series of ${\sf if}$-instructions, it is sufficient that the used functions $\hat \beta_{\cal P}$, $\hat \beta_{\cal P}^+$, and $\hat \beta_{\cal P}^-$ are defined for labels in ${\cal L}_{\cal P}^\circ$. 
\begin{overview}[An algorithm for computing \textcolor{blue}{$enum_{{\cal P},s}^{\sf (O)}$} over ${\cal A}_{{\cal L}_{\cal P}}$]\label{ExamEnumLab0}

\hfill

\nopagebreak 
\noindent \fbox{\parbox{11.8cm}{

\hfill {\small Let ${\cal P}\in{\sf P}_\sigma$, \textcolor{red}{$|{\cal L}_{\cal P}|>1$}, $s$ be fixed and \textcolor{red}{$s\geq 4$}, 
 $i_0=|{\sf Lab}_{\cal P}^{\rm fin} \cap {\cal L}_{\cal P}^s|$, and \textcolor{red}{$i_0\geq 1$}.}

\em 
\hspace{0.62cm}{\sf Input} $\vec 1^{[n]}\in \{1\}^{n}.$

\hspace{0.62cm}{\sf if $n>i_0$ then $n:=i_0;$} \hfill \textcolor{gray}{\small \it (If a ${\cal P}$-path is found and it is not the $n$-th path,} 

{\sf \hspace{0.62cm}$i:=1;$ $Z_1:=1;$   \hfill \textcolor{gray}{\small \it then $i$ is incremented, see below.)}

\hspace{0.62cm}for $Z_2:=1,2,\ldots, \ell_{\cal P}-1$ do   

\hspace{0.82cm}\{ if $Z_2\in \{\hat \beta_{\cal P}(Z_1), \hat \beta_{\cal P}^+(Z_1), \hat \beta_{\cal P}^-(Z_1)\} $ \hfill {\small $(cond\,2)$}

\hspace{1cm} then for $Z_3:=1,2,\ldots, \ell_{\cal P}-1$ do 

\hspace{2.1cm}\{ if $Z_3\in \{\hat \beta_{\cal P}(Z_2), \hat \beta_{\cal P}^+(Z_2), \hat \beta_{\cal P}^-(Z_2)\} $

\hspace{2.28cm} then $\,\cdots$

\vspace{0.2cm}

\hspace{2.44cm} $\,\,\cdots$ for $Z_{s-1}:=1,2,\ldots, \ell_{\cal P}-1$ do 

\hspace{3.4cm} \{ if $Z_{s-1}\in \{\hat \beta_{\cal P}(Z_{s-2}), \hat \beta_{\cal P}^+(Z_{s-2}), \hat \beta_{\cal P}^-(Z_{s-2})\} $ 

\hspace{3.7cm} then \{ $Z_s:=\ell_{\cal P};$

\hspace{4.78cm} if $Z_s\!\in\! \{\hat \beta_{\cal P}(Z_{s-1}), \hat \beta_{\cal P}^+(Z_{s-1}), \hat \beta_{\cal P}^-(Z_{s-1})\}$ 

\hspace{4.78cm} then \textcolor{red}{\{} if $i=n$ then goto \textcolor{magenta}{$l_1$}$;$ $i:=i+1$ \textcolor{red}{\}} 
\}

\hspace{3.4cm} \} $\,\cdots$ \} \}$;$

\noindent \textcolor{magenta}{$l_1$} : {\sf Output} $(c(Z_1),\ldots,c(Z_s))$.
} 
}}\end{overview}

\vspace{0.1cm}
{\bf Stacks and backtracking.} For determining the $n$-th member in a sequence of finite ${\cal P}$-paths ordered by $\leq_{\cal P}^{\rm fin}$ (and ordered by $<_{\cal P}^{\rm fin}$ for the first $i_0$ members), we can use the $Z$-registers like a stack.\,\footnote{ Before the content of a register $Z_{s_0}$ is changed,  the contents of  $Z_{s_0}, \ldots,Z_s$ can be deleted  (cleared) and set to 1 since the previous values are irrelevant for further processing.} For  the input $\vec 1^{[n]}$, the search for the $n$-th ${\cal P}$-path of length $s$ can be realized by a {\sf simple backtracking procedure}.\,\footnote{ This means that the contents of the registers from $Z_{s_0}$ to $Z_{s}$ can be changed if we interpret \linebreak the position $s_0$ as a backtracking point. By saving and  processing  all possible backtracking points, label paths can be generated  recursively in order to enumerate the ${\cal P}$-paths in ${\cal L}_{\cal P}^s$.} If the components $\ell_1,\ldots,\ell_s$ of the $i$-th ${\cal P}$-path in ${\cal L}_{\cal P}^s$ with $i< n \leq i_0 $ are stored in $Z_1,\dots, Z_s$, then let $\ell_{s_0}$ be the last component with $\ell_{s_0}< \ell_{\cal P}-1$. In such a case (including the paths in ${\sf Lab}_{\cal P}^{\rm Co}$), $c(Z_{s_0+1})\!=\!\cdots\! =\!c(Z_{s-1})\!=\ell_{\cal P}-1$ and $c(Z_s)=\ell_{\cal P}$ hold. Because of  $\ell_{s_0}<\ell_{\cal P}-1$, $Z_{s_0}$ obtains~-- via a {\sf backtracking step}~-- a new value $\ell_{s_0}+1$. If $\ell_{s_0}+j $ belongs to $\{ \hat \beta_{\cal P}^+(c(Z_{s_0-1})),\hat \beta_{\cal P}^-(c(Z_{s_0-1}))\}\setminus \{\ell_{s_0}\}$ for some $j= 1,\ldots, \ell_{\cal P}-\ell_{s_0}-1$, then the assignment  $Z_{s_0}:=\ell_{s_0}+j$ leads to the next loop construct \, {\sf for $Z_{s_0+1}:=1,2,$} \ldots\,. Otherwise, a further backtracking to $Z_{s_1}$ with $s_1<s_0$ is necessary. Further explanations are given below.

\vspace{0.2cm}
{\bf An algorithm for enumerating all finite paths.} Let  $|{\sf Lab}_{\cal P}^{\rm fin}|>0$.
The algorithm in Overview \ref{ExamEnumLab1} can be used to compute an enumeration function \textcolor{blue}{$enum_{\cal P}^{\sf (O)}:\{1\}^\infty\!\to {\sf Lab}_{\cal P}^{\rm fin}$}  over ${\cal A}_{{\cal L}_{\cal P}}$. $n_1<n_2$ implies $length(enum_{\cal P}^{\sf (O)}(\vec 1^{[n_1]}))\leq length(enum_{\cal P}^{\sf (O)}(\vec 1^{[n_2]}))$ for all $n_1,n_2\in \bbbn_+$, and thus this enumeration is {\em length-compatible}. This means that we generally need more and more $Z$-registers if we want to use the $Z$-registers like a stack that contains a sequence of labels. If ${\sf Lab}_{\cal P}^{\rm fin}$ is infinite, then ${\sf ass}(Z_1)$ stands for $Z_1:=0$. 
Otherwise, ${\sf ass}(Z_1)$ stands for $Z_1:=1$ and we use $i_0=|{\sf Lab}_{\cal P}^{\rm fin}|$. One of the resulting programs~-- using $Z_1:=0$ or $Z_1:=1$ and some $i_0$~-- enables the wished enumeration.\,\footnote{  No output could  be found only if ${\sf Lab}_{\cal P}^{\rm fin}$ would be empty.}  ${\sf ass}(Z_1)$ can be realized by an instruction of type (2) at the beginning. If ${\sf Lab}_{\cal P}^{\rm fin}$ is infinite, then we obtain the $n$-th finite ${\cal P}$-path for input $\vec 1^{[n]}$. If ${\sf Lab}_{\cal P}^{\rm fin}$ is finite and $n>i_0$ holds, then the $i_0$-th finite ${\cal P}$-path is output for input $\vec 1^{[n]}$. 

\vspace{0.1cm}
{\bf Backtracking.\,} Let $s$ be a fixed length.
Assume that  $\,c(Z_{j-1})\in {\cal L}_{\cal P}^{\rm B}$ holds for some $j<s$.  \footnote{ We assume that $\hat \beta_{\cal P}^+(c(Z_{j-1}))\not=\hat \beta_{\cal P}^-(c(Z_{j-1}))$ holds (cf.\,\,fn.\,\,\ref{BranchGeneral}).} If, for $\ell\in{\cal L}_{\cal P}^\circ$,  $\{1,\ldots, \ell-1\}\cap
Succ_{\cal P}(c(Z_{j-1}))$ is empty \footnote{ We can here use the equality $Succ_{\cal P}(c(Z_{j-1})) =\{ \hat \beta_{\cal P}^+(c(Z_{j-1})),\hat \beta_{\cal P}^-(c(Z_{j-1})) \}$.}
and $\ell\in Succ_{\cal P}(c(Z_{j-1})) $ holds, then the assignment $Z_j:=\ell$ defines a backtracking point $j$ that could be important. The parameters $j$ and $\ell+1$ derived from  $c(Z_j)=\ell$ can be used to identify this point. We have  $\ell=\min Succ_{\cal P}(c(Z_{j-1}))$ and   a second case  could be considered later. If the set  $\{\ell+1,\ldots, \ell_{\cal P}-1\}$ $\cap \, Succ_{\cal P}(c(Z_{j-1}))$ is not empty, then $\max(Succ_{\cal P}(c(Z_{j-1})) \setminus \{\ell_{\cal P}\})$ will be later assigned to $Z_j$. If there are backtracking points  and $s_0$ results from evaluating these points by the subprogram $(sub\,0)$ and thus $s_0\not=0$ holds, then $Z_{s_0}, \ldots, Z_s$ can be changed by $Z_{s_0}:=\max \cdots$ and $Z_{s_0+1}:=1$ and by $Z_{s_0+2}:=1$ if $(cond\, k)$ is satisfied for $k=s_0+1$. 
For  the next $s$, the search for finite paths can start again with $c(Z_1)=1$. At first, there is no further backtracking point. 

\begin{overview}[An algorithm for computing \textcolor{blue}{$enum_{\cal P}^{\sf (O)}$} over ${\cal A}_{{\cal L}_{\cal P}}$]\label{ExamEnumLab1}
\vspace{0.05cm}

\hfill 

\nopagebreak 
\noindent \fbox{\parbox{11.8cm}{

\hfill {\small Let ${\cal P}\in {\sf P}_\sigma$, $|{\cal L}_{\cal P}|>1$, $i_0=|{\sf Lab}_{\cal P}^{\rm fin}|$, and $i_0\geq 1$.}
\em 

\hspace{0.62cm}{\sf Input} $\vec 1^{[n]}\in \{1\}^{n}.$

\hspace{0.62cm}${\sf ass}(Z_1);\,$ {\sf if $Z_1=1$ \& $n>i_0$ then $n:=i_0;$}

{\sf
\hspace{0.62cm}$s:=1;$ $i:=1;$ $\textcolor{blue}{s_0}:=0;$ $Z_1:=1;$

\noindent \textcolor{magenta}{$l_1$} : if $\textcolor{blue}{s_0}=0$ then \{ $s:=s+1;$  $k:=2$ \}

\hspace{1.95cm} else \hspace{0.07cm}\{ $Z_{\textcolor{blue}{s_0}}:= \max\{\hat\beta_{\cal P}^+(Z_{s_0-1}), \hat\beta_{\cal P}^-(Z_{s_0-1} ) \};$  $k:= s_0+1$ \}$;$

\vspace{0.1cm}
\noindent \textcolor{magenta}{$l_2$} : \textcolor{brown}{if $k=s$ then goto \textcolor{magenta}{$l_3$}}

\hspace{1.78cm} \textcolor{brown}{ else} for $Z_{k}:=1,2,\ldots, \ell_{\cal P}-1$ do 

\hspace{2.8cm}\{ if $Z_k\in \{\hat \beta_{\cal P}(Z_{k-1}), \hat \beta_{\cal P}^+(Z_{k-1}), \hat \beta_{\cal P}^-(Z_{k-1})\} $ \hfill {\it\small $(cond\, k)$}

\hspace{2.95cm} then \{ $k:=k+1;$ goto \textcolor{magenta}{$l_2$} \} \}$;$

\vspace{0.15cm}
\textcolor{magenta}{$l_3$} : if $k=s$ \,\&\, $\ell_{\cal P}\!\in\! \{\hat \beta_{\cal P}(Z_{s-1}), \hat \beta_{\cal P}^+(Z_{s-1}), \hat \beta_{\cal P}^-(Z_{s-1})\}$

\hspace{0.62cm}then \textcolor{red}{\{} if $i=n$ then \{ $Z_s:=\ell_{\cal P};$ goto \textcolor{magenta}{$l_4$} \} else $i:=i+1$ \textcolor{red}{\}}$;$

\vspace{0.1cm}
\hspace{0.62cm}$\textcolor{blue}{s_0}\!:=\!\max \{j<k\mid j>1\,\,\&\,\, Z_{j-1} \!\in\! {\cal L}_{\cal P}^{\rm B} \,\,\&\,\,Z_{j}\!=\!\min\{\hat\beta_{\cal P}^+(Z_{j-1}), \hat\beta_{\cal P}^-(Z_{j-1} ) \}$

\hspace{3cm}    $\&\,\, \max\{\hat\beta_{\cal P}^+(Z_{s_0-1}), \hat\beta_{\cal P}^-(Z_{s_0-1} )\} <\ell_{\cal P} \mbox{ or } j=0\};$  \hfill  {\it\small $(sub\,0)$}

\hspace{0.62cm}goto \textcolor{magenta}{$l_1$}$;$ 

\noindent \textcolor{magenta}{$l_4$} : {\sf Output} $(c(Z_1),\ldots,c(Z_s))$.} 
\hfill {\small \em \textcolor{gray}{An example  is attached (cf.\,\,p.\,\pageref{ExampleFirstAlgo}).}}
}}\end{overview}

The algorithm enables to compute a $\leq_{\cal P}^{\rm fin}$-compatible enumeration function of type $(A_1)$ by an ${\cal A}_{{\cal L}_{\cal P}}$-machine. By Proposition \ref{EnumSets2}, the enumeration is also possible by a BSS RAM since $1$ is a constant of ${\cal A}_{{\cal L}_{\cal P}}$.
\begin{proposition}[Enumerability of ${\sf Lab}_{\cal P}^{\rm fin}$ over ${\cal A}_{{\cal L}_{\cal P}}$] For any ${\cal P}\in {\sf P}_\sigma$, the set ${\sf Lab}_{\cal P}^{\rm fin}$ is $(1,A_1)$-enumerable by a BSS RAM in ${\sf M}_{{\cal A}_{{\cal L}_{\cal P}}}$.
\end{proposition}

{\bf An algorithm for an enumeration over ${\cal A}_{\bbbn}$.} It is also possible to compute $enum_{\cal P}^{\sf (O)}: \{1\}^\infty\to {\sf Lab}_{\cal P}^{\rm fin}$ by an ${\cal A}_{\bbbn}$-machine. $1$ is here the only constant of ${\cal A}_{\bbbn}$. For this purpose, the algorithm given in Overview \ref{ExamEnumLab1} can be adapted. The symbol $0$ can be replaced by  $\ell_{\cal P}+1$ or the term $(f_1^1)^{\ell_{\cal P}}(c_1^0)$ since  ${\cal L}_{\cal P}\cup\{succ^{\ell_{\cal P}}(1)\}$ is a subset of $\bbbn_+$. Because of ${\sf Lab}_{\cal P}^{\rm fin}\subset \bbbn_+^\infty$, $enum_{\cal P}^{\sf (O)}: \{1\}^\infty\to \bbbn_+^\infty$ is satisfied. Suitable new functions $\hat{\hat\beta}_{\cal P}$, $ \hat{\hat\beta}^+_{\cal P}$, and $\hat{\hat\beta}^-_{\cal P}$ given by

\vspace{0.2cm} \noindent\,\,\, 
$\hat{\hat\beta}_{\cal P}(\ell)=\left\{ \begin{array}{ll}\!\!\beta_{\cal P}(\ell) &\!\!\mbox {\rm if } \ell\in {\cal L}_{\cal P}^{\rm NB},\\
\!\!\ell_{\cal P}+1&\!\!\mbox {\rm if } \ell\in \bbbn_+\setminus {\cal L}_{\cal P}^{\rm NB}\\
\end{array}\right.$ 
 and  \quad $\hat{\hat\beta}_{\cal P}^{\mbmss{o}}(\ell)=\left\{ \begin{array}{ll}\!\!\beta_{\cal P}^{\mbmss{o}}(\ell) &\!\!\mbox {\rm if } \ell\in {\cal L}_{\cal P}^{\rm B},\\
\!\!\ell_{\cal P}+1&\!\!\mbox {\rm if } \ell\in \bbbn_+\setminus {\cal L}_{\cal P}^{\rm B}\\
\end{array}\right.$

\vspace{0.25cm}
\noindent for $\,\mbm{o}\in \{+,-\}$ and $succ^{\ell_{\cal P}}(1)$ can be computed over $(\bbbn_+;1;succ;=)$.  In such a case, $Z_{s_0}:=\! \hat{\hat\beta}_{\cal P}^+(Z_{s_0-1})$ is a pseudo instruction and $\hat{\hat\beta}_{\cal P}^+$ is no operation symbol.

\begin{proposition}[Enumerability of ${\sf Lab}_{\cal P}^{\rm fin}$ over ${\cal A}_{\bbbn}$]For any ${\cal P}\in {\sf P}_\sigma$, the set ${\sf Lab}_{\cal P}^{\rm fin}$ is $(1,A_1)$-enumerable by a BSS RAM in ${\sf M}_{{\cal A}_{\bbbn}}$.
\end{proposition}

The computation over ${\cal A}_{\bbbn}$ can be simulated by executing index instructions.  We use terms such as {\sf index-computed} and {\sf index-enumerated} in order to express that  functions or the elements of sets are the result of executing only index instructions. 

Each of our labels can be stored in any arbitrary index register of an ${\cal A}$-machine over any arbitrary first-order structure ${\cal A}$. The difference between the computation over ${\cal A}_{\bbbn}$ and the execution of index instructions by an ${\cal A}$-machine results solely from the fact that each machine has only a fixed finite number of index registers. An ${\cal A}$-machine ${\cal M}$ has $k_{\cal M}$ index registers. A consequence is that sequences of labels must be encoded when a certain length is reached. One possibility is to convert each complete sequence of labels into a suitable index-computable G\"odel number that can be stored in a single index register (for details, see \cite{GASS20}). Here, we use the G\"odel number $\nu_{\vec\ell}\,$ given by $\nu_{\vec\ell} =2^{\ell_1 }3^{\ell_2} \cdots p_s^{\ell_s}$ for the label path $\vec\ell$ given by $\vec\ell= (\ell_1,\ldots,\ell_s )$. $p_i$ is the $i$-th number ($i\leq s$) in the sequence of all prime numbers $p_1<p_2<\cdots$. And thus, every value of the functions $enum_{{\cal P},{\sf ind}}$ and $enum_{{\cal P},u}$ considered in Overview \ref{EnumForIndices} and Theorem \ref{EnumGoedel} is a tuple and the length of such a tuple is a G\"odel number of a ${\cal P}$-path. 

\vspace{0.1cm}
\begin{overview}[Enumeration of codes of finite ${\cal P}$-paths  by indices]\label{EnumForIndices}

\hfill

\nopagebreak 

\noindent \fbox{\parbox{11.8cm}{
{\small 
\hfill Let ${\cal P}\in {\sf P}_\sigma$, ${\sf Lab}_{\cal P}^{\rm fin}\not=\emptyset$, and $\vec 1^{[n]} \in \{1\}^{n} $ for $1 \in {\cal L}_{\cal P}\subset \bbbn_+$ and $n\geq 1$. 

\hfill Let ${\cal A}\in {\sf Struc}$ and $\vec u^{[j]}=(u,\ldots, u) \in \{u\}^{j} $ for any $u \in U_{\cal A}$ and $j\geq 1$.

\hfill Let $\{p_i\mid i\geq 1\}$ consist of all prime numbers and $p_1<p_2 <p_3<\cdots $.}

\vspace{0.2cm}
\em
Let $\textcolor{blue}{enum_{{\cal P},{\sf ind}}} :U_{\cal A}^ \infty\to\, \bigcup_{u\in U_{\cal A}}\{u\}^\infty$ and 

\vspace{0.1cm}
\qquad $enum_{{\cal P},{\sf ind}}(\vec x) = \vec u^{[2^{\ell_{1} }3^{\ell_{2}} \cdots p_{s}^{\ell_{s}}]} $ 
\hfill if $\vec x\in U_{\cal A}^n$, $n\geq 1$, $u=_{\rm df} x_1$, and 

\hspace{7.28cm}$enum_{\cal P}^{\sf (O)} (\vec 1^{[n]})=(\ell_{1}, \ldots, \ell_{s})$. 

For any $u\in U_{\cal A}$, let $\textcolor{blue}{enum_{{\cal P},u}}: \{u\}^\infty \to \{u\}^\infty $ and 

\vspace{0.1cm}
\qquad $enum_{{\cal P},\textcolor{blue}{u}} (\vec u^{[n]}) =enum_{{\cal P},\textcolor{blue}{\sf ind}} (\vec u^{[n]}) $ \hspace{1.02cm} for $n\geq 1$.
\vspace{0.1cm}}}
\end{overview}

If the function $enum_{{\cal P},{\sf ind}}$ can be computed by means of index instructions, then this means that the  G\"odel numbers $\nu_{\vec\ell}\,$ of all ${\cal P}$-paths $\vec \ell\in {\sf Lab}_{\cal P}^{\rm fin}$ can be {\sf index-enumerated} and the $n$-th number $\nu_{\vec\ell}\,$ is {\sf index-computable} from index $n$.

\begin{theorem}[Index-enumerability of ${\sf Lab}_{\cal P}^{\rm fin}$]\label{EnumGoedel} Let $\sigma$ be any finite signature, ${\cal P}\in {\sf P}_\sigma$, and ${\sf Lab}_{\cal P}^{\rm fin}\not = \emptyset$. Let ${\cal A}$ be any first-order structure of an arbitrary signature. Then, we have the following properties. 

\begin{itemize} \parskip -0.4mm
\item Each G\"odel number $\nu_{\vec\ell}\,$ of a program path $\vec\ell$ in ${\sf Lab}_{\cal P}^{\rm fin}$ can be index-generated and stored into one index register of any arbitrary ${\cal A}$-machine. 
\item The set of all these G\"odel numbers is index-enumerable by a  BSS RAM over ${\cal A}$ that computes $enum_{{\cal P}, {\sf ind}} :U_{\cal A}^ \infty\to \bigcup_{u\in U_{\cal A}}\{u\}^\infty$. 
\item Each  G\"odel number $\nu_{\vec\ell}\,$ of a finite ${\cal P}$-path $\vec \ell\in {\cal L}_{\cal P}^s$ can be evaluated by a BSS RAM over ${\cal A}$. For $i\leq s$, the single component $\ell_i$ of the path can be computed and assigned to an index register. \end{itemize}
\end{theorem}

{\bf Proof.} The algorithm for the enumeration of the finite ${\cal P}$-paths over ${\cal A}_{{\cal L}_{\cal P}}$  (given in Overview \ref{ExamEnumLab1}) and ${\cal A}_{\bbbn}$, respectively,  can be simulated by a deterministic or non-deterministic BSS RAM  without using $Z$-registers. The G\"odel number of the $n$-th path is index-computable from index $n$ by a BSS RAM over ${\cal A}$ if this machine is equipped with a suitable number of index registers. Useful pseudo instructions for computing or  decoding a G\"odel number are detailed in \cite{GASS20}.\,\qed

\vspace{0.15cm}
\noindent For $u \in U_{\cal A}$, we call $enum_{{\cal P},u}$ a {\em  G\"odel-enumeration function for ${\sf Lab}_{\cal P}^{\rm fin}$}. Alternatively, we can use codes such as $\ell_1 b^0 +\cdots+ \ell_sb^{s-1}$ with $b=|{\cal L}_{\cal P}|+1$  for $\vec \ell\in {\cal L}_{\cal P}^s$ or the restricted Cantor pairing function for defining codes, and the like.

\vspace{0.05cm}
\subsection{Infinite paths and a total order for all paths}  Let us  call a set {\em an uncountable set} if it is not countable. 
The following example shows that there is a $\sigma$-program ${\cal P}$ such that the set ${\sf Lab}_{\cal P}$ of ${\cal P}$-paths is not countable  (and thus not enumerable).  Moreover, we can show that there is an uncountable set of  ${\cal P}$-paths that can be traversed by inputs of a machine ${\cal M}$ whose  input space ${\sf I}_{\cal M}$ is uncountable. 

\begin{example} [Uncountable sets\, of\, computation paths]\label{NichtabzPfade} \hfill Let $\,\,\bbbr_{\sf Ex1}\,$ be\linebreak the structure $(\bbbr;0,1,2,10;-,\cdot\,;\geq)$. An uncountable set ${\sf Lab}_{\cal M} $ of computation paths can be generated, for instance, by the transition system $({\sf S}_{\cal M},\to_{\cal M})$ of the BSS RAM ${\cal M}$ over $\bbbr_{\sf Ex1}$ that uses the following $(4;2,2;2)$-program ${\cal P}$. 

\vspace{0.4cm}
{\sf\small 
\noindent
\begin{tabular}{l}
\,\, Input $\vec x\in \bbbr^\infty.$\vspace{0.05cm}\\
\,\, $1:$ if $r_1^2 (Z_1, c_2^0)$ then goto $8$ else goto $2;$\\
\,\, $2:$ if $r_1^2 (c_1^0,Z_1)$ then goto $8$ else goto $3;$\\
\,\, $3:$ $Z_1:= f_2^2(c_4^0, Z_1);$\\
\,\, $4:$ if $r_1^2(Z_1, c_2^0)$ then goto $5$ else goto $3;$\\
\,\, $5:$ if $r_1^2 (Z_1, c_3^0)$ then goto $8$ else goto $6;$\\
\,\, $6:$ $Z_1:= f_1^2(Z_1,c_2^0);$\\
\,\, $7:$ goto $3;$\\
\,\, $8:$ stop.\vspace{0.05cm}\\
\,\, Output $c(Z_1)$.\\
\end{tabular}
 \begin{tabular}{l}
\,\, Input $\vec x\in \bbbr^\infty.$\vspace{0.05cm}\\
\,\, $1:$ if ($Z_1 \geq 1$) then goto $8$ else goto $2;$\\
\,\, $2:$ if ($0\geq Z_1$) then goto $8$ else goto $3;$\\
\,\, $3:$ $Z_1:= 10* Z_1;$\\
\,\, $4:$ if ($Z_1 \geq 1$) then goto $5$ else goto $3;$\\
\,\, $5:$ if ($Z_1 \geq 2$) then goto $8$ else goto $6;$\\
\,\, $6:$ $Z_1:= Z_1-1;$\\
\,\, $7:$ goto $3;$\\
\,\, $8:$ stop.\vspace{0.05cm}\\
\,\, Output $c(Z_1)$.\\
\end{tabular}}

\vspace{0.4cm}

\noindent
In the left column of the presentation, we find the formal version of ${\cal P}$ written by using the symbols derived directly from the signature. The symbols used in the right column help to better understand the algorithm when it should be executed only over $\bbbr_{\sf Ex1}$. The symbol $*$ stands for the multiplication also denoted by $\cdot$. For simplicity, we moreover use some pseudo instructions. For example, the pseudo instruction $Z_1:= f_2^2(c_4^0, Z_1)$ can stand for the subprogram that consists of the instructions $Z_2:=c_4^0$ and $Z_1:= f_2^2(Z_2, Z_1)$. 
\end{example}

A program path can be a computation path of a BSS RAM over ${\cal A}$ only if there is an input in the input space $U_{\cal A}^\infty$ traversing this path. Whenever $U_{\cal A}$ is countable, this implies the countability of the sets $U_{\cal A}, U_{\cal A}^2, U_{\cal A}^3,\ldots, U_{\cal A}^\infty$, and $\{(\vec x\,.\, \vec y) \mid (\vec x, \vec y) \in U_{\cal A}^\infty\times U_{\cal A}^\infty\}$ (by Cantor), and other sets.

\begin{proposition}[Countability of \,$U_{\cal A}$\, implies a countable $\,{\sf Lab}_{\cal M}$]\hfill 
Let ${\cal A}$ be any first-order structure with a countable universe $U_{\cal A}$. Then, for any BSS RAM ${\cal M}\in {\sf M}_{\cal A}\cup {\sf M}_{\cal A}^{\rm DND} \cup {\sf M}_{\cal A}^{\rm ND}$, the sets ${\sf Lab}_{\cal M}$ and ${\sf Lab}_{\cal M}^{[\leq \omega]}$ are countable sets. 
\end{proposition}
The same does not hold if an input can traverse an uncountable number of computation paths via non-deterministic branching steps. 

\begin{proposition}[BSS RAMs in ${\sf M}_{\cal A}^{\rm NDB}$ with uncountable ${\sf Lab}_{\cal M}$]\hfill 
For any first-order structure ${\cal A}$, there is a BSS RAM ${\cal M}\in{\sf M}_{\cal A}^{\rm NDB}$ such that the sets ${\sf Lab}_{\cal M}$ and ${\sf Lab}_{\cal M}^{[\leq \omega]}$ are not countable. 
\end{proposition}

{\bf Proof.} If the program ${\cal P}_{\cal M}$ of an ${\cal A}$-machine ${\cal M} \in {\sf M}_{\cal A}^{\rm NDB}$ is the following program, then ${\sf Lab}_{{\cal P}_{\cal M}}^{\rm infin}$ and ${\sf Lab}_{\cal M}$ are not countable. 

\vspace{0.1cm}
\quad $ 1:$ {\sf goto $1$ or goto $2;$} \,\, $ 2:$ {\sf goto $1$ or goto $2;$} \,\, $3:$ {\sf stop}. \qed

\vspace{0.3cm}
\noindent Nevertheless, all label paths can be ordered.

\vspace{0.1cm}

{\bf The binary relations $\leq _{{\cal P}\!,[\leq \omega]} ^{\sf lexic}$ and $\leq _{\cal P} ^{\sf lexic}$.} For any $\sigma$-program ${\cal P}$, we can use a lexicographic order for ordering all infinite label paths in  ${\sf Lab}_{\cal P}^{[\leq \omega]}$ in any case and regardless of whether ${\sf Lab}_{\cal P}$ is countable or not.  We denote it by $\leq _{{\cal P}\!,[\leq \omega]} ^{\sf lexic}$ and introduce moreover an order $\leq _{\cal P} ^{\sf lexic}$ for ordering the set ${\sf Lab}_{\cal P}$ of all program paths. For all $B_1,B_2\in {\sf Lab}_{\cal P}$, let $B_1\leq _{\cal P} ^{\sf lexic}B_2$ hold if and only if  $B_1^{[\leq \omega]}\leq _{{\cal P}\!,[\leq \omega]} ^{\sf lexic} B_2^{[\leq \omega]}$ holds. For the notations, see Overview \ref{SetsLabelsProgrPath} and fn.\,\,\ref{Footn}. By the last remark in Overview \ref{SetsLabelsProgrPath}, we obtain a total order on ${\sf Lab}_{\cal P}$ that is lexicographic in a broader sense. An isomorphic embedding  of $({\sf Lab}_{\cal P}, \leq_{\cal P} ^{\sf lexic})$ into $(\bbbr,\leq)$ is possible.

\vspace{0.1cm}

A further very useful relation $\leq_{\cal P}$ can be defined without using an ordered alphabet. By Proposition \ref{NoInitIsPath}, the use of the $\beta_{\cal P}^{\mbmss{o}}$-functions of $\mathfrak{Fl\,}_{\cal P}^{\sf label}$ with $\mbm{o}\in \{+,-\}$ is sufficient for introducing $\leq_{\cal P}$. 

\newpage
\begin{overview}[The binary relation $\leq_{\cal P}$ on ${\sf Lab}_{\cal P}$]\label{OrderAnyPaths}

\hfill \vspace{0.01cm}

\nopagebreak 
\noindent \fbox{\parbox{11.8cm}{\hfill{\small Let ${\cal P}\in {\sf P}_{\sigma}$, $\beta_{\cal P}^+$ and $\beta_{\cal P}^-$ belong to $\mathfrak{Fl\,}_{\cal P}^{\sf label}$, 
 and $1<\cdots <\ell_{\cal P}$ if $1<\ell_{\cal P}$.}

\vspace{0.2cm}

\em
For any ${\cal P}$-paths $B^{(1)}$ and $B^{(2)}$ in ${\sf Lab}_{\cal P}$ given by 

\hspace{1.9cm} $B^{(1)}=(\ell_{1,1},\ell_{1,2},\ldots\big[,\ell_{1,s_1}\big])\in {\cal L}_{\cal P}^\infty \cup{\cal L}_{\cal P}^\omega$ and 

\hspace{1.9cm} $B^{(2)} =(\ell_{2,1},\ell_{2,2},\ldots\big[,\ell_{2,s_2}\big])\in {\cal L}_{\cal P}^\infty \cup{\cal L}_{\cal P}^\omega$,

\vspace{0.2cm}
\noindent let $B^{(1)}\textcolor{red}{ \,<_{\cal P} \,} B^{(2)}$ be valid if there is a $j$ in $\{ 2,3, \ldots\}$ satisfying {\sf (T0)} and {\sf (T1)}.

\vspace{0.2cm}
\qquad \begin{tabular}{ll}

{\sf (T0)}& If $\{B^{(1)}, B^{(2)}\}\cap {\cal L}_{\cal P}^\infty \not=\emptyset$, then \\

&\hspace{1.2cm}$ j\leq \min\{s_i \mid i\in\{0,1\} \,\,\&\,\, B^{(i)}\!\in\!{\cal L}_{\cal P}^\infty \}$.\vspace{0.1cm}\\

{\sf (T1)}& {\sf (T1a)} to {\sf (T1d)} hold. \vspace{0.05cm} \\

& {\sf (T1a)} \, $\ell_{1,i}=\ell_{2,i}$ for all $i \in \{1,\ldots, j-1\}$,\\
& {\sf (T1b)} \, $\ell_{1,j} \not= \ell_{2,j}$, \\
& {\sf (T1c)} \, $\ell_{1,j}=\beta_{\cal P}^+(\ell_{1,j-1})$, \\
& {\sf (T1d)} \, $\ell_{2,j}=\beta_{\cal P}^-(\ell_{1,j-1})$ \hspace{2.2cm} \textcolor{gray}{\small \em (where $\ell_{1,j-1}=\ell_{2,j-1}$)}.\vspace{0.1cm}\\
\end{tabular}

\vspace{0.2cm}
Let $B^{(1)} \textcolor{red}{ \,\leq_{\cal P}\,} B^{(2)} $ be true if $B^{(1)} \textcolor{blue}{\,<_{\cal P}\,} B^{(2)} $ or $B^{(1)}\textcolor{blue}{\,=\,} B^{(2)} $ holds.
 {\small

\em \hfill {\sf T} is derived from {\sf total}.}
}}
\end{overview}

\begin{proposition}[The strict order $<_{\cal P}$ and the total order $\leq_{\cal P}$]\label{PropStand2}\hfill The\linebreak binary relation $<_{\cal P}$ is irreflexive and a strict order. The relationships $B'\not <_{\cal P} B'' $ and $B''\not <_{\cal P} B' $ mean $B'= B'' $. $\leq_{\cal P}$ is a total (or linear) order. It is reflexive, antisymmetric, transitive, and total. \end{proposition}

{\bf The standard order $\leq_{\cal P}$.} We call $\leq_{\cal P}$ {\em the standard order on ${\sf Lab}_{\cal P}$}. 
The real numbers in decimal form $0.\,j_1j_2\cdots\!\big[j_s00\cdots\big]$ with digits  in $ \{0,1,2\}$ can be used as  codes for  ${\cal P}$-paths $(\ell_1,\ell_2, \ldots\!\big[,\ell_s\big])$.  $j_i=1$ can stand for $\ell_{i+1}\in \{\beta_{\cal P}(\ell_i), \beta_{\cal P}^+(\ell_i)\}$ and  $j_i=2$ for  $\ell_{i+1}=\beta_{\cal P}^-(\ell_i)$ \big[if $i< s$\big].  $j_s=0$ implies $\ell_s=\ell_{\cal P}$.
 
\vspace{0.05cm}
\begin{theorem}[Total order on ${\sf Lab}_{\cal P}$]\label{TotOrderLabel}For any ${\cal P}\in {\sf P}_\sigma$ and the standard order $\leq_{\cal P}$ on ${\sf Lab}_{\cal P}$, $({\sf Lab}_{\cal P}, \leq_{\cal P})$ can be isomorphically embedded into $(\bbbr,\leq)$.
\end{theorem}

{\bf The total order $\leq_{\cal P}^{[\leq \omega]}$.} For ${\cal P}\in {\sf P}_\sigma$ and all label paths $B'$ and $B''$ in ${\sf Lab}_{\cal P}^{[\leq \omega]}$, let $B' \textcolor{red}{ \,<_{\cal P}^{[\leq \omega]}\,}B''$ be defined by means of {\sf (T1)}~-- as given in Overview \ref{OrderAnyPaths} for $<_{\cal P}$~-- without any restriction of $j$ by {\sf (T0)}. This means that $j$ can now be any integer in $\bbbn_+\setminus \{1\}$ since all sequences in ${\sf Lab}_{\cal P}^{[\leq \omega]}$ are infinite. Let $B' \textcolor{red}{ \,\leq_{\cal P}^{[\leq \omega]}\,} B'' $ be true if $B' \textcolor{blue}{\,<_{\cal P}^{[\leq \omega]}\,} B''$ or $B'\textcolor{blue}{\,=\,} B'' $ holds. Consequently, for all ${\cal P}$-paths $B_1$ and $B_2$, we have $B_1\leq _{\cal P}B_2 \mbox{ if and only if }B_1^{[\leq \omega]}\leq _{\cal P} ^{[\leq \omega]} B_2^{[\leq \omega]}$.

\begin{theorem}[Total order on ${\sf Lab}_{\cal P}^{[\leq \omega]}$]\label{TotOrderOmega}For ${\cal P}\in {\sf P}_\sigma$, $\leq _{\cal P} ^{[\leq \omega]}$ is a total order on ${\sf Lab}_{\cal P}^{[\leq \omega]}$ and $({\sf Lab}_{\cal P}^{[\leq \omega]}, \leq_{\cal P}^{[\leq \omega]})$ can be isomorphically embedded into $(\bbbr,\leq)$.
\end{theorem}

As Example \ref{NichtabzPfade} shows, the question whether there is a well-order on ${\sf Lab}_{\cal P}$ or ${\sf Lab}_{\cal P}^{[\leq \omega]}$ is related to the question whether there is a well-order on the set of real numbers. If the {\sf axiom of choice} is assumed and we consider only models of ZFC, then we have also a well-ordering on ${\sf Lab}_{\cal P}$ as Ernst Zermelo \cite{Zerm04} proved.

However, for any ${\cal P}\in {\sf P}_\sigma$ and any $s\geq 1$, ${\cal L}_{\cal P}^s$ is finite.  Thus, all finite ${\cal P}$-paths can be length-compatibly ordered and, for each fixed length $s$, enumerated in accordance with the well-order $\leq_{\cal P}\cap ({\sf Lab}_{\cal P}^{\rm fin}\cap {\cal L}_{\cal P}^s)^2$. 
An algorithm for computing  a suitable enumeration function \textcolor{blue}{$enum_{{\cal P}}^{\sf (\leq, T)} \!:\!\{1\}^\infty\!\to {\sf Lab}_{\cal P}^{\rm fin}$} is a little more complicated than the algorithm in Overview \ref{ExamEnumLab1} and easier to understand if we describe a possible execution of this algorithm by using a 2-tape machine over ${\cal A}_{{\cal L}_{\cal P}}$ (see Overview \ref{ExamEnumLab2}). In $enum_{{\cal P}}^{\sf (\leq, T)}$, the symbol $\leq$ stands for {\sf length-compatible} and the symbol ${\sf T}$ after the symbol $\leq$ for the {\sf restricted $\leq_{\cal P}$-compatibility}. $\leq$ has a higher priority. In this way, a total order $\leq_{\cal P}$ resulting from {\sf (T0)} and {\sf (T1)} leads to a well-ordered sequence of all finite ${\cal P}$-paths 

\vspace{0.1cm}
\qquad $(enum_{{\cal P}}^{\sf (\leq, T)}(\vec 1^{[1]}), enum_{{\cal P}}^{\sf (\leq, T)}(\vec 1^{[2]}), enum_{{\cal P}}^{\sf (\leq, T)}(\vec 1^{[3]}), \ldots) \in ({\sf Lab}_{\cal P}^{\rm fin})^\omega$. 

\vspace{0.2cm}
{\bf An enumeration of finite paths using $enum_{{\cal P}}^{\sf (\leq, T)}$.} For any ${\cal P}\in {\sf P}_\sigma$, the generation of the first $n$ finite ${\cal P}$-paths can be realized by executing a 2-tape $(\ell_{\cal P}+1;1,1,1;2)$-program by a 2-tape machine in ${\sf M}^{(2)}_{{\cal A}_{{\cal L}_{\cal P}}}$. The first tape is used for storing the labels of sequences in ${\cal L}_{\cal P}^\infty$. The second tape can be used to mark backtracking points. All $Z$-registers of tape 1 can again be recursively processed by using loops that allow to repeat the execution of subprograms as long as necessary. For the input $\vec 1^{[n]}$, the initial value of $I_{1,1}$ is $n$ and the other index registers and the $Z$-registers receive the value $1$ at the beginning. During the search for finding the $n$-th path, $I_{1,2}$ and $I_{2,1}$ obtain the values $2$ to ${s-1}$ or $s$ for $s=2,3,\ldots$ step by step. A further register, $I_{1,3}$, obtains the value $c(I_{1,2})-1$ for executing the subprograms $(sub\,1)$ to $(sub\,3)$. The labels of a path can be assigned to $Z_{1,1},\ldots,Z_{1,s}$  using pseudo instructions given by $Z_{1,k}:=\hat \beta_{\cal P}(Z_{1,k-1})$, $Z_{1,k}:=\hat \beta_{\cal P}^+(Z_{1,k-1})$, or $Z_{1,k}:=\hat \beta_{\cal P}^-(Z_{1,k-1})$ for $k\in \{2,\ldots,s-1\}$ and $Z_{1,s}:=\ell_{\cal P}$.

\vspace{0.2cm}
{\bf Backtracking.} 
For finding the next ${\cal P}$-path, a backtracking procedure is used. Certain positions on the tapes can be marked as backtracking points and thus reached by backtracking in order to have the possibility to assign new values to the last used registers $Z_{1,s_0}, \ldots, Z_{1,s-1}$. To save notes about the backtracking points, we can use the second tape. The information stored on tape 2 helps to find and evaluate the backtracking points that are marked with a 0. These points must be recursively processed step by step from a marked register with a higher index to the registers with lower  index numbers. After generating $c(Z_{1,1})$ to $c(Z_{1,s-1})$ or  to $\ell_{\cal P}$, the  registers $Z_{1,s}$ to $Z_{1,s_0}$ can obtain new values if $Z_{1,s_0}$ is marked as backtracking point by $c(Z_{2,s_0})=0$ and $c(Z_{2,s_0+1})=\cdots =c(Z_{2,s})=1$ holds. All registers from the last register $Z_{1,s}$ to the register $Z_{1,s_0}$ can be changed.  Since we are assuming that $\hat \beta_{\cal P}^+(c(Z_{1,s_0-1}))\not = \hat \beta_{\cal P}^-(c(Z_{1,s_0-1}))$ holds for $c(Z_ {1,s_0-1})\in {\cal L}_{\cal P}^{\rm B}$, $Z_{1,s_0}$ obtains a new value and, for $k>s_0$, $Z_{1,k}$ can obtain new values. 

\vspace{0.1cm}
For any arbitrary subprograms {\sf  inst$_1$}, {\sf  inst$_2$}, and {\sf  inst}, the execution of a subprogram of the form

\vspace{0.05cm}
\qquad{\sf if $cond$ then inst$_1$ else inst$_2;$\, $l_0:$ inst} \quad

\vspace{0.05cm}
\noindent can be realized by executing

\vspace{0.05cm}
\qquad{\sf if $cond$ then goto $\bar l_1$ else goto $\bar l_2;$ \, $\bar l_1:$  inst$_1;$\, goto $l_0;$ \, $\bar l_2:$  inst$_2;$  \, $l_0:$ inst}\,\,\,.

\begin{overview}[An algorithm for computing \textcolor{blue}{$enum_{{\cal P}}^{\sf (\leq, T)}$} over ${\cal A}_{{\cal L}_{\cal P}}$]\label{ExamEnumLab2}

\hfill \vspace{0.01cm}

\nopagebreak 
\noindent \fbox{\parbox{11.8cm}{
\hfill {\small Let ${\cal P}\in {\sf P}_\sigma$, $|{\cal L}_{\cal P}|>1$, $i_0=|{\sf Lab}_{\cal P}^{\rm fin}|$, and $i_0\geq 1$.}

\em 
\hspace{0.62cm}{\sf Input} $\vec 1^{[n]}\in \{1\}^{n}$.

\hspace{0.62cm}${\sf ass}(Z_{1,1});\,$ {\sf if $Z_{1,1}=1$ \& $n>i_0$ then $n:=i_0;$}

{\sf
\hspace{0.62cm}$s:=1;$ $i:=1;$ $Z_{1,1}:=1;$

\vspace{0.1cm}
\noindent \textcolor{magenta}{$l_1$} : 
 $\textcolor{blue}{s_0}:=\max\{j\leq s\mid Z_{2,j}=0 \mbox{ or } j=0\};$

\hspace{0.64cm}if $\textcolor{blue}{s_0}=0$ then \{ $s:=s+1;$  $k:=2$ \} 

\hspace{1.92cm} else \hspace{0.13cm}$k:=\textcolor{blue}{s_0};$ \hfill \textcolor{gray}{\small \it (back to the last backtracking point)}

\vspace{0.1cm}
\noindent \textcolor{magenta}{$l_2$} : \textcolor{brown}{if $k=s$ then goto \textcolor{magenta}{$l_3$}}

\hspace{1.78cm} \textcolor{brown}{ else \{} if $ Z_{1,k-1}\in {\cal L}_{\cal P}^{\rm NB}$

\hspace{2.7cm} then $Z_{1,k}:=\hat \beta_{\cal P}(Z_{1,k-1})$ \hfill {\small $(sub\,1)$}

\hspace{2.7cm} else \hfill \textcolor{gray}{\small \it (this means $ Z_{1,k-1}\in {\cal L}_{\cal P}^{\rm B}$)}
 
\hspace{3.3cm} if $Z_{2,k}=1$ \hfill \textcolor{gray}{\small \it (the backtracking info)}
 
\hspace{3.3cm} then \{ $Z_{1,k}:=\hat \beta_{\cal P}^+(Z_{1,k-1});$ \hfill {\small $(sub\,2)$}

\hspace{4.35cm} $Z_{2,k}:=0$ \}  \hfill \textcolor{gray}{\small \it (a new  info for the next backtracking)}
 
\hspace{3.3cm} else \hspace{0.13cm}\{ $Z_{1,k}:=\hat \beta_{\cal P}^-(Z_{1,k-1});$ \hfill  {\small $(sub\,3)$}

\hspace{4.35cm} $Z_{2,k}:=1$ \}$;$ \hfill \textcolor{gray}{\small \it (delete the backtracking info $0$)}

\hspace{2.7cm} if $Z_{1,k}=\ell_{\cal P}$ then   goto \textcolor{magenta}{$l_1$}$;$ \hfill \textcolor{gray}{\small \it (not enough labels)}

\hspace{2.7cm} $k:=k+1;$ goto \textcolor{magenta}{$l_2$} \textcolor{brown}{\}}$;$ 

\vspace{0.15cm}
\textcolor{magenta}{$l_3$} :  if $\ell_{\cal P} \in \{\hat \beta_{\cal P}(Z_{1,s-1}), \hat \beta_{\cal P}^+(Z_{1,s-1}), \hat \beta_{\cal P}^-(Z_{1,s-1})\}$ 

\hspace{0.64cm}then \textcolor{red}{\{} if $i=n$ then \{ $Z_{1,s}:=\ell_{\cal P};$ goto \textcolor{magenta}{$l_4$} \} else $i:=i+1$ \textcolor{red}{\}}$;$ 

\hspace{0.64cm}goto \textcolor{magenta}{$l_1$}$;$}

\vspace{0.1cm}
\noindent \textcolor{magenta}{$l_4$} : {\sf Output} $(c(Z_{1,1}),\ldots,c(Z_{1,s}))$.  

 \hfill {\small \em  \textcolor{gray}{An example  is attached (cf.\,\,p.\,\pageref{ExampleSecondAlgo}).}}

}}\end{overview}

\subsection{The meaning of standard orders}
Well-orders on sets of finite program paths and, in particular, the enumeration functions to enumerate the paths in these sets are good prerequisites for a systematic analysis of the behavior of machines over $\sigma$-structures. For evaluating the behavior of machines, length-compatible orders such as $\leq _{\cal P}^{\rm fin}$ and the discussed enumerations offer significant advantages. They are useful, for instance, when the execution of a non-deterministic machine ${\cal M}$ along single finite ${\cal P}_{\cal M}$-paths should be simulated by a deterministic machine as described in \cite{GASS25C}, etc. 
 
\vspace{0.2cm}
{\bf Graphical representation of program paths in the 2D plane.} The ${\cal P}$-paths of a $\sigma$-program ${\cal P}$ can be ordered and meaningfully grouped and represented by suitable graph-theoretical labeled trees (for the definition and more details, see {\sf Extras II}). In a 2D plane, the root of such a tree is generally at the top, represents the first label of the ${\cal P}$-paths and can, thus, be labeled by 1. The leaves are further down. By Proposition \ref{NoInitIsPath}, each finite ${\cal P}$-path can be uniquely represented by one walk from the root to a certain leaf in such a tree. Thus, the leaves of this kind can be labeled by $\ell_{\cal P}$ and could also be a representative for identifying exactly one of the finite ${\cal P}$-paths. 

Many trees are only generated layer by layer. In this way, we also obtain finite decision trees. Then, it is generally not known whether a walk starting at the root is incomplete and represents only an initial  segment of a finite ${\cal P}$-path and which length a complete ${\cal P}$-path could have if it contains such an initial part. The walks representing the finite ${\cal P}$-paths in $\bigcup _{s\leq r}{\cal L}_{\cal P}^s$ and the initial segments $(\ell_1,\ldots, \ell_r)$ of the other ${\cal P}$-paths~-- and just like the leaves of all these walks~-- can be ordered from the left to the right using an order such as $<_{\cal P}$, and thus the graphical representation of this order can be used to illustrate $\leq_{\cal P}$-compatible enumerations of the represented label paths. These trees can help, among other things, to evaluate the execution of programs by machines. They are suitable in particular for the derivation of finite systems of $\sigma$-literals from branching conditions that determine the branching behavior for inputs of a fixed length. Each of these conditions  results from passing a branching node and executing instructions including the corresponding branching instruction  of type $(4)$, and its truth value determines the next label to be selected. The branching behavior along a ${\cal P}$-path can be uniquely described by $\sigma$-terms in a standard way and thus by $\sigma$-literals. Certain systems of $\sigma$-literals, in turn, enables a characterization of the inputs of a given length traversing an initial segment of one of the considered and described ${\cal P}$-paths. 

Alternatively, for representing finite ${\cal P}$-paths, trees can be introduced by using $<_{\cal P}^{\rm fin}$. The order $<_{\cal P}^{\rm fin}$ has useful properties, but it is not really intended to define the usual decision trees known from computational geometry, etc. Unless we consider programs for machines over structures  whose family of relations is closed with respect to  complementation and inputs of a fixed length.
\begin{remark}We use also other representations of decision trees in figures since the space for figures is limited. In such cases, the order of the leaves does not correspond to $\leq_{\cal P}$ or $\leq_{\cal P}^{\rm fin}$. An example is attached below (see Figure \ref{BerWurzelnProgrammpfade}).
\end{remark}

\vspace{0.1cm}
{\bf An analysis  requires more information and  suitable orders.} Let ${\cal P}\in {\sf P}_\sigma$, $\kappa\geq k_{\cal P}$, and $n\geq 1$.
The paths in ${\sf Con}_{{\cal P},\kappa,n}^{\rm ind}$ and in particular in the sets ${\sf Con}_{{\cal P},\kappa,n}^{{\rm ind},[\omega>]}$ defined below are important for the evaluation of machines. 
The transformation of  ${\cal P}$-paths into sequences of partial $({\cal P},\kappa)$-configurations~-- as far as possible~-- can lead to fewer paths  in ${\sf Con}_{{\cal P},\kappa,n}^{\rm ind}$ since certain branching conditions  result in fewer jumps for all suitable machines  on inputs of length $n$. For $\ell \in {\cal L}_{{\cal P},{\rm B}}$, the branching process is unconditional and thus unrestricted possible. For $\ell \in {\cal L}_{{\cal P},{\rm H}_{\rm T}}$, the branching  depends on $n$ and is uniquely determined. The other branching processes depend primarily on the used ${\rm T}$-instructions and  the underlying structures. 

\vspace{0.2cm}
{\bf The sets ${\sf Con}_{{\cal P},\kappa,n}^{{\rm ind},[\omega>]}$ of $({\cal P},\kappa,n)^{[\omega>]}$-paths.}
An initial part $(\ell_t\,.\,\vec \nu ^{(t)})_{t=1..s}$ of a $({\cal P},\kappa,n)$-path in  ${\sf Con}_{{\cal P},\kappa,n}^{\rm ind}$  for which  there is  a path $\bar \ell$  in ${\sf Lab}_{\cal P}^{[\leq \omega]}$ satisfying   $\bar \ell= (\ell_t)_{t\geq 1}$ and    $B_{\bar \ell}=(\ell_1,\ldots,\ell_s)\in {\sf Lab}_{\cal P}^{\rm fin}$ is called a {\em $({\cal P},\kappa,n)^{[\omega>]}$-path}. Let ${\sf Con}_{{\cal P},\kappa,n}^{{\rm ind},[\omega>]}$ consist of all $({\cal P},\kappa,n)^{[\omega>]}$-paths.\,\footnote{ The superscript $^{[\omega>]}$ after the symbol $N$ for a set states that we come from sequences of length $ \omega$ in $ N^\omega$ and go back to finite sequences in $N^\infty$. Here, we derive ${\sf Con}_{{\cal P},\kappa,n}^{{\rm ind},[\omega>]}$ from ${\sf Con}_{{\cal P},\kappa,n}^{\rm ind}$. We consider the paths in ${\sf Con}_{{\cal P},\kappa,n}^{\rm ind}$ and ${\sf Seq}_{{\cal P},\kappa,n}^{\rm ind}$ and then the initial parts in sets such as ${\sf Con}_{{\cal P},\kappa,n}^{{\rm ind},[\omega>]}$  and ${\sf Seq}_{{\cal P},\kappa,n}^{{\rm ind},[\omega>]}$  consisting of finite paths $(\ell_t\,.\,\vec \nu^{(t)})_{t\leq s}$ for which $(\ell_t)_{t\leq s}$ are ${\cal P}$-paths in ${\sf Lab}_{\cal P}^{\rm fin}$.}  Let ${\sf Seq}_{{\cal P},\kappa,n}^{{\rm ind},[\omega>]}$ be defined in the same way and derived from ${\sf Seq}_{{\cal P},\kappa,n}^{\rm ind}$ (cf.\,\,Overview \ref{SetsPartCon2}).

\vspace{0.2cm}
Now, we want to derive new orders from $\leq_{\cal P}^{\rm fin}$ and $\leq_{\cal P}$ (cf.\,\,Overviews \ref{OrderOnFinP} and \ref{OrderAnyPaths}). For any $\kappa\geq k_{\cal P}$ and $n\geq 1$,  $<_{\cal P}^{\rm fin}$ can be transformed into a relation $<_{{\cal P},\kappa,n}^{\rm fin}$ on $({\sf S}_{{\cal P},\kappa}^{\rm ind})^\infty$  by converting each label  into a  member $(\ell\,.\,\vec \nu)$ of a sequence such that, in turn, $<_{\cal P}^{\rm fin}$ could be derived from $<_{{\cal P},\kappa,n}^{\rm fin}$ if ${\sf Seq}_{{\cal P},\kappa,n}^{{\rm ind},[\omega>]}$ is projected onto ${\sf Lab}_{\cal P}^{\rm fin}$ by mapping the sequences in ${\sf Seq}_{{\cal P},\kappa,n}^{{\rm ind},[\omega>]}$ to paths in  ${\sf Lab}_{\cal P}^{\rm fin}$ and  each member $(\ell\,.\,\vec \nu)$ of a sequence to $\ell$. A subset  of \textcolor{red}{$<_{{\cal P},\kappa,n}^{\rm fin}$}  can then be used to order ${\sf Con}_{{\cal P},\kappa,n}^{{\rm ind},[\omega>]}$. $<_{\cal P}$ can be transformed into an order $<_{{\cal P},\kappa,n}$ on $({\sf S}_{{\cal P},\kappa}^{\rm ind})^\omega\cup ({\sf S}_{{\cal P},\kappa}^{\rm ind})^\infty$ and used to define a subset \textcolor{red}{$<_{{\cal P},\kappa,n}^{[\omega>]}$} of $<_{{\cal P},\kappa,n}$ for ordering ${\sf Seq}_{{\cal P},\kappa,n}^{{\rm ind},[\omega>]}$. 

\vspace{0.3cm}
{\bf Standard orders $\leq_{{\cal P},\kappa,n}^{[\omega >]}$ and $<_{{\cal P},\kappa,n}^{\rm fin}$.} 
Let $\vec \ell$ and $\vec \ell'$ be any finite ${\cal P}$-paths $(\ell_t)_{t\leq s}$ and $(\ell_t')_{t\leq s'}$, respectively, in ${\sf Lab}_{\cal P}^{\rm fin}$ and let $(\ell_t\,.\,\vec \nu ^{(t)})_{t\leq s}$ and $(\ell_t'\,.\,\vec \mu ^{(t)})_{t\leq s'}$ be sequences that belong to $ {\sf Seq}_{{\cal P},\kappa,n}^{{\rm ind},[\omega>]}$.

\vspace{0.1cm}
Let $(\ell_t\,.\,\vec \nu ^{(t)})_{t\leq s} \textcolor{red}{\,<_{{\cal P},\kappa,n}^{\rm fin}\,} (\ell_t'\,.\,\vec \mu ^{(t)})_{t\leq s'}$
be satisfied if $\vec \ell \textcolor{blue}{\,<_{\cal P}^{\rm fin}\, } \vec \ell'$ holds. 

\vspace{0.1cm}
Let  $(\ell_t\,.\,\vec \nu ^{(t)})_{t\leq s} \textcolor{red}{\,<_{{\cal P},\kappa,n}^{[\omega >]}\,} (\ell_t'\,.\,\vec \mu ^{(t)})_{t\leq s'}$
be satisfied if $\vec \ell \textcolor{blue}{\,<_{\cal P}\,} \hspace{0.09cm} \vec \ell' $ holds. 

\vspace{0.15cm}
\noindent For sequences ${\rm seq}$ and ${\rm seq}'$ in ${\sf Seq}_{{\cal P},\kappa,n}^{{\rm ind},[\omega>]}$, let ${\rm seq} \textcolor{red}{\,\leq_{{\cal P},\kappa,n}^{[\omega >]}\,} {\rm seq}'$ hold if ${\rm seq} <_{{\cal P},\kappa,n}^{[\omega >]} {\rm seq}'$ or ${\rm seq} = {\rm seq}'$ holds. Let ${\rm seq} \textcolor{red}{\,\leq_{{\cal P},\kappa,n}^{\rm fin}\,} {\rm seq}'$ hold if ${\rm seq} <_{{\cal P},\kappa,n}^{\rm fin} {\rm seq}'$ or ${\rm seq} = {\rm seq}'$ holds. The definition of ${\sf Con}_{{\cal P},\kappa,n}^{\rm ind}$ and the paths in this set was inspired by and is based on statements such as Theorem \ref{TheoUniq1} and has led to Theorem \ref{TheoUniq2}. If  $\vec \ell \in {\sf Lab}_{\cal P} ^{\rm fin}\cap{\cal L}_{\cal P}^s$ holds, then there is one and only one  sequence $(\vec \nu^{(t)})_{t=1..s}\in (\bbbn_+^{\kappa})^s$ such that $(\ell_t\,.\,\vec\nu^{(t)})_{t=1..s}$  belongs to ${\sf Seq}_{{\cal P},\kappa,n}^{{\rm ind},[\omega>]}$. Consequently, for any $\sigma$-program ${\cal P}$, the algorithms given in Overviews \ref{ExamEnumLab1} and \ref{ExamEnumLab2} can be easily adapted in order to enumerate  the sequences in  ${\sf Seq}_{{\cal P},\kappa,n}^{{\rm ind},[\omega>]}$  or only the $({\cal P},\kappa,n)^{[\omega>]}$-paths in ${\sf Con}_{{\cal P},\kappa,n}^{{\rm ind},[\omega>]}$ over ${\cal A}_\bbbn$ or  index-enumerate the G\"odel numbers of these sequences of partial $({\cal P},\kappa)$-configurations. The adapted algorithms allow to compute the single members $({\ell\,.\,\vec \nu})$ of the $i$-th sequence step by step in exactly the same way as the algorithms described in Overview \ref{ExamEnumLab1}  and Overview \ref{ExamEnumLab2}  allow to generate the labels $\ell$ of the $i$-th finite ${\cal P}$-path.
The adapted programs start with the partial configuration $(1\,.\,(n,1,\ldots,1))\in \bbbn_+^{\kappa+1}$ and not just with label 1.

The enumeration functions  for $\!{\sf Seq}_{{\cal P},\kappa,n}^{{\rm ind},[\omega>]}$ computed  over ${\cal A}_\bbbn$ by executing the adapted programs are $\leq_{{\cal P},\kappa,n}^{\rm fin}$-compatible  and  functions whose $\leq_{{\cal P},\kappa,n}^{[\omega >]}$-compatibility is restricted to sequences of the same length, respectively.  If we compose
\begin{itemize} [label={--}]\parskip -1.5mm
\item  a decoding function assigning  the sequence  $((\ell_{1}\,.\,\vec\nu^{(1)}), \ldots, (\ell_{s}\,.\,\vec\nu^{(s)}))$ in ${\sf Seq}_{{\cal P},\kappa,n}^{{\rm ind},[\omega>]}$  to the tuple $1^{[2^{\ell_{1} }3^{\ell_{2}} \cdots p_{s}^{\ell_{s}}]}$  whose length is the  index-computable   G\"odel number $2^{\ell_{1} }3^{\ell_{2}} \cdots p_{s}^{\ell_{s}}$

\vspace{0.1cm}
\noindent $\!\!\!\!\!\!$and

\item  a G\"odel-enumeration function  for ${\sf Lab}_{\cal P}^{\rm fin}$  resulting from applying one of   
 the algorithms in Overviews \ref{ExamEnumLab1} and \ref{ExamEnumLab2}, 
\end{itemize}

\noindent then we obtain enumeration functions for ${\sf Seq}_{{\cal P},\kappa,n}^{{\rm ind},[\omega>]}$  that  are computable over  structures  with ${\cal A}_{\bbbn}$ as a reduct and that are
 $\,\leq_{{\cal P},\kappa,n}^{\rm fin}$-compatible or whose $\,\leq_{{\cal P},\kappa,n}^{[\omega >]}$-compatibility is restricted to sequences of the same length, respectively. 

\vspace{0.3cm}
{\bf Restrictions of ${\sf P}_\sigma$ contribute to standardization of analyses.} 
Each structure ${\cal A}$ can be expanded to a $\sigma$-structure ${\cal A}'$ that is closed under complementation for relations without changing the decision behavior of machines over these structures. We have, for example, ${\rm DEC}_{\cal A}={\rm DEC}_{{\cal A}'}$ even if we replace arbitrary branching instructions of type (4) and (11) in a $\sigma$-program ${\cal P}$ by {\sf label-ordered branching instructions}. These substitutions lead, for every resulting {\sf partially label-ordered $\sigma$-program ${\cal P}^{\rm ord}$}, to the equality of the new standard orders on domains such as ${\sf Con}_{{\cal P}^{\rm ord},\kappa,n}^{{\rm ind},[\omega>]}$. For this reason, we give some more precise definitions and rules for transforming arbitrary $\sigma$-programs into {\sf partially label-ordered $\sigma$-programs}. 

\vspace{0.2cm}
{\bf Partially label-ordered programs.} 
A $\sigma$-program is called a {\em partially label-ordered $\sigma$-program} if all its branching instructions of types (4) and (11)  are  label-ordered. Any branching instruction of ${\rm B}$-type $(cond_\ell,\ell_1,\ell_2)$ or $(\ell_1,\ell_2)$ is {\em label-ordered} if $\ell_1\leq \ell_2$ is satisfied.

\vspace{0.2cm}
{\bf Structures closed under complementation.} 
Let us consider first-order structures that are  closed with respect to the complementation for relations. Then, such a structure ${\cal A}$ contains, for any $k_{i}$-ary relation $r_{i}\subseteq U_{\cal A}^{k_i}$, also its {\em relative  complement} $r_{j}=U_{\cal A}^{k_i}\!\setminus r_{i}$. $r_{j}$ is a $k_{i}$-ary relation which means that $k_{i}=k_{j}$ holds. 

\vspace{0.2cm}
{\bf A suitable signature $\sigma$.} 
If a $\sigma$-structure contains the relations $r_{1},\ldots,r_{n'}$ and their relative complements $r_{n'+1},\ldots,r_{2n'}$, then $\sigma$ contains a family $(k_i)_{i\leq 2n'}$ of arities with $k_i=k_{n'+i}$ for all $i\leq n'$. In the following, let  \textcolor{blue}{$\sigma$}  be a signature of the form  $(n_1; (m_i)_{i\leq n_2}; (k_i)_{i\leq n_3})$ such that $(k_i)_{i\leq n_3}$  is a family of this kind for which $n_3=2n'$ holds. $(k_i)_{i\leq n_3}$ determines the arities  of $n_3$ relations and their symbols $r_1^{k_1},\ldots, r_{n_3}^{k_{n_3}}$  and  the arities satisfy $k_i=k_{n'+i}$ for all $i\leq n'$.  
Since each signature $\sigma_0$  can be extended to such a signature $\sigma$ and each $\sigma_0$-program   is also  a $\sigma$-program, let ${\cal P}$ be any  \textcolor{blue}{$\sigma$}-program.

\vspace{0.2cm}
{\bf From ${\cal P}$  to a partially label-ordered $\sigma$-programs ${\cal P}^{\rm ord}$.}  
For all $\ell \in {\cal L}_{{\cal P},{\rm T}}\cup {\cal L}_{{\cal P},{\rm B}}$, the branching instructions of ${\rm B}$-types $(cond_\ell,\ell_1,\ell_2)$ and $(\ell_1,\ell_2)$ in a $\sigma$-program ${\cal P}$ with $\ell_2<\ell_1$ can be replaced by {\em label-ordered branching instructions} as follows. 
\begin{itemize}[label={--}] \parskip -0.6mm

\item Instructions of the form\hspace{0.2cm}{\sf if $r_{i}^{k_{i}}(Z_{j_1},\ldots, Z_{j_{k_{i}}})$ then goto $\ell _1$ else goto $\ell _2$} 

\vspace{0.05cm}
can be replaced by\hspace{1cm}{\sf  if $r_{n'+i}^{k_{i}}(Z_{j_1},\ldots, Z_{j_{k_{i}}})$ then goto $\ell _2$ else goto $\ell _1$} . 

\item Instructions of the form\hspace{0.2cm}{\sf  goto $\ell_1$ or goto $\ell_2$} 

\vspace{0.05cm}
can be replaced by\hspace{1cm}{\sf goto $\ell_2$ or goto $\ell_1$} . 
\end{itemize}
In this way,  we obtain a  {\em partially label-ordered $\sigma$-program} denoted by ${\cal P}^{\rm ord}$. 
Because of ${\cal P}^{\rm ord}= {\cal P}$ for partially label-ordered $\sigma$-programs ${\cal P}$, we have  the idempotence $({\cal P}^{\rm ord})^{\rm ord}= {\cal P}^{\rm ord}$ for all programs ${\cal P}$. 

\begin{proposition}[Same results for partially label-ordered programs]\hfill
Let ${\cal A}$ be a $\sigma$-structure closed under complementation for relations and $n\geq 1$. Then for any BSS RAM ${\cal M}\in {\sf M}_{\cal A}\cup {\sf M}_{\cal A}^{\rm ND}\cup {\sf M}_{\cal A}^{\rm DND}\cup {\sf M}_{\cal A}^{\rm NDB}$, there is a BSS RAM ${\cal M}'$ over ${\cal A}$ such that ${\cal P}_{{\cal M}'} =({\cal P}_{\cal M})^{\rm ord}$ holds and ${\rm Res}_{\cal M}(\vec x)={\rm Res}_{{\cal M}'}(\vec x)$ holds for all $\vec x\in U_{\cal A}^n$. 
\end{proposition}

Let us consider only finite program paths in $ {\sf Lab}_{{\cal P},n}^{[*]}$ that could be traversed by BSS RAMs on inputs of a fixed length $n$.  For any $n\geq 1$, the successor of a member $(\ell_t\,.\,\vec \nu^{(t)})$ in any $({\cal P},\kappa,n)$- or $({\cal P}^{\rm ord},\kappa,n)$-path is uniquely determined by the function $T_{\ell_t}^{(\kappa)}\in {\cal F}_{{\cal P},\kappa}^{\rm ind}$ if $\ell_t \in {\cal L}_{{\cal P},{\rm H}_{\rm T}}$ holds, and we have $T_{{\cal P}^{\rm ord},\kappa, \ell_t}(\vec \nu^{(t)}) =T_{{\cal P},\kappa, \ell_t}(\vec \nu^{(t)})$.  This means that the order of two paths is determined solely by instructions of types $(4)$ and $(11)$. A replacement of index instructions of ${\rm B}$-type $(cond_\ell,\ell_1,\ell_2)$ is not necessary. 

\begin{theorem}[One order on ${\sf Con}_{{\cal P},\kappa,n}^{{\rm ind},[\omega>]}$]Let $\,{\cal P}$ be a partially label-ordered $\sigma$-program, $n\geq 1$, $\kappa\geq k_{\cal P}$, and $\vec \ell\in {\sf Lab}_{{\cal P},n}^{[*]}\!\cap\! {\sf Lab}_{\cal P} ^{\rm fin}$.  On the set $ {\sf Lab}_{{\cal P},n}^{[*]}\!\cap\! {\sf Lab}_{\cal P} ^{\rm fin}$,  the standard orders $\leq _{\cal P}$ and $\leq _{\cal P}^{\rm fin}$ match and thus the orders $\leq_{{\cal P},\kappa,n}^{[\omega >]}$ and $\leq_{{\cal P},\kappa,n}^{\rm fin}$ on ${\sf Con}_{{\cal P},\kappa,n}^{{\rm ind},[\omega>]}$ also match.
\end{theorem}

For this reason, an evaluation of ${\cal P}^{\rm ord}$-paths using one of the standard orders on ${\sf Lab}_{{{\cal P}^{\rm ord}},n}^{[*]}\cap {\sf Lab}_{{\cal P}^{\rm ord}} ^{\rm fin}$ can be sufficient  for any ${\cal P}\in {\sf P}_{\sigma}$ and  $n\geq 1$ if we allow  only structures closed under complementation  with respect to relations for interpreting the instructions.
 \begin{consequ}For BSS RAMs over structures closed under complementation for relations, it is enough to consider, use,  and evaluate decision trees derived from partially label-ordered programs ${\cal P}$ with the following properties.
\begin{itemize} \parskip -0.8mm
\item  For inputs of length $n$,   only the finite ${\cal P}$-paths in ${\sf Lab}_{{\cal P},n}^{[*]}$ are important.
\item Arbitrary initial segments of ${\cal P}$-paths  can be represented  in such a tree  from the left to the right in accordance with $\leq_{\cal P}$ (cf.\,\,\cite{Klein97} and  \cite[p.\,\,36]{Klein05}). 
\item All  ${\cal P}$-paths of a fixed length  can be $\leq_{\cal P}^{\rm fin}$- and $\leq_{\cal P}$-compatibly enumerated and the resulting sequence of paths can be represented  from the left to the right in such a tree.   
\item The finite ${\cal P}$-paths   can be $\leq_{\cal P}^{\rm fin}$-compatibly enumerated and the resulting sequence of paths can be represented  from the left to the right in such a tree.   
\item  The mentioned systems ${\sf sy}_{{\cal P},n}^{\sf decision}(B^*)$ derived from these trees and used for  characterizing the behavior of machines, consist of atomic $\sigma$-formulas. 
\end{itemize}
\end{consequ}
For more details, see {\sf Extras II}, {\sf Extras III}, and Figure \ref{BerWurzelnProgrammpfade}. 

\newpage
\section*{Summary and outlook}\label{Summary}
\addcontentsline{toc}{section}{\bf Summary and outlook}
\markboth{SUMMARY: TRANSFORMATIONS}{SUMMARY: TRANSFORMATIONS}

\subsection*{Transformations}
\addcontentsline{toc}{subsection}{Transformations}

\begin{overview}[Transformation of ${\cal M}$-configurations for BSS RAMs]\label{TransfConf}
\hfill

\nopagebreak 

\noindent \fbox{\parbox{11.8cm}{

{\small 
\hfill Let ${\cal A}$ be a structure of finite signature $\sigma$, ${\cal M}\in {\sf M}_{\cal A}\cup {\sf M}_{\cal A}^{\rm ND}\cup {\sf M}_{\cal A}^{\rm DND}\cup {\sf M}_{\cal A}^{\rm NDB}$, 

\hfill ${\cal P}_{\cal M}$ be a $\sigma$-program,
 and $(\ell\,.\,\vec\nu\,.\, \bar u)\in {\sf S}_{\cal M}$.

\vspace{ 0.2cm}
\em\begin{tabular}{l}
(1) \fbox{$\ell \!: \, Z_j:= $\colorbox{blue!7}{$f_i^{m_i}(Z_{j_1},\ldots, Z_{j_{m_i}})$} } in ${\cal P}_{\cal M}$ implies \vspace{ 0.2cm}\\
\hspace{1.4cm}$(\ell \,.\,\vec \nu\,.\,\bar u) \textcolor{red}{ \quad\to_{\cal M}\quad} (\textcolor{blue}{\ell +1} \,.\,\vec \nu\,.\, ( u_1,\ldots,u_{j-1},\textcolor{blue}{f_i(u_{j_1},\ldots, u_{j_{m_i}})},u_{j+1}, \ldots)) $. 
\vspace{ 0.2cm}\\ 

(2) \fbox{ $\ell \!: \ Z_j:= $\colorbox{blue!7}{$c_i^0 $}} in ${\cal P}_{\cal M}$ implies \vspace{ 0.2cm}\\
\hspace{1.45cm} $(\ell \,.\,\vec \nu\,.\,\bar u)\textcolor{red}{ \quad\to_{\cal M}\quad} (\textcolor{blue}{\ell +1} \,.\,\vec \nu\,.\, ( u_1,\ldots,u_{j-1},\textcolor{blue}{c_{i}},u_{j+1}, \ldots) )$.
\vspace{ 0.2cm}\\ 

(3) \fbox{ $\ell \!: \,Z_{I_j}:=$\colorbox{blue!7}{$Z_{I_k}$}} in ${\cal P}_{\cal M}$ implies \vspace{ 0.2cm}\\
\hspace{1.4cm} $(\ell \,.\,\vec \nu\,.\,\bar u)\textcolor{red}{ \quad\to_{\cal M}\quad} (\textcolor{blue}{\ell +1} \,.\,\vec \nu\,.\, (u_1,\ldots,u_{\nu_j-1},\textcolor{blue}{u_{\nu_k}},u_{\nu_j+1}, \ldots) )$.
\vspace{ 0.2cm}\\ 

(4) \fbox{ \sf $\ell \!: \,$ if $r_i^{k_i}(Z_{j_1},\ldots, Z_{j_{k_i}})$ then goto \colorbox{blue!7}{$\ell _1$} else goto \colorbox{blue!7}{$\ell _2$}} in ${\cal P}_{\cal M}$ implies \vspace{ 0.2cm}\\
\hspace{1.5cm} $(\ell \,.\,\vec \nu\,.\,\bar u)\textcolor{red}{ \quad\to_{\cal M}\quad} (\textcolor{blue}{\ell_1}\, .\,\vec \nu\,.\,\bar u)$ \qquad\qquad {\sf if } $ (u_{j_1},\ldots, u_{j_{k_i}})\in r_i$,\\
\hspace{1.5cm} $(\ell \,.\,\vec \nu\,.\,\bar u) \textcolor{red}{\quad \to_{\cal M}\quad} (\textcolor{blue}{\ell_2}\,.\,\vec \nu\,.\,\bar u)$ \qquad\qquad {\sf if } $ (u_{j_1},\ldots, u_{j_{k_i}})\not \in r_i$.
\vspace{ 0.2cm}\\ 

(5) \fbox{ \sf $\ell \!: \,$ if $I_j=I_k$ then goto \colorbox{blue!7}{$\ell _1$} else goto \colorbox{blue!7}{$\ell _2$}} in ${\cal P}_{\cal M}$ implies \vspace{ 0.2cm}\\
\hspace{1.5cm} $(\ell \,.\,\vec \nu\,.\,\bar u) \textcolor{red}{ \quad\to_{\cal M}\quad} (\textcolor{blue}{\ell_1}\, .\,\vec \nu\,.\,\bar u)$ \qquad\qquad {\sf if } $ \nu_j=\nu_k$,\\
\hspace{1.5cm} $(\ell \,.\,\vec \nu\,.\,\bar u)\textcolor{red}{ \quad\to_{\cal M}\quad} (\textcolor{blue}{\ell_2}\,.\,\vec \nu\,.\,\bar u)$ \qquad\qquad {\sf if } $\nu_j\not = \nu_k$.
\vspace{ 0.2cm}\\

(6) \fbox{ $\ell \!: \, I_j:=$ \colorbox{blue!7}{$1$}} in ${\cal P}_{\cal M}$ implies \vspace{ 0.2cm}\\
\hspace{1.4cm} $(\ell \,.\,\vec \nu\,.\,\bar u) \textcolor{red}{ \quad\to_{\cal M}\quad} (\textcolor{blue}{\ell +1}\,.\, ( \nu_1,\ldots,\nu_{j-1},\textcolor{blue}{1},\nu_{j+1}, \ldots,\nu_{k_{\cal M}})\,.\,\bar u )$. 
\vspace{ 0.2cm}\\

(7) \fbox{ $\ell \!: \, I_j:=$ \colorbox{blue!7}{$I_j+1$}} in ${\cal P}_{\cal M}$ implies \vspace{ 0.2cm}\\
\hspace{1.4cm} $(\ell \,.\,\vec \nu\,.\,\bar u) \textcolor{red}{ \quad\to_{\cal M}\quad} (\textcolor{blue}{\ell +1}\,.\, ( \nu_1,\ldots,\nu_{j-1},\textcolor{blue}{\nu_j+1},\nu_{j+1}, \ldots,\nu_{k_{\cal M}})\,.\,\bar u)$. 
\vspace{ 0.2cm}\\

(8) \fbox{ $\ell _{{\cal P}_{\cal M}}\!: $ {\sf stop}} in ${\cal P}_{\cal M}$ implies \vspace{ 0.2cm}\\
\hspace{1.4cm}$(\ell_{{\cal P}_{\cal M}}\,.\,\vec \nu\,.\,\bar u) \textcolor{red}{ \,\,\to_{\cal M}\quad} (\ell_{{\cal P}_{\cal M}}\,.\,\vec \nu\,.\,\bar u) $.
\vspace{ 0.2cm}\\

$\!\!\!(11)$ \fbox{\sf $\ell :$ \,goto \colorbox{blue!7}{$\ell _1$} or goto \colorbox{blue!7}{$\ell _2$}} \quad in ${\cal P}_{\cal M}$ implies \vspace{0.2cm}\\
\hspace{1.4cm} $(\ell \,.\,\vec \nu\,.\,\bar u) \textcolor{red}{ \quad\to_{\cal M}\quad} (\textcolor{blue}{\ell_1}\, .\,\vec \nu\,.\,\bar u)$ \colorbox{blue!7}{and}\\
\hspace{1.4cm} $(\ell \,.\,\vec \nu\,.\,\bar u) \textcolor{red}{ \quad\to_{\cal M}\quad} (\textcolor{blue}{\ell_2}\, .\,\vec \nu\,.\,\bar u)$.\\
\end{tabular}}}}
\end{overview}

\newpage
\begin{overview}
[Transformation of partial ${\cal P}$-configurations]\label{RestrTransf}
\hfill

\nopagebreak 
\noindent \fbox{\parbox{11.8cm}{\small 

\hfill Let ${\cal P}$ be a $\sigma$-program and $(\ell\,.\,\vec\nu)\in {\sf S}_{{\cal P}}^{\rm ind}$.

\hfill In {\rm (5)}, {\rm (6)}, and {\rm (7)}, we explicitly assume that 

\hfill $ j,k\leq k_{\cal P}\leq \kappa$ and $(\ell\,.\,\vec\nu)\in {\sf S}_{{\cal P},\kappa}^{\rm ind}$ hold.

\em\vspace{0.3cm}
\begin{tabular}{rl}
$(1$--$3)$ &For $\ell \in {\cal L}_{{\cal P},{\rm F}}\cup {\cal L}_{{\cal P},{\rm F}_0}\cup{\cal L}_{{\cal P},{\rm C}}$, let\vspace{0.2cm}\\
&\hspace{1.5cm} $(\ell \,.\,\vec \nu) \textcolor{red}{ \,\to^{\rm ind}_{\cal P}\,} (\textcolor{blue}{\ell+1}\,.\,\vec \nu)$. \vspace{0.2cm}\\

$(4)$& For $\ell \in {\cal L}_{{\cal P},{\rm T}}$ \vspace{0.15cm}\\& \,\, and \quad \fbox{\sf $\ell : \,$ if $r_i^{k_i}(Z_{j_1},\ldots, Z_{j_{k_i}})$ then goto \colorbox{blue!7}{$\ell _1$} else goto \colorbox{blue!7}{$\ell _2$}} \, in ${\cal P}$, let \vspace{0.2cm}\\
&\hspace{1.5cm} $(\ell \,.\,\vec \nu) \textcolor{red}{ \,\to^{\rm ind}_{\cal P}\,} (\textcolor{blue}{\ell_1}\, .\,\vec \nu)$ \colorbox{blue!7}{and}\\
&\hspace{1.5cm} $(\ell \,.\,\vec \nu) \textcolor{red}{ \,\to^{\rm ind}_{\cal P}\,} (\textcolor{blue}{\ell_2}\, .\,\vec \nu)$. 
\vspace{0.2cm}\\

$(5)$& For $\ell \in {\cal L}_{{\cal P},{\rm H}_{\rm T}}$ \vspace{0.15cm}\\& \,\, and \quad \fbox{ \sf $\ell \!: \,$ if $I_j=I_k$ then goto \colorbox{blue!7}{$\ell _1$} else goto \colorbox{blue!7}{$\ell _2$}} \quad in ${\cal P}$, let\vspace{0.2cm}\\
&\hspace{1.5cm} $(\ell \,.\,\vec \nu) \textcolor{red}{ \,\to^{\rm ind}_{\cal P}\,} (\textcolor{blue}{\ell_1}\,.\,\vec \nu)$ \quad \colorbox{blue!7}{if $ \nu_j=\nu_k$} and thus $\textcolor{magenta}{T_\ell^{(\kappa)}}(\vec \nu)=\ell_1$ \colorbox{blue!7}{or}\\&\hspace{1.5cm} $(\ell \,.\,\vec \nu)\textcolor{red}{ \,\to^{\rm ind}_{\cal P}\,} (\textcolor{blue}{\ell_2}\,.\,\vec \nu)$ \quad \colorbox{blue!7}{if $\nu_j\not = \nu_k$} and thus $\textcolor{magenta}{T_\ell^{(\kappa)}}(\vec \nu)=\ell_2$. 
\vspace{0.2cm}\\

$(6)$& For $\ell \in {\cal L}_{{\cal P},{\rm H}_1}$ \vspace{0.15cm}\\&\,\, and \quad \fbox{ $\ell \!: \, I_j:=1$} \qquad in ${\cal P}$, let\vspace{0.2cm}\\
&\hspace{1.5cm} $(\ell \,.\,\vec \nu) \textcolor{red}{ \,\to^{\rm ind}_{\cal P}\,} (\textcolor{blue}{\ell +1}\,.\, \underbrace {( \nu_1,\ldots,\nu_{j-1},\textcolor{blue}{1},\nu_{j+1}, \ldots,\nu_{\kappa})}_{ \textcolor{magenta}{H_\ell^{(\kappa)}}(\vec \nu)} )$. 
\vspace{0.2cm}\\

$(7)$& For $\ell \in {\cal L}_{{\cal P},{\rm H}_{+1}}$ \vspace{0.15cm}\\&\,\, and \quad \fbox{ $\ell \!: \, I_j:=I_j+1$} \quad in ${\cal P}$, let\vspace{0.2cm}\\
&\hspace{1.5cm} $(\ell \,.\,\vec \nu) \textcolor{red}{ \,\to^{\rm ind}_{\cal P}\,} (\textcolor{blue}{\ell +1}\,.\, \underbrace {( \nu_1,\ldots,\nu_{j-1},\textcolor{blue}{\nu_j+1},\nu_{j+1}, \ldots,\nu_{\kappa})}_{ \textcolor{magenta}{H_\ell^{(\kappa)}}(\vec \nu)} )$. 
\vspace{0.2cm}\\

$(8)$& For $\ell \in {\cal L}_{{\cal P},{\rm S}}$, let\vspace{0.2cm}\\
&\hspace{1.5cm} $(\ell \,.\,\vec \nu) \textcolor{red}{ \,\to^{\rm ind}_{\cal P}\,} (\ell\,.\,\vec \nu)$. \vspace{0.2cm}\\

$\!\!\!(11)$& For $\ell \in {\cal L}_{{\cal P},{\rm B}}$ \vspace{0.15cm}\\& \,\, and \fbox{\sf $\ell :$ \,goto \colorbox{blue!7}{$\ell _1$} or goto \colorbox{blue!7}{$\ell _2$}} \quad in ${\cal P}$, let\vspace{0.2cm}\\
&\hspace{1.5cm} $(\ell \,.\,\vec \nu) \textcolor{red}{ \,\to^{\rm ind}_{\cal P}\,} (\textcolor{blue}{\ell_1}\, .\,\vec \nu)$ \colorbox{blue!7}{and}\\
&\hspace{1.5cm} $(\ell \,.\,\vec \nu) \textcolor{red}{ \,\to^{\rm ind}_{\cal P}\,} (\textcolor{blue}{\ell_2}\, .\,\vec \nu)$. 
\vspace{0.1cm}\\
\end{tabular}
}}
\end{overview}
\newpage

\subsection*{Transition systems}\label{OverviewsTransSyst}
\addcontentsline{toc}{subsection}{Transition systems}
\markboth{SUMMARY: TRANSITION SYSTEMS}{SUMMARY: TRANSITION SYSTEMS}

\begin{overview}[Transition systems ${\cal S}_{\cal P}^{\rm lab}$ for generating ${\cal P}$-paths]\label{TransPPaths} 
\hfill

\nopagebreak 
\noindent \fbox{\parbox{11.8cm}{\fbox{${\cal S}_{\cal P}^{\rm lab}=({\sf S}_{\cal P}^{\rm lab}, \to_{\cal P}^{\rm lab})$} 
\hfill{\small Let ${\cal P}$ be a $\sigma$-program.}

\em 
\vspace{0.1cm}
\em 
$\begin{array}{lll}
{\sf S}_{\cal P}^{\rm lab} \hspace{0.17cm}&=\hspace{0.24cm} \{\ell\mid 1\leq \ell\leq \ell_{\cal P}\}
\end{array}$

\vspace{0.1cm}
$\begin{array}{lll}

\textcolor{red}{\to_{\cal P}^{\rm lab}} &= \hspace{0.23cm}\{(\ell, \textcolor{blue}{\ell +1})&\!\!\in ({\sf S}_{\cal P}^{\rm lab})^2\mid \ell \in{\cal L}_{\cal P}^{ \rm NB} \}\\
&\hspace{0.3cm}\cup \,\{(\ell, \textcolor{blue}{\beta_{\cal P}^+(\ell)})&\!\!\in ( {\sf S}_{\cal P}^{\rm lab})^2\mid \ell\in {\cal L}_{\cal P}^{ \rm B} \}\\
&\hspace{0.3cm}\cup \,\{(\ell, \textcolor{blue}{\beta_{\cal P}^-(\ell)})&\!\!\in ( {\sf S}_{\cal P}^{\rm lab})^2\mid \ell\in {\cal L}_{\cal P}^{ \rm B} \}.
 \end{array}$

\vspace{0.05cm}
${\cal F}_{\cal P}^{\rm lab}=\{ \beta_{\cal P},\beta_{\cal P}^+,\beta_{\cal P}^-\}$ \hfill {\small (with $\textcolor{blue}{\beta_{\cal P}(\ell)}=\ell+1$)}.
}}
\end{overview}

\vspace{0.05cm}
\begin{overview}[Transition systems ${\cal S}_{{\cal P},\kappa}^{\rm ind}$ for generating $({\cal P}\!,\kappa,n)$-paths]\label{TransPKappaPaths} 
\hfill

\nopagebreak 
\noindent \fbox{\parbox{11.8cm}{
\fbox{${\cal S}_{{\cal P},\kappa}^{\rm ind}=({\sf S}_{{\cal P},\kappa}^{\rm ind}, \to_{{\cal P},\kappa}^{\rm ind})$} 
\hfill{\small Let ${\cal P}$ be a $\sigma$-program and $\kappa\geq k_{\cal P}$.}

\vspace{0.1cm}
\em 
$\begin{array}{lll}
{\sf S}_{{\cal P},\kappa}^{\rm ind} \hspace{0.17cm}&=\hspace{0.24cm} \{ (\ell\,.\,\vec \nu)\mid 1\leq \ell\leq \ell_{\cal P} \,\,\&\,\, \vec \nu\in (\mathbb{N}_+)^{\kappa}\}.
\end{array}$

\vspace{0.1cm}
$\begin{array}{lll}
\textcolor{red}{\to_{{\cal P},\kappa}^{\rm ind}} &= \hspace{0.23cm}\{((\ell\,.\,\vec \nu), (\ell +1\,.\,\vec \nu))&\!\!\in ({\sf S}_{{\cal P},\kappa}^{\rm ind})^2\mid \ell \in{\cal L}_{{\cal P}, { \rm F} } \cup {\cal L}_{{\cal P}, { \rm F}_0 } \cup {\cal L}_{{\cal P}, { \rm C} } \}\\
 &\hspace{0.3cm}\cup \,\{((\ell\,.\,\vec \nu), (\ell +1\,.\,\textcolor{blue}{H_\ell^{(\kappa)}(\vec \nu)}))\!\!&\!\!\in ({\sf S}_{{\cal P},\kappa}^{\rm ind})^2\mid \ell\in {\cal L}_{{\cal P},{\rm H}_1}\cup \, {\cal L}_{{\cal P},{\rm H}_{+1}} \}\\
&\hspace{0.3cm}\cup \,\{((\ell\,.\,\vec \nu), (\textcolor{blue}{\beta_{\cal P}^+(\ell)}\,.\,\vec \nu))&\!\!\in ({\sf S}_{{\cal P},\kappa}^{\rm ind})^2\mid \ell\in {\cal L}_{{\cal P},{\rm T}}\cup {\cal L}_{{\cal P},{\rm B}}\}\\
&\hspace{0.3cm}\cup \,\{((\ell\,.\,\vec \nu), (\textcolor{blue}{\beta_{\cal P}^-(\ell)}\,.\,\vec \nu))&\!\!\in ({\sf S}_{{\cal P},\kappa}^{\rm ind})^2\mid \ell\in {\cal L}_{{\cal P},{\rm T}}\cup {\cal L}_{{\cal P},{\rm B}}\}\\
&\hspace{0.3cm}\cup \,\{((\ell\,.\,\vec \nu), (\textcolor{blue}{T_\ell^{(\kappa)}(\vec \nu)}\,.\,\vec \nu))&\!\!\in ({\sf S}_{{\cal P},\kappa}^{\rm ind})^2\mid \ell\in {\cal L}_{{\cal P},{\rm H}_{\rm T}} \}\\
&\hspace{0.3cm}\cup \, \{((\ell\,.\,\vec \nu), (\ell\,.\,\vec \nu))&\!\!\in ({\sf S}_{{\cal P},\kappa}^{\rm ind})^2 \mid \ell=\ell_{\cal P}\}.\end{array}$

\vspace{0.1cm}
${\cal F}_{{\cal P},\kappa}^{\rm ind}= \{H_\ell^{(\kappa)} \mid \ell\in {\cal L}_{{\cal P},{\rm H}_1}\cup {\cal L}_{{\cal P},{\rm H}_{+1}} \}\, \cup\,\{T_\ell^{(\kappa)}\mid \ell\in {\cal L}_{{\cal P},{\rm H}_{\rm T}} \}\,\cup \, {\cal F}_{\cal P}^{\rm lab}$.

\hfill {\small\em $H_\ell^{(\kappa)},T_\ell^{(\kappa)}$ depend on ${\cal P}$ (cf.\,\,Overview \ref{TransRelConfM}).} 

\hfill {\small\em If necessary, we denote them by $H_{{\cal P},\kappa,\ell},T_{{\cal P},\kappa,\ell}$.}}}
\end{overview}

\begin{overview}[The change of partial $({\cal P},\kappa)$-configurations]

\hfill

\nopagebreak 

\noindent \fbox{\parbox{11.8cm}{\em\begin{tabular}{l}

(5) ${\rm H}_{\rm T}$-instructions { \sf $\ell \!: \,$ if $I_j=I_k$ then goto $\ell _1$ else goto $\ell _2$}\\

\qquad $(\ell \,.\,\vec \nu)\, \textcolor{red}{\to_{{\cal P},\kappa}^{\rm ind}}\,(\underbrace{\ell_1}_{ T_\ell^{(\kappa)}(\vec \nu)} .\,\vec \nu)$ \qquad \qquad if $ \nu_j=\nu_k$\\

\qquad $(\ell \,.\,\vec \nu)\,\textcolor{red}{\to_{{\cal P},\kappa}^{\rm ind}}\,(\,\,\overbrace{\,\ell_2\,}^{}\,.\,\vec \nu)$ \qquad\qquad if $\nu_j\not = \nu_k$

\vspace{0.1cm}\\
(6) ${{\rm H}_1}$-instructions { $\ell \!: \, I_j:=1$}\\

\qquad $(\ell \,.\,\vec \nu)\,\textcolor{red}{\to_{{\cal P},\kappa}^{\rm ind}}\,(\ell +1\,.\, \underbrace {( \nu_1,\ldots,\nu_{j-1},1,\nu_{j+1}, \ldots,\nu_{\kappa})}_{ H_\ell^{(\kappa)}(\vec \nu)} )$\\

(7) ${{\rm H}_{+1}}$-instructions { $\ell \!: \, I_j:=I_j+1$}\\

\qquad $(\ell \,.\,\vec \nu)\, \textcolor{red}{\to_{{\cal P},\kappa}^{\rm ind}}\, (\ell +1\,.\, \underbrace {( \nu_1,\ldots,\nu_{j-1},\nu_j+1,\nu_{j+1}, \ldots,\nu_{\kappa})}_{ H_\ell^{(\kappa)}(\vec \nu)} )$ 
\end{tabular}

\hfill {\small\em (For more details, see Overview \ref{RestrTransf}.)}}}
\end{overview}

\newpage
\begin{overview}[Transition systems ${\cal S}_{\cal M}$ for generating $({\cal M},\vec x)$-paths]\label{TransRelConfM} 
\hfill

\nopagebreak 
\noindent \fbox{\parbox{11.8cm}{\fbox{${\cal S}_{\cal M}=({\sf S}_{\cal M}, \to_{\cal M})$} 
\hfill{\small Let ${\cal M}$ be a BSS RAM over ${\cal A}$.}

\vspace{0.1cm}
\em 
$\begin{array}{lll}
{\sf S}_{\cal M}\hspace{0.34cm}&=\hspace{0.24cm}
\{ (\ell\,.\,\vec \nu\,.\, \bar u)\mid 1\leq \ell\leq \ell_{{\cal P}_{\cal M}} \,\,\&\,\, \vec \nu\in (\mathbb{N}_+)^{k_{\cal M}} \,\,\&\,\, \bar u\in U_{\cal A} ^\omega\}.\\
\end{array}$

\vspace{0.2cm}
$\begin{array}{lll}
\, \,\textcolor{red}{\to_{\cal M}} \,&=\hspace{0.26cm}\{((\ell\,.\,\vec \nu\,.\, \bar u), (\ell +1\,.\,\vec \nu\,.\, \textcolor{blue}{F_{\ell }(\bar u)}))&\!\!\in {\sf S}_{\cal M}^2\mid \ell \in{\cal L}_{{\cal M}, { \rm F} } \cup \,{\cal L}_{{\cal M}, { \rm F} _0} \}\\
&\hspace{0.3cm}\cup \,\{((\ell\,.\,\vec \nu\,.\, \bar u), (\ell +1\,.\,\vec \nu\,.\, \textcolor{blue}{C_{\ell }(\vec \nu,\bar u)}))\!\!&\!\!\in {\sf S}_{\cal M}^2\mid \ell \in{\cal L}_{{\cal M}, { \rm C} } \}\\
 &\hspace{0.3cm}\cup \,\{((\ell\,.\,\vec \nu\,.\, \bar u), (\ell +1\,.\,\textcolor{blue}{H_{\ell }(\vec \nu)}\,.\, \bar u))\!&\!\!\in {\sf S}_{\cal M}^2\mid \ell\in {\cal L}_{{\cal M},{\rm H}_1}\cup \, {\cal L}_{{\cal M},{\rm H}_{+1}} \}\\
&\hspace{0.3cm}\cup \,\{((\ell\,.\,\vec \nu\,.\, \bar u), (\textcolor{blue}{T_{\ell }(\bar u)}\,.\,\vec \nu\,.\,\bar u))&\!\!\in {\sf S}_{\cal M}^2\mid \ell\in {\cal L}_{{\cal M},{\rm T}}\}\\
&\hspace{0.3cm}\cup \,\{((\ell\,.\,\vec \nu\,.\, \bar u), (\textcolor{blue}{\beta_{{\cal P}_{\cal M}}^+(\ell)}\,.\,\vec \nu\,.\,\bar u))&\!\!\in {\sf S}_{\cal M}^2\mid \ell\in {\cal L}_{{\cal M},{\rm B}} \}\\
&\hspace{0.3cm}\cup \,\{((\ell\,.\,\vec \nu\,.\, \bar u), (\textcolor{blue}{\beta_{{\cal P}_{\cal M}}^-(\ell)}\,.\,\vec \nu\,.\,\bar u))&\!\!\in {\sf S}_{\cal M}^2\mid \ell\in {\cal L}_{{\cal M},{\rm B}} \}\\
&\hspace{0.3cm}\cup \,\{((\ell\,.\,\vec \nu\,.\, \bar u), (\textcolor{blue}{T_{\ell }(\vec \nu)}\,.\,\vec \nu\,.\,\bar u))&\!\!\in {\sf S}_{\cal M}^2\mid \ell\in {\cal L}_{{\cal M},{\rm H}_{\rm T}} \}\\
&\hspace{0.3cm}\cup \, \{((\ell\,.\,\vec \nu\,.\, \bar u), (\ell\,.\,\vec \nu\,.\, \bar u))&\!\!\in {\sf S}_{\cal M}^2 \mid \ell=\ell_{{\cal P}_{\cal M}}\}.\end{array}$

\vspace{0.2cm}
$\begin{array}{ll} {\cal F}_{\cal M}\hspace{0.26cm}&=\{F_\ell \mid \ell\in {\cal L}_{{\cal M},{\rm F}}\cup {\cal L}_{{\cal M},{\rm F}_0} \} 
\,\cup\,\{C_\ell \mid \ell\in {\cal L}_{{\cal M},{\rm C}}\}\\
 &\hspace{0.4cm} \cup\, \{H_\ell \mid \ell\in {\cal L}_{{\cal M},{\rm H}_1}\cup {\cal L}_{{\cal M},{\rm H}_{+1}} \} \,\cup\,\{T_\ell \mid \ell\in {\cal L}_{{\cal M},{\rm H}_{\rm T}} \cup {\cal L}_{{\cal M},{\rm T}} \}\\
&\hspace{0.4cm} \cup \,\, {\cal F}_{{\cal P}_{\cal M}}^{\rm lab}.\end{array}$

\hfill {\small\em $F_\ell, C_\ell, H_\ell,T_\ell$ depend an ${\cal M}$ (cf.\,\,\cite{GASS25A}).}

\hfill {\small\em If necessary, we denote them by $F_{{\cal M},\ell}, C_{{\cal M},\ell}, H_{{\cal M},\ell},T_{{\cal M},\ell}$.}

}} \end{overview}

\begin{overview}
[The change of ${\cal M}$-configurations]

\hfill

\nopagebreak 

\noindent \fbox{\parbox{11.8cm}{ \em\begin{tabular}{l}

(1) ${\rm F}$-instructions { $\ell \!: \, Z_j:= f_i^{m_i}(Z_{j_1},\ldots, Z_{j_{m_i}}) $} \hfill \vspace{0.1cm}\\

\qquad $(\ell \,.\,\vec \nu\,.\,\bar u) \textcolor{red}{ \,\to_{\cal M}\,} (\ell +1 \,.\,\vec \nu\,.\, \underbrace {( u_1,\ldots,u_{j-1},f_i(u_{j_1},\ldots, u_{j_{m_i}}),u_{j+1}, \ldots)}_{ F_{\ell }(\bar u)}) $ 

\vspace{0.3cm}\\
(2) ${\rm F}_0 $-instructions { $\ell \!: \ Z_j:= c_i^0 $} \hfill \vspace{0.1cm}\\

\qquad $(\ell \,.\,\vec \nu\,.\,\bar u)\textcolor{red}{ \,\to_{\cal M}\,} (\ell +1 \,.\,\vec \nu\,.\, \underbrace {( u_1,\ldots,u_{j-1},c_{\alpha_i},u_{j+1}, \ldots) }_{ F_{\ell }(\bar u)})$

\vspace{0.3cm}\\
(3) ${\rm C}$-instructions { $\ell \!: \,Z_{I_j}:=Z_{I_k}$} \hfill \vspace{0.1cm}\\

\qquad $(\ell \,.\,\vec \nu\,.\,\bar u)\textcolor{red}{ \,\to_{\cal M}\,} (\ell +1 \,.\,\vec \nu\,.\,\underbrace {(u_1,\ldots,u_{\nu_j-1},u_{\nu_k},u_{\nu_j+1}, \ldots)}_{\ C_{\ell }(\vec \nu,\bar u)} )$

\vspace{0.3cm}\\
(4) ${\rm T}$-instructions 

{ \sf $\ell \!: \,$ if $r_i^{k_i}(Z_{j_1},\ldots, Z_{j_{k_i}})$ then goto $\ell _1$ else goto 
 $\ell _2$} \vspace{0.1cm}\\
\qquad $(\ell \,.\,\vec \nu\,.\,\bar u)\textcolor{red}{ \,\to_{\cal M}\,} (\,\underbrace{\ell_1}_{ T_{\ell }(\bar u)} .\,\vec \nu\,.\,\bar u)$ \qquad\qquad if $ (u_{j_1},\ldots, u_{j_{k_i}})\in r_i$\\
\qquad $(\ell \,.\,\vec \nu\,.\,\bar u) \textcolor{red}{\, \to_{\cal M}\,} (\,\overbrace{\ell_2}^{}.\,\vec \nu\,.\,\bar u)$ \qquad\qquad if $ (u_{j_1},\ldots, u_{j_{k_i}})\not \in r_i$
\vspace{0.3cm}\\

\end{tabular}

\hfill {\small\em (For more details, see Overview \ref{TransfConf}.)}

}}
\end{overview}

\newpage
\addcontentsline{toc}{subsection}{Outlook: Transition systems for term configurations}
\markboth{OUTLOOK: TRANSITION SYSTEMS}{OUTLOOK: TRANSITION SYSTEMS}

\subsection*{Outlook: Transition systems for term configurations}

\vspace{0.1cm}
\begin{overview}[Systems ${\cal S}_{{\cal P},\kappa} $ for generating $({\cal P},\kappa,n)$-term-paths] \label{OutlookTermConf1}
\hfill

\nopagebreak 
\noindent \fbox{\parbox{11.8cm}{\fbox{${\cal S}_{{\cal P},\kappa}=({\sf S}_{{\cal P},\kappa}, \to_{{\cal P},\kappa}) $} \hfill{\small Let ${\cal P}$ be a $\sigma$-program and $\kappa\geq k_{\cal P}$.}

\em

\vspace{0.1cm}
$\begin{array}{lll}{\sf S}_{{\cal P},\kappa} \hspace{0.22cm}&=\hspace{0.24cm} \{(\ell\,.\,\vec \nu\,.\,\bar\tau)\mid 1\leq \ell\leq \ell_{\cal P}\, \,\,\&\,\,\, \vec \nu\in (\mathbb{N}_+)^{\kappa}\,\,\,\&\,\,\,\bar\tau \in ({\rm Term}^{\sigma})^\omega\}.\vspace{0.1cm}\\
\end{array}$

$\begin{array}{lll}
\,\textcolor{red}{\to_{{\cal P},\kappa}}&=\hspace{0.2cm}\{((\ell\,.\,\vec \nu\,.\, \bar \tau), (\ell +1\,.\,\vec \nu\,.\, \textcolor{blue}{F_{\ell }^{\sigma}(\bar \tau)}))&\!\!\in {\sf S}_{{\cal P},\kappa}^2\mid \ell \in{\cal L}_{{\cal P}, { \rm F} } 
\cup \,{\cal L}_{{\cal P}, { \rm F} _0} \}\\
&\hspace{0.3cm}\cup \,\{((\ell\,.\,\vec \nu\,.\, \bar \tau), (\ell +1\,.\,\vec \nu\,.\, \textcolor{blue}{C_{\ell }^{\kappa,\sigma}(\vec \nu,\bar \tau)}))\!\!\!&\!\!\in {\sf S}_{{\cal P},\kappa}^2\mid \ell \in{\cal L}_{{\cal P}, { \rm C} } \}\\
 &\hspace{0.3cm}\cup \,\{((\ell\,.\,\vec \nu\,.\, \bar \tau), (\ell +1\,.\,\textcolor{blue}{H_\ell^{(\kappa)}(\vec \nu)}\,.\, \bar \tau))\!&\!\!\in {\sf S}_{{\cal P},\kappa}^2\mid \ell\in {\cal L}_{{\cal P},{\rm H}_1}\cup \, {\cal L}_{{\cal P},{\rm H}_{+1}} \}\\
&\hspace{0.3cm}\cup \,\{((\ell\,.\,\vec \nu\,.\, \bar \tau), (\textcolor{blue}{\beta_{\cal P}^+(\ell)}\,.\,\vec \nu\,.\,\bar \tau))&\!\!\in {\sf S}_{{\cal P},\kappa}^2\mid \ell\in {\cal L}_{{\cal P},{\rm T}}\cup {\cal L}_{{\cal P},{\rm B}} \}\\
&\hspace{0.3cm}\cup \,\{((\ell\,.\,\vec \nu\,.\, \bar \tau), (\textcolor{blue}{\beta_{\cal P}^-(\ell)}\,.\,\vec \nu\,.\,\bar \tau))&\!\!\in {\sf S}_{{\cal P},\kappa}^2\mid \ell\in {\cal L}_{{\cal P},{\rm T}}\cup {\cal L}_{{\cal P},{\rm B}} \}\\
&\hspace{0.3cm}\cup \,\{((\ell\,.\,\vec \nu\,.\, \bar \tau), (\textcolor{blue}{T_\ell^{(\kappa)}(\vec \nu)}\,.\,\vec \nu\,.\,\bar \tau))&\!\!\in {\sf S}_{{\cal P},\kappa}^2\mid \ell\in {\cal L}_{{\cal P},{\rm H}_{\rm T}} \}\\
&\hspace{0.3cm}\cup \, \{((\ell\,.\,\vec \nu\,.\, \bar \tau), (\ell\,.\,\vec \nu\,.\, \bar \tau))&\!\!\in {\sf S}_{{\cal P},\kappa}^2 \mid \ell=\ell_{\cal P}\}.\end{array}$

\vspace{0.1cm}
$\begin{array}{ll} {\cal F}_{{\cal P},\kappa}=\hspace{0.24cm}
\{F_\ell ^{\sigma}\mid \ell\in {\cal L}_{{\cal P},{\rm F}}\cup {\cal L}_{{\cal P},{\rm F}_0} \} 
\,\cup\,\{C_\ell ^{\kappa,\sigma}\mid \ell\in {\cal L}_{{\cal P},{\rm C}}\} \,\cup\, {\cal F}_{{\cal P},\kappa}^{\rm ind} . \end{array}$
\vspace{0.1cm}

\hfill {\small\em ${\rm Term}^{\sigma}$ is the set of all $\sigma$-terms. $F_\ell ^{\sigma},C_\ell ^{\kappa,\sigma}, H_\ell^{(\kappa)},T_\ell^{(\kappa)}$ depend on ${\cal P}$.}

\hfill {\small\em If necessary, we denote them by $F_{{\cal P},\ell}^{\sigma}, C_{{\cal P},\ell} ^{\kappa,\sigma},H_{{\cal P},\kappa,\ell},T_{{\cal P},\kappa,\ell}$.}
}}\end{overview}

\vspace{0.1cm}
\begin{overview}[The change of $({\cal P},\kappa)$-term-configurations]\label{OutlookTermConf2}

\hfill

\nopagebreak 

\noindent \fbox{\parbox{11.8cm}{

\hfill {\small Let $\tau_1,\tau_2,\ldots $ be terms of signature $\sigma$.}

\vspace{0.3cm}
\em\begin{tabular}{l}
(1) ${\rm F}$-instructions { $\ell \!: \, Z_j:= f_i^{m_i}(Z_{j_1},\ldots, Z_{j_{m_i}}) $} \hfill \vspace{0.3cm}\\

\qquad $(\ell \,.\,\vec \nu\,.\,\bar \tau) \textcolor{red}{ \,\,\,\to_{{\cal P},\kappa}\,\,\,} (\ell +1 \,.\,\vec \nu\,.\, \underbrace {( \tau_1,\ldots,\tau_{j-1},f_i^{m_i}(\tau_{j_1},\ldots, \tau_{j_{m_i}}),\tau_{j+1}, \ldots)}_{ F_{\ell }^{\sigma}(\bar \tau)}) $ 

\vspace{0.3cm}\\
(2) ${\rm F}_0 $-instructions { $\ell \!: \ Z_j:= c_i^0 $} \hfill \vspace{0.3cm}\\

\qquad $(\ell \,.\,\vec \nu\,.\,\bar \tau)\textcolor{red}{ \,\,\,\to_{{\cal P},\kappa}\,\,\,} (\ell +1 \,.\,\vec \nu\,.\, \underbrace {( \tau_1,\ldots,\tau_{j-1},c_i^0,\tau_{j+1}, \ldots) }_{ F_{\ell }^{\sigma}(\bar \tau)})$

\vspace{0.3cm}\\
(3) ${\rm C}$-instructions { $\ell \!: \,Z_{I_j}:=Z_{I_k}$} \hfill \vspace{0.3cm}\\

\qquad $(\ell \,.\,\vec \nu\,.\,\bar \tau)\textcolor{red}{ \,\,\,\to_{{\cal P},\kappa}\,\,\,}(\ell +1 \,.\,\vec \nu\,.\,\underbrace {(\tau_1,\ldots,\tau_{\nu_j-1},\tau_{\nu_k},\tau_{\nu_j+1}, \ldots)}_{\ C_{\ell }^{\kappa,\sigma}(\vec \nu,\bar \tau)} )$

\vspace{0.4cm}\\
\end{tabular}

\hfill {\small\em (For more details, see {\sf Extras III}.)}
}}
\end{overview}

\vspace{0.3cm}

$\textcolor{red}{\to_{{\cal P},\kappa}^{\rm ind}}$ can be derived from $\textcolor{red}{\to_{\cal P}^{\rm lab}}$ using ${\cal A}_{\bbbn}$.

$\textcolor{red}{\to_{{\cal P},\kappa}}$ can be derived from $\textcolor{red}{\to_{{\cal P},\kappa}^{\rm ind}}$ using ${\sigma}$-terms.

$\textcolor{red}{\to_{\cal M}}\,\,$ could also be derived from $\textcolor{red}{\to_{{\cal P},\kappa}}$ using a ${\sigma}$-structure.

\newpage
\section*{Acknowledgment}
\markboth{ACKNOWLEDGMENT}{ACKNOWLEDGMENT}
{\small
I would like to thank the participants of my lectures and, in particular, {\sf Paul} {\sc Grieger} for useful questions and the discussions. My thanks go also to {\sf Michael} {\sc Rathjen} and {\sf Vincenzo} {\sc Mantova} and {\sf Peter} {\sc Schuster}, {\sf Gabriele} {\sc Buriola}, and {\sf Giulio} {\sc Fellin} for discussions and the opportunity to give talks on Abstract computation over first-order structures in Leeds and Verona. Moreover, I thank the staff of the International Office of the University of Greifswald for their support.

\vspace{0.1cm}
My research was also supported by {\sf Michael} {\sc Schürmann}, {\sf Rainer} \text{\sc Schimming}, and {\sf Volkmar} {\sc Liebscher}. I thank them for giving me the opportunity to present and discuss my results on the complexity of decision problems over algebraic structures in Greifswald. This paper addresses some of the more profound questions that were touched upon earlier and are now discussed in detail. In particular, I would like to thank  {\sf Adrian} {\sc Rezu\c{s}} for his work as the editor of the first volume of {\sf Landscapes in Logic} that also contains an introduction to BSS RAMs. 

\vspace{0.1cm}
Moreover, I would also like to take this opportunity to remember {\sf Günter} {\sc Asser} (Februar 26, 1926 -- March 23, 2015) and his 100th birthday this year. My research began with the study of mathematical logic in his research group at the University of Greifswald.}

{\small

\noindent For this article, we also used translators such as those of Google and DeepL and, for questions about English grammar and for the translation and appropriate use of technical terms, we moreover used Microsoft Copilot, Google, and Google AI Overview.}

\newpage

\section*{Attachment: The third path and a tree}\label{Attachment}
\addcontentsline{toc}{section}{\bf Attachment: The third path and a tree}
\markboth{ATTACHMENT: THE THIRD PATH AND A TREE}{ATTACHMENT: THE THIRD PATH AND A TREE}

For ${\cal P}$ given in Fig.\,\,\ref{BerWurzeln1}, we apply the algorithms summarized in Overviews \ref{ExamEnumLab1} and \ref{ExamEnumLab2} to $(1,1,1)\in U_{{\cal A}_{{\cal L}_{\cal P}}}^3$ and compute some values by using the $\hat\beta_{\cal P}^\circ$-functions.

\vspace{0.1cm}

{\scriptsize
\noindent \begin{tabular}{r|cccccccccccccccccc}
\hline
$c(Z_{j-1})$&0&1&2&3&4&5&6&7&8\\
\textcolor{gray}{(or $c(Z_{1,k-1})$)}&&$\in  {\cal L}_{\cal P}^{\rm B}$& $\in  {\cal L}_{\cal P}^{\rm NB}$& $\in  {\cal L}_{\cal P}^{\rm NB}$& $\in  {\cal L}_{\cal P}^{\rm NB}$& $\in  {\cal L}_{\cal P}^{\rm NB}$& $\in  {\cal L}_{\cal P}^{\rm B}$ &$\in  {\cal L}_{\cal P}^{\rm NB}$\\\hline
$\hat\beta_{\cal P}(c(Z_{j-1}))$&0&0&3&4&5&6&0&8&0\\
$\hat\beta_{\cal P}^+(c(Z_{j-1}))$&0& 2&0&0&0&0&7&0&0\\
$\hat\beta_{\cal P}^-(c(Z_{j-1}))$&0&1&0&0&0&0&5&0&0\\\hline
\end{tabular}}

 \subsection*{Computing the third path $enum_{\cal P}^{\sf (O)}(1,\!1,\!1)$ over ${\cal A}_{{\cal L}_{\cal P}}$}\label{ExampleFirstAlgo}

\markboth{ATTACHMENT:  A THIRD PATH}{ATTACHMENT:  A THIRD PATH}

{\scriptsize
\noindent \begin{tabular}{c|cccc|cccccccccccccc}\hline
$\ell\!$&$s\!$&$i\!$&$s_0\!$&$k\!$&$\!\!Z_1\!$&$\!\!Z_2\!$&$\!\!Z_3\!$&$\!\!Z_4\!$&$\!\!Z_5\!$&$\!\!Z_6\!$&$\!\!Z_7\!$&$\!\!Z_8\!$&$\!\!Z_9\!$&$\!\!Z_{10}\!$&$\!\!Z_{11}\!$&$k=s$?\\\hline
&&&&&\textcolor{blue}{1}&\textcolor{blue}{1}&\textcolor{blue}{1}&\\
$\cdots$&\textcolor{blue}{1}&\textcolor{blue}{1}&\textcolor{blue}{0}&&\textcolor{blue}{1}&\textcolor{gray}{$\cdot\cdot$}\\
$l_1$&\textcolor{blue}{2}&\textcolor{gray}{1}&\textcolor{gray}{0}&\textcolor{blue}{2}&\textcolor{gray}{1}&\\
$l_2$&\textcolor{gray}{2}&\textcolor{gray}{1}&\textcolor{gray}{0}&\textcolor{gray}{2}&\textcolor{gray}{1}&&&&&&&&&&&$2=2$\\
$l_3$&\textcolor{gray}{2}&\textcolor{gray}{1}&\textcolor{blue}{0}&\textcolor{gray}{2}&\textcolor{gray}{1}&\textcolor{gray}{?}\\\hline
$l_1$&\textcolor{blue}{3}&\textcolor{gray}{1}&\textcolor{gray}{0}&\textcolor{blue}{2}&\textcolor{gray}{1}&\\
$l_2$&\textcolor{gray}{3}&\textcolor{gray}{1}&\textcolor{gray}{0}&\textcolor{blue}{3}&\textcolor{gray}{1}&\textcolor{blue}{1}\\
$l_2$&\textcolor{gray}{3}&\textcolor{gray}{1}&\textcolor{gray}{0}&\textcolor{gray}{3}&\textcolor{gray}{1}&\textcolor{gray}{1}&&&&&&&&&&$3=3$\\
$l_3$&\textcolor{gray}{3}&\textcolor{gray}{1}&\textcolor{blue}{2}&\textcolor{gray}{3}&\textcolor{gray}{1}&\textcolor{gray}{1}&\textcolor{gray}{?}\\
$l_1$&\textcolor{gray}{3}&\textcolor{gray}{1}&\textcolor{gray}{2}&\textcolor{blue}{3}&\textcolor{gray}{1}&\textcolor{blue}{2}\\
$l_2$&\textcolor{gray}{3}&\textcolor{gray}{1}&\textcolor{gray}{2}&\textcolor{gray}{3}&\textcolor{gray}{1}&\textcolor{gray}{2}&&&&&&&&&&$3=3$\\
$l_3$&\textcolor{gray}{3}&\textcolor{gray}{1}&\textcolor{blue}{0}&\textcolor{gray}{3}&\textcolor{gray}{1}&\textcolor{gray}{2}&\textcolor{gray}{?}\\\hline
$l_1$&\textcolor{blue}{4}&\textcolor{gray}{1}&\textcolor{gray}{0}&\textcolor{blue}{2}&\textcolor{gray}{1}&\textcolor{gray}{$\cdot\cdot$}\\
$l_2$&\textcolor{gray}{4}&\textcolor{gray}{1}&\textcolor{gray}{0}&\textcolor{blue}{3}&\textcolor{gray}{1}&\textcolor{blue}{1}\\
$l_2$&\textcolor{gray}{4}&\textcolor{gray}{1}&\textcolor{gray}{0}&\textcolor{blue}{4}&\textcolor{gray}{1}&\textcolor{gray}{1}&\textcolor{blue}{1}\\
$l_2$&\textcolor{gray}{4}&\textcolor{gray}{1}&\textcolor{gray}{0}&\textcolor{gray}{4}&\textcolor{gray}{1}&\textcolor{gray}{1}&\textcolor{gray}{1}&&&&&&&&&$4=4$\\
$l_3$&\textcolor{gray}{4}&\textcolor{gray}{1}&\textcolor{blue}{3}&\textcolor{gray}{4}&\textcolor{gray}{1}&\textcolor{gray}{1}&\textcolor{gray}{1}&\textcolor{gray}{?}\\
$l_1$&\textcolor{gray}{4}&\textcolor{gray}{1}&\textcolor{gray}{3}&\textcolor{blue}{4}&\textcolor{gray}{1}&\textcolor{gray}{1}&\textcolor{blue}{2}\\
$l_2$&\textcolor{gray}{4}&\textcolor{gray}{1}&\textcolor{gray}{3}&\textcolor{gray}{4}&\textcolor{gray}{1}&\textcolor{gray}{1}&\textcolor{gray}{2}&&&&&&&&&$4=4$\\
$l_3$&\textcolor{gray}{4}&\textcolor{gray}{1}&\textcolor{blue}{2}&\textcolor{gray}{4}&\textcolor{gray}{1}&\textcolor{gray}{1}&\textcolor{gray}{2}&\textcolor{gray}{?}\\
$l_1$&\textcolor{gray}{4}&\textcolor{gray}{1}&\textcolor{gray}{2}&\textcolor{blue}{3}&\textcolor{gray}{1}&\textcolor{blue}{2}\\
$l_2$&\textcolor{gray}{4}&\textcolor{gray}{1}&\textcolor{gray}{2}&\textcolor{blue}{4}&\textcolor{gray}{1}&\textcolor{gray}{2}&\textcolor{blue}{3}\\
$l_2$&\textcolor{gray}{4}&\textcolor{gray}{1}&\textcolor{gray}{2}&\textcolor{gray}{4}&\textcolor{gray}{1}&\textcolor{gray}{2}&\textcolor{gray}{3}&&&&&&&&&$4=4$\\
$l_3$&\textcolor{gray}{4}&\textcolor{gray}{1}&\textcolor{blue}{0}&\textcolor{gray}{4}&\textcolor{gray}{1}&\textcolor{gray}{2}&\textcolor{gray}{3}&\textcolor{gray}{?}\\\hline
$l_1$&\textcolor{blue}{5}&\textcolor{gray}{1}&\textcolor{gray}{0}&\textcolor{blue}{2}&\textcolor{gray}{1}&\textcolor{gray}{$\cdot\cdot$}\\
$l_2$&\textcolor{gray}{5}&\textcolor{gray}{1}&\textcolor{gray}{0}&\textcolor{blue}{3}&\textcolor{gray}{1}&\textcolor{blue}{1}\\
$l_2$&\textcolor{gray}{5}&\textcolor{gray}{1}&\textcolor{gray}{0}&\textcolor{blue}{4}&\textcolor{gray}{1}&\textcolor{gray}{1}&\textcolor{blue}{1}\\
$l_2$&\textcolor{gray}{5}&\textcolor{gray}{1}&\textcolor{gray}{0}&\textcolor{blue}{5}&\textcolor{gray}{1}&\textcolor{gray}{1}&\textcolor{gray}{1}&\textcolor{blue}{1}\\
$l_2$&\textcolor{gray}{5}&\textcolor{gray}{1}&\textcolor{gray}{0}&\textcolor{gray}{5}&\textcolor{gray}{1}&\textcolor{gray}{1}&\textcolor{gray}{1}&\textcolor{gray}{1}&&&&&&&&$5=5$\\
$l_3$&\textcolor{gray}{5}&\textcolor{gray}{1}&\textcolor{blue}{4}&\textcolor{gray}{5}&\textcolor{gray}{1}&\textcolor{gray}{1}&\textcolor{gray}{1}&\textcolor{gray}{1}&\textcolor{gray}{?}\\
$l_1$&\textcolor{gray}{5}&\textcolor{gray}{1}&\textcolor{gray}{4}&\textcolor{blue}{5}&\textcolor{gray}{1}&\textcolor{gray}{1}&\textcolor{gray}{1}&\textcolor{blue}{2}\\
$l_2$&\textcolor{gray}{5}&\textcolor{gray}{1}&\textcolor{gray}{4}&\textcolor{gray}{5}&\textcolor{gray}{1}&\textcolor{gray}{1}&\textcolor{gray}{1}&\textcolor{gray}{2}&&&&&&&&$5=5$\\
$l_3$&\textcolor{gray}{5}&\textcolor{gray}{1}&\textcolor{blue}{3}&\textcolor{gray}{5}&\textcolor{gray}{1}&\textcolor{gray}{1}&\textcolor{gray}{1}&\textcolor{gray}{2}&\textcolor{gray}{?}\\
$\cdots$&&&&&&&\\
$l_2$&\textcolor{gray}{8}&\textcolor{gray}{1}&\textcolor{gray}{2}&\textcolor{blue}{7}&\textcolor{gray}{1}&\textcolor{gray}{2}&\textcolor{gray}{3}&\textcolor{gray}{4}&\textcolor{gray}{5}&\textcolor{blue}{6}&\\
$l_2$&\textcolor{gray}{8}&\textcolor{gray}{1}&\textcolor{gray}{2}&\textcolor{blue}{8}&\textcolor{gray}{1}&\textcolor{gray}{2}&\textcolor{gray}{3}&\textcolor{gray}{4}&\textcolor{gray}{5}&\textcolor{gray}{6}&\textcolor{blue}{5}&\\
$l_2$&\textcolor{gray}{8}&\textcolor{gray}{1}&\textcolor{gray}{2}&\textcolor{gray}{8}&\textcolor{gray}{1}&\textcolor{gray}{2}&\textcolor{gray}{3}&\textcolor{gray}{4}&\textcolor{gray}{5}&\textcolor{gray}{6}&\textcolor{gray}{5}&&&&&$8=8$\\
$l_3$&\textcolor{gray}{8}&\textcolor{gray}{1}&\textcolor{blue}{7}&\textcolor{gray}{8}&\textcolor{gray}{1}&\textcolor{gray}{2}&\textcolor{gray}{3}&\textcolor{gray}{4}&\textcolor{gray}{5}&\textcolor{gray}{6}&\textcolor{gray}{5}&\textcolor{gray}{?}\\
$l_1$&\textcolor{gray}{8}&\textcolor{gray}{1}&\textcolor{gray}{7}&\textcolor{blue}{8}&\textcolor{gray}{1}&\textcolor{gray}{2}&\textcolor{gray}{3}&\textcolor{gray}{4}&\textcolor{gray}{5}&\textcolor{gray}{6}&\textcolor{blue}{7}\\
\end{tabular}

\noindent \begin{tabular}{c|cccc|cccccccccccccc}\hline
$\ell\!$&$s\!$&$i\!$&$s_0\!$&$k\!$&$\!\!Z_1\!$&$\!\!Z_2\!$&$\!\!Z_3\!$&$\!\!Z_4\!$&$\!\!Z_5\!$&$\!\!Z_6\!$&$\!\!Z_7\!$&$\!\!Z_8\!$&$\!\!Z_9\!$&$\!\!Z_{10}\!$&$\!\!Z_{11}\!$&$k=s$?\\\hline

$l_2$&\textcolor{gray}{8}&\textcolor{gray}{1}&\textcolor{gray}{7}&\textcolor{gray}{8}&\textcolor{gray}{1}&\textcolor{gray}{2}&\textcolor{gray}{3}&\textcolor{gray}{4}&\textcolor{gray}{5}&\textcolor{gray}{6}&\textcolor{gray}{7}&&&&&$8=8$\\

$l_3$&\textcolor{gray}{8}&\textcolor{blue}{2}&\textcolor{blue}{5}&\textcolor{gray}{8}&\textcolor{gray}{1}&\textcolor{gray}{2}&\textcolor{gray}{3}&\textcolor{gray}{4}&\textcolor{gray}{5}&\textcolor{gray}{6}&\textcolor{gray}{7}&\textcolor{gray}{?}\\

$l_1$&\textcolor{gray}{8}&\textcolor{gray}{2}&\textcolor{gray}{5}&\textcolor{blue}{6}&\textcolor{gray}{1}&\textcolor{gray}{2}&\textcolor{gray}{3}&\textcolor{gray}{4}&\textcolor{blue}{7}\\

$l_2$&\textcolor{gray}{8}&\textcolor{gray}{2}&\textcolor{gray}{5}&\textcolor{gray}{6}&\textcolor{gray}{1}&\textcolor{gray}{2}&\textcolor{gray}{3}&\textcolor{gray}{4}&\textcolor{gray}{7}&$\!\!$\textcolor{gray}{(?)}$\!\!$&\\

$l_3$&\textcolor{gray}{8}&\textcolor{gray}{2}&\textcolor{blue}{0}&\textcolor{gray}{6}&\textcolor{gray}{1}&\textcolor{gray}{2}&\textcolor{gray}{3}&\textcolor{gray}{4}&\textcolor{gray}{7}&$\!\!$\textcolor{gray}{(?)}$\!\!$&&&&&&\textcolor{gray}{$6\not=8$}\\

$\cdots$&&&&&&&\\\hline

$l_1$&\textcolor{blue}{9}&\textcolor{gray}{2}&\textcolor{gray}{0}&\textcolor{blue}{2}&\textcolor{gray}{1}&$\cdot\cdot$\\

$\cdots$&&&&&&&\\

$l_3$&\textcolor{gray}{10}&\textcolor{gray}{3}&\textcolor{gray}{9}&\textcolor{gray}{10}&\textcolor{black}{1}&\textcolor{black}{1}&\textcolor{black}{1}&\textcolor{black}{2}&\textcolor{black}{3}&\textcolor{black}{4}&\textcolor{black}{5}&\textcolor{black}{6}&\textcolor{black}{7}&\textcolor{blue}{8}\\
\end{tabular}

\vspace{0.4cm}

\noindent \quad $l_4:$ \quad {\sf Output of the third path $(1,1,1,2,3,4,5,6,7,8)$} 

\vspace{0.1cm}
\noindent
This path was determined in accordance with $\leq_{\cal P}^{\rm fin}$.}

\vspace{0.1cm}

\subsection*{Computing the third path $enum_{\cal P}^{\sf (\leq,T)}(1,\!1,\!1)$ over ${\cal A}_{{\cal L}_{\cal P}}$}\label{ExampleSecondAlgo}

\markboth{ATTACHMENT:  A  THIRD PATH}{ATTACHMENT:  A  THIRD PATH}

\vspace{0.05cm}
{\scriptsize
\noindent \begin{tabular}{c|cccc|ccccccccccccc}\hline
$\!\!\ell\!\!$&$\!\!s\!\!$&$\!\!i\!\!$&$\!\!s_0\!\!$&$\!\!k\!\!$&$\!\!Z_{1,1}\!\!$&$\!\!Z_{1,2}\!\!$&$\!\!Z_{1,3}\!\!$&$\!\!Z_{1,4}\!\!$&$\!\!Z_{1,5}\!\!$&$\!\!Z_{1,6}\!\!$&$\!\!Z_{1,7}\!\!$&$\!\!Z_{1,8}\!\!$&$\!\!Z_{1,9}\!\!$&$\!\!Z_{1,{10}}\!\!$&(tape 1)$\!\!$\\
&&&&&$\!\!Z_{2,1}\!\!$&$\!\!Z_{2,2}\!\!$&$\!\!Z_{2,3}\!\!$&$\!\!Z_{2,4}\!\!$&$\!\!Z_{2,5}\!\!$&$\!\!Z_{2,6}\!\!$&$\!\!Z_{2,7}\!\!$&$\!\!Z_{2,8}\!\!$&$\!\!Z_{2,9}\!\!$&$\!\!Z_{2,{10}}\!\!$&(tape 2)$\!\!$\\
\hline

&&&&&\textcolor{blue}{1}&\textcolor{blue}{1}&\textcolor{blue}{1}&\textcolor{gray}{1}&\textcolor{gray}{1}&\textcolor{gray}{1}&\textcolor{gray}{1}&\textcolor{gray}{1}&\textcolor{gray}{1}&\textcolor{gray}{1}&\textcolor{gray}{(tape 1)}$\!\!$\\
&&&&&\textcolor{gray}{1}&\textcolor{gray}{1}&\textcolor{gray}{1}&\textcolor{gray}{1}&\textcolor{gray}{1}&\textcolor{gray}{1}&\textcolor{gray}{1}&\textcolor{gray}{1}&\textcolor{gray}{1}&\textcolor{gray}{1}&\textcolor{gray}{(tape 2)}$\!\!$\\
&&&&\\

&&&&&\textcolor{blue}{0}&\textcolor{gray}{1}&\textcolor{gray}{1}&\textcolor{gray}{1}&\textcolor{gray}{1}&\textcolor{gray}{1}&\textcolor{gray}{1}&\textcolor{gray}{1}&\textcolor{gray}{1}&\textcolor{gray}{1}&\textcolor{gray}{$\cdots$}\\
&&&&&\textcolor{gray}{1}&\textcolor{gray}{1}&\textcolor{gray}{1}&\textcolor{gray}{1}&\textcolor{gray}{1}&\textcolor{gray}{1}&\textcolor{gray}{1}&\textcolor{gray}{1}&\textcolor{gray}{1}&\textcolor{gray}{1}\\
&&&&\\

&\textcolor{blue}{1}&\textcolor{blue}{1}&&&\textcolor{blue}{$1$}&\textcolor{gray}{1}&\textcolor{gray}{1}&\textcolor{gray}{1}&\textcolor{gray}{1}&\textcolor{gray}{1}&\textcolor{gray}{1}&\textcolor{gray}{1}&\textcolor{gray}{1}&\textcolor{gray}{1}\\
&&&&&\textcolor{gray}{1}&\textcolor{gray}{1}&\textcolor{gray}{1}&\textcolor{gray}{1}&\textcolor{gray}{1}&\textcolor{gray}{1}&\textcolor{gray}{1}&\textcolor{gray}{1}&\textcolor{gray}{1}&\textcolor{gray}{1}\\
&&&&\\

$l_1$&\textcolor{blue}{2}&\textcolor{gray}{1}&\textcolor{blue}{0}&\textcolor{blue}{2}&\textcolor{gray}{1}&\textcolor{gray}{1}&\textcolor{gray}{1}&\textcolor{gray}{1}&\textcolor{gray}{1}&\textcolor{gray}{1}&\textcolor{gray}{1}&\textcolor{gray}{1}&\textcolor{gray}{1}&\textcolor{gray}{1}\\
&&&&&\textcolor{gray}{1}&\textcolor{gray}{1}&\textcolor{gray}{1}&\textcolor{gray}{1}&\textcolor{gray}{1}&\textcolor{gray}{1}&\textcolor{gray}{1}&\textcolor{gray}{1}&\textcolor{gray}{1}&\textcolor{gray}{1}\\
&&&&\\

$l_2$&\textcolor{gray}{2}&\textcolor{gray}{1}&\textcolor{gray}{0}&\textcolor{gray}{2}&\textcolor{gray}{1}&\textcolor{gray}{$\cdot\cdot$}\\
&&&&&\textcolor{gray}{1}&\textcolor{gray}{1}&\textcolor{gray}{1}&\textcolor{gray}{1}&\textcolor{gray}{1}&\textcolor{gray}{1}&\textcolor{gray}{1}&\textcolor{gray}{1}&\textcolor{gray}{1}&\textcolor{gray}{1}\\
&&&&\\

$l_3$&\textcolor{gray}{2}&\textcolor{gray}{1}&\textcolor{gray}{0}&\textcolor{gray}{2}&\textcolor{gray}{1}&\textcolor{gray}{?}\\
&&&&&\textcolor{gray}{1}&\textcolor{gray}{1}&\textcolor{gray}{1}&\textcolor{gray}{1}&\textcolor{gray}{1}&\textcolor{gray}{1}&\textcolor{gray}{1}&\textcolor{gray}{1}&\textcolor{gray}{1}&\textcolor{gray}{1}\\
&&&&\\

$l_1$&\textcolor{blue}{3}&\textcolor{gray}{1}&\textcolor{blue}{0}&\textcolor{blue}{2}&\textcolor{gray}{1}&\textcolor{gray}{$\cdot\cdot$}\\
&&&&&\textcolor{gray}{1}&\textcolor{gray}{1}&\textcolor{gray}{1}&\textcolor{gray}{1}&\textcolor{gray}{1}&\textcolor{gray}{1}&\textcolor{gray}{1}&\textcolor{gray}{1}&\textcolor{gray}{1}&\textcolor{gray}{1}\\
&&&&\\

$l_2$&\textcolor{gray}{3}&\textcolor{gray}{1}&\textcolor{gray}{0}&\textcolor{blue}{3}&\textcolor{gray}{1}&\textcolor{blue}{2}&\textcolor{gray}{$\cdot\cdot$}\\
&&&&&\textcolor{gray}{1}&\textcolor{blue}{0}&\textcolor{gray}{1}&\textcolor{gray}{1}&\textcolor{gray}{1}&\textcolor{gray}{1}&\textcolor{gray}{1}&\textcolor{gray}{1}&\textcolor{gray}{1}&\textcolor{gray}{1}\\
&&&&\\

$l_2$&\textcolor{gray}{3}&\textcolor{gray}{1}&\textcolor{gray}{0}&\textcolor{gray}{3}&\textcolor{gray}{1}&\textcolor{gray}{2}&\textcolor{gray}{$\cdot\cdot$}\\
&&&&&\textcolor{gray}{1}&\textcolor{gray}{0}&\textcolor{gray}{1}&\textcolor{gray}{1}&\textcolor{gray}{1}&\textcolor{gray}{1}&\textcolor{gray}{1}&\textcolor{gray}{1}&\textcolor{gray}{1}&\textcolor{gray}{1}\\
&&&&\\

$l_3$&\textcolor{gray}{3}&\textcolor{gray}{1}&\textcolor{gray}{0}&\textcolor{gray}{3}&\textcolor{gray}{1}&\textcolor{gray}{2}&\textcolor{gray}{$?$}\\
&&&&&\textcolor{gray}{1}&\textcolor{gray}{0}&\textcolor{gray}{1}&\textcolor{gray}{1}&\textcolor{gray}{1}&\textcolor{gray}{1}&\textcolor{gray}{1}&\textcolor{gray}{1}&\textcolor{gray}{1}&\textcolor{gray}{1}\\
&&&&\\

$l_1$&\textcolor{gray}{3}&\textcolor{gray}{1}&\textcolor{blue}{2}&\textcolor{blue}{2}&\textcolor{gray}{1}&\textcolor{gray}{2}&\textcolor{gray}{$\cdot\cdot$}\\
&&&&&\textcolor{gray}{1}&\textcolor{gray}{0}&\textcolor{gray}{1}&\textcolor{gray}{1}&\textcolor{gray}{1}&\textcolor{gray}{1}&\textcolor{gray}{1}&\textcolor{gray}{1}&\textcolor{gray}{1}&\textcolor{gray}{1}\\
&&&&\\

$l_2$&\textcolor{gray}{3}&\textcolor{gray}{1}&\textcolor{gray}{2}&\textcolor{blue}{3}&\textcolor{gray}{1}&\textcolor{blue}{1}&\textcolor{gray}{$\cdot\cdot$}\\
&&&&&\textcolor{gray}{1}&\textcolor{blue}{1}&\textcolor{gray}{1}&\textcolor{gray}{1}&\textcolor{gray}{1}&\textcolor{gray}{1}&\textcolor{gray}{1}&\textcolor{gray}{1}&\textcolor{gray}{1}&\textcolor{gray}{1}\\
\end{tabular}

\noindent \begin{tabular}{c|cccc|cccccccccccccc}\hline
$\!\!\ell\!\!$&$\!\!s\!\!$&$\!\!i\!\!$&$\!\!s_0\!\!$&$\!\!k\!\!$&$\!\!Z_{1,1}\!\!$&$\!\!Z_{1,2}\!\!$&$\!\!Z_{1,3}\!\!$&$\!\!Z_{1,4}\!\!$&$\!\!Z_{1,5}\!\!$&$\!\!Z_{1,6}\!\!$&$\!\!Z_{1,7}\!\!$&$\!\!Z_{1,8}\!\!$&$\!\!Z_{1,9}\!\!$&$\!\!Z_{1,{10}}\!\!$&(tape 1)$\!\!$\\
&&&&&$\!\!Z_{2,1}\!\!$&$\!\!Z_{2,2}\!\!$&$\!\!Z_{2,3}\!\!$&$\!\!Z_{2,4}\!\!$&$\!\!Z_{2,5}\!\!$&$\!\!Z_{2,6}\!\!$&$\!\!Z_{2,7}\!\!$&$\!\!Z_{2,8}\!\!$&$\!\!Z_{2,9}\!\!$&$\!\!Z_{2,{10}}\!\!$&(tape 2)$\!\!$\\
\hline

$l_2$&\textcolor{gray}{3}&\textcolor{gray}{1}&\textcolor{gray}{2}&\textcolor{gray}{3}&\textcolor{gray}{1}&\textcolor{gray}{1}&\textcolor{gray}{$\cdot\cdot$}\\
&&&&&\textcolor{gray}{1}&\textcolor{gray}{1}&\textcolor{gray}{1}&\textcolor{gray}{1}&\textcolor{gray}{1}&\textcolor{gray}{1}&\textcolor{gray}{1}&\textcolor{gray}{1}&\textcolor{gray}{1}&\textcolor{gray}{1}\\
&&&&&\\

$l_3$&\textcolor{gray}{3}&\textcolor{gray}{1}&\textcolor{gray}{2}&\textcolor{gray}{3}&\textcolor{gray}{1}&\textcolor{gray}{1}&\textcolor{gray}{$?$}\\
&&&&&\textcolor{gray}{1}&\textcolor{gray}{1}&\textcolor{gray}{1}&\textcolor{gray}{1}&\textcolor{gray}{1}&\textcolor{gray}{1}&\textcolor{gray}{1}&\textcolor{gray}{1}&\textcolor{gray}{1}&\textcolor{gray}{1}\\
&&&&\\

$l_1$&\textcolor{blue}{4}&\textcolor{gray}{1}&\textcolor{blue}{0}&\textcolor{blue}{2}&\textcolor{gray}{1}&\textcolor{gray}{$\cdot\cdot$}\\
&&&&&\textcolor{gray}{1}&\textcolor{gray}{1}&\textcolor{gray}{1}&\textcolor{gray}{1}&\textcolor{gray}{1}&\textcolor{gray}{1}&\textcolor{gray}{1}&\textcolor{gray}{1}&\textcolor{gray}{1}&\textcolor{gray}{1}\\
&&&&&\\
$\!\!\!\cdots\!\!\!$&&&&&\\

$l_1$&\textcolor{blue}{8}&\textcolor{gray}{1}&\textcolor{blue}{0}&\textcolor{blue}{2}&\textcolor{gray}{1}&\textcolor{gray}{$\cdot\cdot$}\\
&&&&&\textcolor{gray}{1}&\textcolor{gray}{1}&\textcolor{gray}{1}&\textcolor{gray}{1}&\textcolor{gray}{1}&\textcolor{gray}{1}&\textcolor{gray}{1}&\textcolor{gray}{1}&\textcolor{gray}{1}&\textcolor{gray}{1}\\

&&&&&\\
$\!\!\!\cdots\!\!\!$&&&&&\\

$l_2$&\textcolor{gray}{8}&\textcolor{gray}{1}&\textcolor{gray}{2}&\textcolor{blue}{8}&\textcolor{gray}{1}&\textcolor{gray}{2}&\textcolor{gray}{3}&\textcolor{gray}{4}&\textcolor{gray}{5}&\textcolor{gray}{6}&\textcolor{blue}{7}&\textcolor{gray}{$\cdot\cdot$}\\
&&&&&\textcolor{gray}{1}&\textcolor{gray}{0}&\textcolor{gray}{1}&\textcolor{gray}{1}&\textcolor{gray}{1}&\textcolor{gray}{1}&\textcolor{blue}{0}&\textcolor{gray}{1}&\textcolor{gray}{1}&\textcolor{gray}{1}\\
&&&&\\

$l_3$&\textcolor{gray}{8}&\textcolor{blue}{2}&\textcolor{gray}{2}&\textcolor{gray}{8}&\textcolor{gray}{1}&\textcolor{gray}{2}&\textcolor{gray}{3}&\textcolor{gray}{4}&\textcolor{gray}{5}&\textcolor{gray}{6}&\textcolor{gray}{7}&\textcolor{gray}{$?$}\\
&&&&&\textcolor{gray}{1}&\textcolor{gray}{0}&\textcolor{gray}{1}&\textcolor{gray}{1}&\textcolor{gray}{1}&\textcolor{gray}{1}&\textcolor{gray}{0}&\textcolor{gray}{1}&\textcolor{gray}{1}&\textcolor{gray}{1}\\
&&&&\\

$l_1$&\textcolor{gray}{8}&\textcolor{gray}{2}&\textcolor{blue}{7}&\textcolor{blue}{7}&\textcolor{gray}{1}&\textcolor{gray}{2}&\textcolor{gray}{3}&\textcolor{gray}{4}&\textcolor{gray}{5}&\textcolor{gray}{6}&\textcolor{gray}{7}&\textcolor{gray}{$\cdot\cdot$}\\
&&&&&\textcolor{gray}{1}&\textcolor{gray}{0}&\textcolor{gray}{1}&\textcolor{gray}{1}&\textcolor{gray}{1}&\textcolor{gray}{1}&\textcolor{gray}{0}&\textcolor{gray}{1}&\textcolor{gray}{1}&\textcolor{gray}{1}\\
&&&&\\

$l_2$&\textcolor{gray}{8}&\textcolor{gray}{2}&\textcolor{gray}{7}&\textcolor{blue}{8}&\textcolor{gray}{1}&\textcolor{gray}{2}&\textcolor{gray}{3}&\textcolor{gray}{4}&\textcolor{gray}{5}&\textcolor{gray}{6}&\textcolor{blue}{5}&\textcolor{gray}{$\cdot\cdot$}\\
&&&&&\textcolor{gray}{1}&\textcolor{gray}{0}&\textcolor{gray}{1}&\textcolor{gray}{1}&\textcolor{gray}{1}&\textcolor{gray}{1}&\textcolor{blue}{1}&\textcolor{gray}{1}&\textcolor{gray}{1}&\textcolor{gray}{1}\\
&&&&\\

$l_2$&\textcolor{gray}{8}&\textcolor{gray}{2}&\textcolor{gray}{7}&\textcolor{gray}{8}&\textcolor{gray}{1}&\textcolor{gray}{2}&\textcolor{gray}{3}&\textcolor{gray}{4}&\textcolor{gray}{5}&\textcolor{gray}{6}&\textcolor{gray}{5}&\textcolor{gray}{$\cdot\cdot$}\\
&&&&&\textcolor{gray}{1}&\textcolor{gray}{0}&\textcolor{gray}{1}&\textcolor{gray}{1}&\textcolor{gray}{1}&\textcolor{gray}{1}&\textcolor{gray}{1}&\textcolor{gray}{1}&\textcolor{gray}{1}&\textcolor{gray}{1}\\
&&&&\\

$l_3$&\textcolor{gray}{8}&\textcolor{gray}{2}&\textcolor{gray}{7}&\textcolor{gray}{8}&\textcolor{gray}{1}&\textcolor{gray}{2}&\textcolor{gray}{3}&\textcolor{gray}{4}&\textcolor{gray}{5}&\textcolor{gray}{6}&\textcolor{gray}{5}&\textcolor{gray}{?}\\
&&&&&\textcolor{gray}{1}&\textcolor{gray}{0}&\textcolor{gray}{1}&\textcolor{gray}{1}&\textcolor{gray}{1}&\textcolor{gray}{1}&\textcolor{gray}{1}&\textcolor{gray}{1}&\textcolor{gray}{1}&\textcolor{gray}{1}\\
&&&&\\

$l_1$&\textcolor{gray}{8}&\textcolor{gray}{2}&\textcolor{blue}{2}&\textcolor{blue}{2}&\textcolor{gray}{1}&\textcolor{gray}{2}&\textcolor{gray}{$\cdot\cdot$}\\
&&&&&\textcolor{gray}{1}&\textcolor{gray}{0}&\textcolor{gray}{1}&\textcolor{gray}{1}&\textcolor{gray}{1}&\textcolor{gray}{1}&\textcolor{gray}{1}&\textcolor{gray}{1}&\textcolor{gray}{1}&\textcolor{gray}{1}\\
&&&&\\

$l_2$&\textcolor{gray}{8}&\textcolor{gray}{2}&\textcolor{gray}{2}&\textcolor{blue}{3}&\textcolor{gray}{1}&\textcolor{blue}{1}&\textcolor{gray}{$\cdot\cdot$}\\
&&&&&\textcolor{gray}{1}&\textcolor{blue}{1}&\textcolor{gray}{1}&\textcolor{gray}{1}&\textcolor{gray}{1}&\textcolor{gray}{1}&\textcolor{gray}{1}&\textcolor{gray}{1}&\textcolor{gray}{1}&\textcolor{gray}{1}\\
&&&&\\

$\!\!\!\cdots\!\!\!$&&&&&\\

$l_2$&\textcolor{gray}{10}$\!\!\!$&\textcolor{gray}{3}&\textcolor{gray}{9}&\textcolor{blue}{10}&\textcolor{gray}{1}&\textcolor{gray}{2}&\textcolor{gray}{3}&\textcolor{gray}{4}&\textcolor{gray}{5}&\textcolor{gray}{6}&\textcolor{gray}{5}&\textcolor{gray}{6}&\textcolor{blue}{7}&\textcolor{gray}{$\cdot\cdot$}\\
&&&&&\textcolor{gray}{1}&\textcolor{gray}{0}&\textcolor{gray}{1}&\textcolor{gray}{1}&\textcolor{gray}{1}&\textcolor{gray}{1}&\textcolor{gray}{1}&\textcolor{gray}{1}&\textcolor{blue}{0}&\textcolor{gray}{1}\\
&&&&\\

$l_2$&\textcolor{gray}{10}$\!\!\!$&\textcolor{gray}{3}&\textcolor{gray}{9}&\textcolor{gray}{10}&\textcolor{gray}{1}&\textcolor{gray}{2}&\textcolor{gray}{3}&\textcolor{gray}{4}&\textcolor{gray}{5}&\textcolor{gray}{6}&\textcolor{gray}{5}&\textcolor{gray}{6}&\textcolor{gray}{7}&\textcolor{gray}{$\cdot\cdot$}\\
&&&&&\textcolor{gray}{1}&\textcolor{gray}{0}&\textcolor{gray}{1}&\textcolor{gray}{1}&\textcolor{gray}{1}&\textcolor{gray}{1}&\textcolor{gray}{1}&\textcolor{gray}{1}&\textcolor{gray}{0}&\textcolor{gray}{1}\\
&&&&\\

$l_3$&\textcolor{gray}{10}$\!\!\!$&\textcolor{gray}{3}&\textcolor{gray}{9}&\textcolor{gray}{10}&\textcolor{black}{1}&\textcolor{black}{2}&\textcolor{black}{3}&\textcolor{black}{4}&\textcolor{black}{5}&\textcolor{black}{6}&\textcolor{black}{5}&\textcolor{black}{6}&\textcolor{black}{7}&\textcolor{blue}{8}\\
&&&&&\textcolor{gray}{1}&\textcolor{gray}{0}&\textcolor{gray}{1}&\textcolor{gray}{1}&\textcolor{gray}{1}&\textcolor{gray}{1}&\textcolor{gray}{1}&\textcolor{gray}{1}&\textcolor{gray}{0}&\textcolor{gray}{1}\\

\end{tabular}

\vspace{0.4cm}
\noindent \quad $l_4:$ \quad {\sf Output of the third path $(1,2,3,4,5,6,5,6,7,8)$}

\vspace{0.2cm}
\noindent
This path was enumerated in accordance with the order

$\bigcup_{s\geq 1}(\leq_{\cal P} \cap ({\cal L}_{\cal P}^s)^2 )  \cup  \bigcup_{s\geq 1}\bigcup_{r> s}(({\cal L}_{\cal P}^s \cap {\sf Lab}_{\cal P}^{\rm fin})\times ({\cal L}_{\cal P}^r\cap {\sf Lab}_{\cal P}^{\rm fin}))$}.

\newpage
\subsection*{A decision tree for representing further paths}

\markboth{ATTACHMENT: FURTHER PATHS}{ATTACHMENT: FURTHER PATHS}

\begin{figure}[th]\centering
\includegraphics[height=115mm]{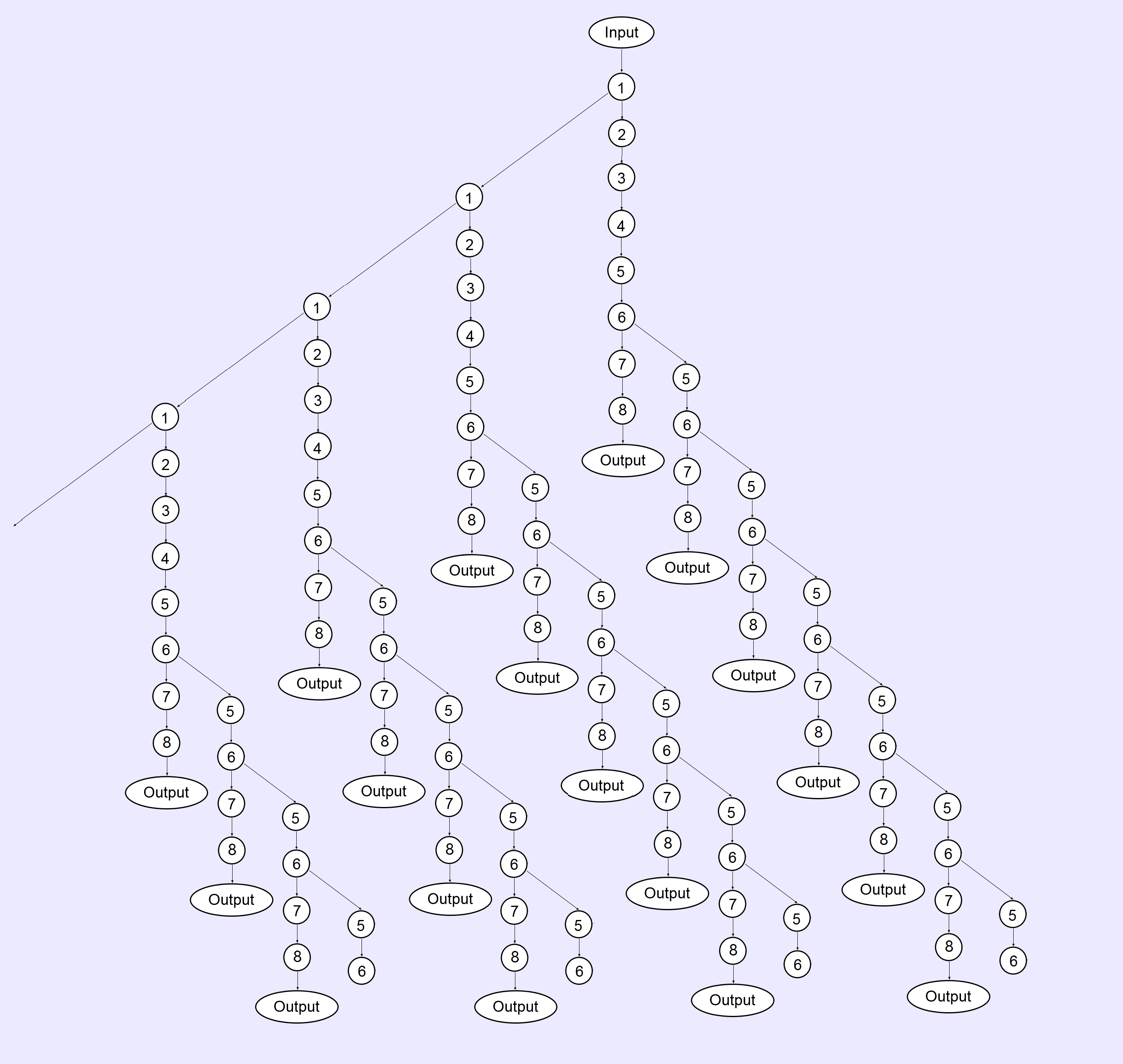}
\caption{A tree extended by an {\sf Input} node and {\sf Output} nodes}\label{BerWurzelnProgrammpfade}
\end{figure}

In Figure \ref{BerWurzelnProgrammpfade}, 
 the tree is determined by ${\cal P}$ and $\mathfrak{Fl\,}_{\cal P}^{\sf label}$ considered in Figure \ref{BerWurzeln1}. The walks originating from the root (number 1) represent  ${\cal P}$-paths in ${\sf Lab}_{\cal P}^{\rm fin}$ and initial segments of ${\cal P}$-paths  discussed also in Examples \ref{ExBerWurzeln1} and \ref{ExBerWurzeln2}. The leaves are not arranged corresponding to one of the orders $\leq_{\cal P}$ and $\leq_{\cal P}^{\rm fin}$. The sequences of nodes labeled by $1,\ldots, 1,2,3,\ldots$ on the left side of the picture belong to walks representing ${\cal P}$-paths in ${\sf Lab}_{\cal P}^{\rm Co}$.

\vfill

\hfill {\scriptsize\sf C $\cdot $ H $\cdot$ T}
\end{document}